\documentclass[12pt]{amsart}
\usepackage{epsfig,color}

\makeatother

\theoremstyle{definition}

\def\fnum{equation}
\newtheorem{Thm}[\fnum]{Theorem}
\newtheorem{Cor}[\fnum]{Corollary}

\newtheorem{Lem}[\fnum]{Lemma}
\newtheorem{Con}[\fnum]{Conjecture}

\newtheorem{Rem}[\fnum]{Remark}
\newtheorem{Pro}[\fnum]{Proposition}

\numberwithin{equation}{section}
\newcommand{\Vol}{{\text{Vol}}}

\newcommand{\nn}{{\bf{n}}}

\def\RR{{\bold R}}

\def\SS{{\bold S}}

\newcommand{\dv}{{\text {div}}}
\newcommand{\e}{{\text {e}}}

\newcommand{\Area}{{\text {Area}}}

\newcommand{\cT}{{\mathcal{T}}}

\newcommand{\cB}{{\mathcal{B}}}

\newcommand{\cL}{{\mathcal{L}}}

\newcommand{\eqr}[1]{(\ref{#1})}

\title[Generic singularities]{Generic mean curvature flow I; generic singularities}

\author{Tobias H. Colding}%
\address{MIT, Dept. of Math.\\
77 Massachusetts Avenue, Cambridge, MA 02139-4307.}
\author{William P. Minicozzi II}%
\address{Johns Hopkins University\\
Dept. of Math.\\
3400 N. Charles St.\\
Baltimore, MD 21218}

\thanks{The   authors
were partially supported by NSF Grants DMS  0606629, DMS
0405695,  and NSF FRG grants DMS 
 0854774 and DMS 0853501}

\email{colding@math.mit.edu  and minicozz@math.jhu.edu}

\begin{document}

\maketitle

\begin{abstract}
 It has long been conjectured that starting at a generic smooth closed  embedded surface in $\RR^3$, the mean curvature flow remains smooth until it arrives at a singularity in a neighborhood of which the flow looks like concentric spheres or cylinders.  That is, the only   singularities of a generic flow are spherical or cylindrical.  We will address this conjecture here and in a sequel.  The higher dimensional case will be addressed elsewhere.

The key  in showing this conjecture is to show that
  shrinking spheres, cylinders and planes are the only stable self-shrinkers under the mean curvature flow.  We prove this here in all dimensions.  An easy consequence of this is that every other singularity than spheres and cylinders can be perturbed away.
\end{abstract}

\section{Introduction}

One of the most important problems in mean curvature flow is to understand the possible singularities that the flow goes through.  Singularities are unavoidable as the flow contracts any closed embedded hypersurface in Euclidean space eventually leading to extinction of the evolving hypersurfaces.  Classically, mean curvature flow  was only defined up to the first singular time, but a number of ways to define weak solutions have been developed over the last $30$ years by Brakke, \cite{B}, Evans-Spruck, \cite{EvSp}, and Chen-Giga-Goto, \cite{CGG}; cf. Osher-Sethian, \cite{OSe}.

Through the combined work of Abresch-Langer, Calabi, Epstein, Gage, Grayson, and Hamilton, among others, singularities for the curve shortening flow in the plane, that is, mean curvature flow for curves, are well understood.  In higher dimensions, a key starting point for singularity analysis is Huisken's montonicity formula, \cite{H3}.  This is because the monotonicity implies, according to Huisken, \cite{H3}, that the flow is asymptotically self-similar near a given singularity  and, thus,  is modelled by self-shrinking solutions of the flow.    Huisken's original proof applied to so-called type I singularities, but Ilmanen and White  extended Huisken's formula to weak solutions and proved asymptotic self-similarity for all singularities of the flow with convergence now in the weak sense of geometric measure theory; see lemma $8$ in \cite{I1}.

Here and in a paper to follow, we will show a long-standing   conjecture (of Huisken) classifying the singularities of   mean curvature flow starting at a generic closed embedded surface.  In a third paper, we will address the higher dimensional case.
The current paper shows that, in all dimensions,  the only singularities that cannot be perturbed away are cylinders and spheres.  It also gives an application to generic mean curvature flow in $\RR^3$ showing that   generic mean curvature flow that disappears in a compact point does so in a round point.
 Well known examples  of Angenent and numerics of Angenent, Chopp, and Ilmanen discussed below show that there is virtually no hope of classifying the singularities  for mean curvature flow starting at an arbitrary hypersurface in $\RR^{n+1}$ for $n > 1${\footnote{Since time slices of self-shrinkers are minimal hypersurfaces for a conformally changed metric on $\RR^{n+1}$ (see Section \ref{s:one}), some classification of self-shrinkers may be possible in $\RR^3$ using the ideas of \cite{CM5}--\cite{CM10}.}}; the thrust of Huisken's conjecture is that this can be done for generic initial hypersurfaces.

Under various assumptions, like convexity, mean convexity, and two-convexity, or for curves in the plane, the blow ups of singularities have been classified by various authors including Huisken, Gage-Hamilton, Grayson, Huisken-Sinestrari, and White.
 Unlike in our case,   in all previous classifications of possible singularities for $n > 1$, assumptions, like mean convexity, were made that immediately guaranteed that none of the exotic singularities described above could occur.  For instance, since mean convexity is preserved under the flow, if the initial surface is mean convex, then so are  all blow ups   and, since none of the exotic self-similar flows mentioned above are mean convex, they are immediately ruled out as singularities.

\bigskip
A one-parameter family $M_t$ of hypersurfaces in $\RR^{n+1}$ flows by mean curvature if
\begin{equation}
	\left( \partial_t x \right)^{\perp} =  \bar{H} \, ,
\end{equation}
where
  $\bar{H} = - H \, \nn$ is the mean curvature vector, $\nn$ is the outward unit normal, $v^{\perp}$ is the  normal part of a vector $v$,
  and the mean curvature $H$ is given by
\begin{equation}	\label{e:defH}
	H = \dv \,  \nn \, .
\end{equation}
With this convention, $H$ is $n/R$ on the $n$-sphere of radius $R$ in $\RR^{n+1}$ and $H$ is $k/R$ on the ``cylinder'' $\SS^k \times \RR^{n-k} \subset \RR^{n+1}$ of radius $R$.  Thus, in either case, the mean curvature vector ``points inwards'' and, hence, the flow contracts.

The simplest (non-static) mean curvature flow is given by the one-parameter family of shrinking spheres $M_t \subset \RR^{n+1}$ centered at the origin and with radius  $\sqrt{-2nt}$ for $t \leq 0$.   This  is a smooth flow except at the origin at time $0$ when the flow becomes extinct.  In \cite{H1}, Huisken showed that MCF starting at any smooth compact convex initial hypersurface in $\RR^{n+1}$ remains smooth and convex until it becomes extinct at a point and if we rescale the flow about the point in space-time where it becomes extinct, then the rescalings converge to round spheres.
  Huisken-Sinestrari, \cite{HS1}, \cite{HS2},   and  White, \cite{W2}, \cite{W3}, have proven a number of striking and important results about MCF of mean convex  hypersurfaces and their singularities and Huisken-Sinestrari, \cite{HS3}, have developed a theory for MCF with
  surgery for two-convex hypersurfaces in $\RR^{n+1}$ ($n\geq 3$) using their analysis of singularities (and their blow ups).

  Huisken's proof that   convex hypersurfaces become extinct in round points applied for $n \geq2$, but the corresponding result for convex curves was proven   by Gage and Hamilton in \cite{GH}.  In fact, in \cite{G1} (see also \cite{G3}, \cite{EpG},  \cite{Ha1}, \cite{H5}),  Grayson showed that any simple closed smooth curve in $\RR^2$ stays smooth under the curve shortening flow, eventually becomes convex, and thus will become extinct in a ``round point''.  The situation is more complicated for surfaces   where there are many other potential types of singularities that can arise.  For instance,  Grayson constructed a rotationally symmetric dumbbell in \cite{G2} where the neck pinches off before the two bells become extinct.   For rescalings of the singularity at the neck, the resulting blow ups cannot be extinctions and, thus, are certainly not spheres.  In fact, rescalings of the singularity converge to shrinking cylinders; we refer to White's survey
 \cite{W1} for further discussion of this example.

The family of shrinking spheres of radius $\sqrt{-2nt}$ is self-similar in the sense that
 $M_t$ is given by
\begin{equation}	\label{e:selfshr}
	M_t = \sqrt{-t} \, \, M_{-1} \, .
\end{equation}
A MCF  $M_t$ satisfying \eqr{e:selfshr} is called a self-shrinker.
Self-shrinkers play an important role in the study of mean curvature flow, not least because they describe all possible blow ups at a given singularity of a mean curvature flow.    To explain this,  we will need the notion of a tangent flow, cf. \cite{I1}, \cite{W4}, which generalizes the tangent cone construction from minimal surfaces.  The basic idea is that we can rescale a MCF in space and time to obtain a new MCF thereby expanding a small neighborhood of the point that we want to focus on.  Huisken's monotonicity gives uniform control over these rescalings and a standard compactness theorem then gives a subsequence converging to a limiting solution of the MCF.  This limit is called a tangent flow.  A tangent flow will
 achieve equality in Huisken's monotonicity formula and, thus, must   be a self-shrinker by \cite{H3},
 \cite{I1}.

 The precise definition of a {\it {tangent flow}} at a point  $(x_0 , t_0)$  in space-time of a MCF
  $M_t$ is as follows:
   First translate $M_t$ in space-time to move $(x_0 , t_0)$ to $(0,0)$ and then take a sequence of parabolic dilations
$(x,t) \to (c_j \, x ,c_j^2 \, t)$ with $c_j \to \infty$ to get MCF's $M^j_t = c_j \,  \left( M_{c_j^{-2} \, t+ t_0} - x_0 \right)$.  Using Huisken's monotonicity formula, \cite{H3}, and Brakke's compactness theorem, \cite{B}, White \cite{W4} and Ilmanen \cite{I1} show that a subsequence of the $M^j_t$'s converges weakly
 to a limiting flow $\cT_t$ that we will call
  a {\emph{tangent flow}} at   $(x_0,t_0)$. Moreover, another application of Huisken's monotonicity shows that
   $\cT_t$ is a self-shrinker.  It is not known whether $\cT_t$ is unique. That is, whether different sequences of dilations might lead to different tangent flows.

In \cite{I1}, Ilmanen proved that in $\RR^3$  tangent flows at the first singular time must be smooth, although he left open the possibility of multiplicity.
However, he conjectured that the multiplicity must be one:

\begin{Con}	\label{c:m1}
(Ilmanen, see page $7$ of \cite{I1}; cf. \cite{E3})
For a smooth one-parameter family of closed embedded surfaces in $\RR^3$ flowing by mean curvature,
  every tangent flow at the first singular time has multiplicity one.
\end{Con}

If this conjecture holds, then it would follow from Brakke's regularity theorem that  near a singularity the flow can be written as a graph of a function with small gradient over the tangent flow.

We will say that a MCF is smooth up to and including the first singular time if every tangent flow at the first singular time is smooth and has multiplicity one.  Conjecturally,  all MCF in $\RR^3$  are smooth up to the first singular time.

\bigskip
A self-similarly shrinking solution to MCF is completely determined by the $t=-1$ time-slice and, thus, we sometimes think of a self-similar flow as just that time-slice.

The simplest  self-shrinkers are $\RR^n$, the sphere of radius
$\sqrt{-2nt}$, and more generally  cylindrical products  $\SS^k \times \RR^{n-k}$ (where the $\SS^k$ has radius $\sqrt{-2kt}$).   All of these examples are mean convex (i.e., have $H \geq 0$) and, in fact, these are the only mean convex examples under mild assumptions; see \cite{H3}, \cite{H4}, \cite{AbLa}, and Theorem  \ref{t:huisken}.

Without the assumption on mean convexity, then there are expected to be many more examples of self-shrinkers in $\RR^3$.   In particular, Angenent, \cite{A}, has constructed a self-shrinking torus
(``shrinking donut'')  of revolution{\footnote{In fact, for every $n\geq 2$, by rotating a curve in plane, Angenent constructed an embedded self-shrinker in $\RR^{n+1}$ that is topologically  $\SS^1 \times \SS^{n-1}$.  The curve   satisfies  an ODE that  can be interpreted as the geodesic equation for a singular metric.}} and there is numerical evidence for a number of other examples  (see Chopp, \cite{Ch}, Angenent-Chopp-Ilmanen, \cite{AChI}, Ilmanen, \cite{I3}, and Nguyen, \cite{N1}, \cite{N2}).  We will see in this paper that all self-shrinkers except the simplest are highly unstable and, thus, hard to find.
We will use this instability to perturb them away in a generic flow.

Angenent used his self-shrinking torus to give a new proof that the dumbbell has a neck pinching singularity before the two bells become extinct.  The idea is to make the neck of the dumbbell long and thin and the bells on either side large enough to contain two large round spheres.  By the maximum principle, the  interior of the MCF of the dumbbell will contain the shrinking spheres and, thus, cannot become extinct until after the spheres do.  On the other hand, Angenent used the self-shrinking torus to encircle the neck of the dumbbell and, thus, conclude that the neck  would pinch off before the spheres had shrunk to points; see also White, \cite{W1}, for a beautiful expository discussion of the dumbbell, tangent flows, and self-similar solutions.

While $M_t$ will always be a one-parameter family of hypersurfaces flowing by mean curvature,
we will use $\Sigma$ to denote a single hypersurface and $\Sigma_s$ to denote a one-parameter variation of $\Sigma$.  Frequently, $\Sigma$ will be the time $t=-1$ slice of a self-shrinking solution $M_t$ of the mean curvature flow.

Given $x_0 \in \RR^{n+1}$ and $t_0 > 0$,  define the functional $F_{x_0 , t_0}$ (see \cite{H3}, $(6)$ in \cite{AChI}, or page $6$ in \cite{I3}; cf.  $2.4$ in \cite{I2}) by
 \begin{equation}	\label{e:Ft0}
 	F_{x_0 , t_0} (\Sigma) = (4\pi t_0)^{-\frac{n}{2}} \, \int_{\Sigma} \, \e^{\frac{-|x- x_0|^2}{4t_0}} \, d\mu \, .
\end{equation}
The main point of these functionals is that $\Sigma$ is a critical point of $F_{x_0,t_0}$ precisely when it is the time $t=-t_0$ slice of a self-shrinking solution of the mean curvature flow that becomes extinct at $x=x_0$ and $t=0$.
The {\it {entropy}} $\lambda = \lambda (\Sigma)$ of $\Sigma$ will
be the supremum  of the $F_{x_0,t_0}$ functionals
\begin{equation}	\label{e:lamb}
	\lambda =   \sup_{x_0,t_0} \, \, F_{x_0 , t_0} (\Sigma)   \, .
\end{equation}
The key properties of the entropy $\lambda$ are:
\begin{itemize}
\item   $\lambda$ is non-negative and invariant under dilations, rotations, or translations of $\Sigma$.
\item $\lambda (M_t)$ is non-increasing in $t$ if the hypersurfaces $M_t$ flow by mean curvature.
\item The critical points of $\lambda$ are self-shrinkers for the mean curvature flow.
\end{itemize}
These properties are the main {\underline{advantages}} of the entropy functional over the $F$ functionals.  The main {\underline{disadvantage}} of the entropy is that it need not depend smoothly on $\Sigma$.  To deal with this,
we will say that a self-shrinker is {\emph{entropy-stable}} if it is a local minimum   for the entropy functional.{\footnote{Here ``local'' means with respect to hypersurfaces that can be written as a graph  over the given hypersurface of a function with small $C^2$ norm. In particular, we do not require the support to be compact.}}

To illustrate our results, we will first specialize to the case where $n=2$, that is to mean curvature flow of surfaces in $\RR^3$.

\begin{Thm}	\label{c:nonlin1a}
Suppose that $\Sigma \subset \RR^3$ is a smooth complete embedded self-shrinker without boundary and with polynomial volume growth.
\begin{itemize}
\item If $\Sigma$ is not a sphere, a plane, or a cylinder,  then there is a graph $\tilde{\Sigma}$ over $\Sigma$ of a compactly supported function with arbitrarily small $C^m$ norm (for any fixed $m$) so that
 $\lambda( \tilde{\Sigma})<\lambda(\Sigma)$.
 \end{itemize}
In  particular,   $\Sigma$  cannot arise as a tangent flow to the MCF starting from $\tilde{\Sigma}$.
\end{Thm}

Motivated by this theorem, we will next define an ad hoc notion
of generic MCF that requires the least amount of technical set-up, yet should suffice for many applications.
A piece-wise MCF    is a finite collection of MCF's $M^i_t$ on time intervals $[t_i , t_{i+1}]$ so that each $M^{i+1}_{t_{i+1}}$ is the graph over $M^i_{t_{i+1}}$ of a function $u_{i+1}$,
   \begin{align}
   	\Area \, \left( M^{i+1}_{t_{i+1}} \right) &= \Area \, \left( M^i_{t_{i+1}} \right) \, , \\
    \lambda \, \left( M^{i+1}_{t_{i+1}} \right) &\leq  \lambda \, \left( M^i_{t_{i+1}} \right) \, .
   \end{align}
With this definition, area is non-increasing in $t$ even across the jumps.

For simplicity, we assume  in  Theorem  \ref{c:grayson} below that the MCF is smooth up to and including the first singular time.  As mentioned, this would be the case for any MCF in $\RR^3$ if the multiplicity one conjecture, Conjecture \ref{c:m1}, holds.

 The following theorem can be thought of as a generalization of the  results of Gage-Hamilton \cite{GH}, Grayson \cite{G1},
 and Huisken \cite{H1}.

 \begin{Thm}	\label{c:grayson}
For any closed embedded surface $\Sigma \subset \RR^3$, there exists a piece-wise MCF $M_t$ starting at $\Sigma$ and defined up to time $t_{0}$ where the surfaces become singular.  Moreover, $M_t$ can be chosen so that if
\begin{equation}	\label{e:boundeddiam}
\liminf_{t \to t_{0}} \, \, \frac{\text{diam} M_t}{\sqrt {t_{0}-t}}< \infty\, ,
\end{equation}
then $M_t$ becomes extinct in a round point.
\end{Thm}

The meaning of \eqr{e:boundeddiam} is that each time slice of a tangent flow (at  time $t_0$) has uniformly bounded
diameter after rescaling.

Theorem \ref{c:grayson} will eventually be  a corollary of our main theorem in \cite{CM1} about generic MCF.  However, it follows directly from our  
classification of entropy stable self-shinkers together with compactness of the space of all self-shrinkers with a fixed bound on area and  genus  and  serves to illustrate some of the central ideas about generic MCF.  (The compactness of self-shrinkers with area and genus bound was
 proven in \cite{CM3}.)  Here is why it follows directly from these two results (a detailed proof is given in Section \ref{s:8} of this paper):

\begin{quote}
  Starting at the given surface,   flow by mean curvature till the evolving surface is sufficiently close to a time slice in a tangent flow.  If this time slice is not a sphere yet has diameter bounded by a fixed number, then by the classification of stable self-shrinkers we can find a small graph over it where the entropy has gone down by a fixed amount.  Start the MCF at this new surface and flow till the evolving surface is sufficiently close to a time slice in the corresponding tangent flow.  If this time slice is also not a sphere yet has a fixed bound for the diameter,  we can continue the process of making a replacement and get again that the entropy has gone down a fixed amount.  As the entropy is always positive this process has to terminate and we get the theorem.
Note that the entropy goes down by a definite amount after each replacement follows from the compactness theorem of \cite{CM3}.
  \end{quote}

One consequence of Theorem \ref{c:grayson} is that if the initial surface is topologically not a sphere, then the piece-wise flow must develop a non-compact   (after rescaling) singularity.

Note that even flows that become extinct at a point can develop non-compact singularities.
For example,
Altschuler-Angenent-Giga, \cite{AAG},  constructed a mean convex initial surface in $\RR^3$ whose   MCF is smooth until it becomes extinct in a ``non-compact'' point.  In fact, the tangent flow at the extinction point is a cylinder.  This is in contrast to the case of curves where Grayson,
\cite{G1}, has shown  that all tangent flows are compact; cf. \cite{Ha1} and \cite{H5}.

 \subsection{Higher dimensions}

 As already mentioned, the main result of this paper is the classification of generic singularities. In higher dimensions, this is the following:

 \begin{Thm}	\label{t:nonlin1a}
Suppose that $\Sigma$ is a smooth complete embedded
self-shrinker without boundary and with polynomial volume growth.
\begin{enumerate}
\item If $\Sigma$ is not equal to  $\SS^k\times \RR^{n-k}$, then
there is a graph $\tilde{\Sigma}$ over $\Sigma$ of a  function with arbitrarily small $C^m$ norm (for any fixed $m$) so that
  $\lambda ( \tilde{\Sigma}) < \lambda (\Sigma)$.
  \item If $\Sigma$ is not $\SS^n$ and does not split off a line, then   the function in (1) can be taken to have compact support.
  \end{enumerate}
  In  particular,   in either case, $\Sigma$  cannot arise as a tangent flow to the MCF starting from $\tilde{\Sigma}$.
\end{Thm}

In our earlier theorem showing that in $\RR^3$ entropy stable self-shrinkers are all standard, we assumed that the self-shrinker was smooth.  This was a reasonable assumption in $\RR^3$ for applications as (by the theorem of Ilmanen) any tangent flow in $\RR^3$  is indeed smooth.  However,  examples of Vel\'azquez, \cite{V}, show  that tangent flows need not be smooth in higher dimensions{\footnote{In \cite{V},
Vel\'azquez constructed smooth embedded hypersurfaces evolving by MCF in $\RR^8$ whose tangent flow  at the first  singular time is the static Simons cone and is, in particular, not smooth.}}  and, instead, one has the following well known conjecture
(see page $8$ of \cite{I1}):

\begin{Con}  \label{c:codimconj}
Suppose that $M_0 \subset \RR^{n+1}$is a smooth closed embedded hypersurface.  A time slice of any tangent flow of the MCF starting at $M_0$ has multiplicity one and the singular set is of dimension at most $n-3$.{\footnote{By the dimension reduction of \cite{W5}, the estimate on the singular set would follow from ruling out static planes, quasi-static planes, and various polyhedral cones as potential blow ups.  The key for   proving the conjecture is to rule out higher multiplicity static planes.}}
\end{Con}

 ÊWe will show that when $n\leq 6$ our results classifying entropy stable self-shrinkers in $\RR^{n+1}$ hold even for self-similar shrinkers satisfying the smoothness of Conjecture \ref{c:codimconj}.  In fact, we show here the following stronger result:

 \begin{Thm}
 Theorem \ref{t:nonlin1a} holds when $n\leq 6$ and $\Sigma$ is an oriented integral varifold that is smooth off of a singular set with locally finite $(n-2)$-dimensional Hausdorff measure.
 \end{Thm}

\subsection{Outline of the proof}

The key results of this paper are the classification of entropy-stable self-shrinkers given in
Theorem   \ref{c:nonlin1a} and its higher dimensional generalization
Theorem \ref{t:nonlin1a}.  The main technical tool for proving
this is the concept of $F$-stability  which we will define below.

We will see that there are several  characterizations of self-shrinkers.  One of the most useful is that self-shrinkers are   the critical points for the $F_{0,1}$ functional (see Proposition
 \ref{p:critall} below).
 Since this functional is the volume in a conformally changed metric,
they are also minimal surfaces in the conformally changed metric.
 From either point of view,
 every self-shrinker $\Sigma$ is unstable in the usual sense; i.e., there are always nearby hypersurfaces where the $F_{0 , 1}$ functional is strictly less; see Theorem \ref{t:spectral0} below.  This is because translating a self-shrinker
 in space (or time) always lowers the functional.  The stability that we are interested in, which we will call $F$-stability, mods out for these translations:
 \begin{quote}
 A critical point $\Sigma$ for $F_{x_0,t_0}$ is {\emph{$F$-stable}} if for every compactly supported variation{\footnote{A compactly supported variation $\Sigma_s$ of $\Sigma$ is a one-parameter family of graphs over $\Sigma$ given by
 $\{ x + s \, f(x) \, \nn (x) \, | \, x \in \Sigma \}$ where $\nn$ is the unit normal and $f$ is a  function on $\Sigma$ with compact support.}} $\Sigma_s$ with $\Sigma_0 = \Sigma$, there exist variations $x_s$ of $x_0$ and $t_s$ of $t_0$ that make
 $F'' = \left( F_{x_s , t_s}(\Sigma_s) \right)'' \geq 0$ at $s=0$.
 \end{quote}

 \vskip2mm
 We will show that entropy-stable self-shrinkers that do not split off a line must be $F$-stable:

 \begin{Thm}	\label{t:nonlin1c}
Suppose that $\Sigma \subset \RR^{n+1}$ is a  smooth complete embedded self-shrinker with $\partial \Sigma = \emptyset$,  with polynomial volume growth, and $\Sigma$ does not split off a line isometrically.  If $\Sigma$ is $F$-unstable, then there is a compactly supported variation $\Sigma_s$ with $\Sigma_0 = \Sigma$ so that $\lambda (\Sigma_s) < \lambda (\Sigma)$ for all $s \ne 0$.
\end{Thm}

Thus, we are led to classifying
  $F$-stable self-shrinkers:

 \begin{Thm}	\label{t:liketo}
 If $\Sigma$ is a smooth\footnote{The theorem holds when $n\leq 6$ and $\Sigma$ is an oriented integral varifold that is smooth off of a singular set with locally finite $(n-2)$-dimensional Hausdorff measure.}
 complete embedded self-shrinker in  $\RR^{n+1}$ without boundary and with polynomial volume growth that is $F$-stable with respect to compactly supported variations, then it is either the round sphere or a hyperplane.
 \end{Thm}

 The main steps in the proof of Theorem \ref{t:liketo} are:
 \begin{itemize}
 \item Show that $F$-stability implies mean convexity (i.e., $H \geq 0$).
 \item Classify the mean convex self-shrinkers (see Theorem \ref{t:huisken}
below).
\end{itemize}

The connection between $F$-stability and mean convexity comes from that the mean curvature $H$ 
turns out to be an eigenfunction for the second variation operator for the $F_{0,1}$ functional on a self-shrinker $\Sigma$; see Theorem \ref{t:secvar} and Theorem \ref{t:spectral}.   When $\Sigma$ is closed (so that the spectral theory is easiest), it is then almost immediate to see that $F$-stability is equivalent to $H$ being the lowest eigenfunction which, in turn, is equivalent to that $H$ does not change sign (i.e., mean convexity).  Although spectral theory of open manifolds is  more complicated, we show the corresponding result  in Section \ref{s:noncompactspec} to  get that  $F$-stability implies mean convexity.

The classification of mean convex self-shrinkers began with
  \cite{H3},  where Huisken showed that the only smooth {\underline{closed}}  self-shrinkers with non-negative mean curvature in $\RR^{n+1}$ (for $n > 1$) are round spheres (i.e., $\SS^n$).  When $n=1$,   Abresch and Langer, \cite{AbLa},   had already shown that the circle is the only simple closed self-shrinking curve.
 In a second paper, \cite{H4}, Huisken dealt with the non-compact case.  He showed in \cite{H4} that  the only smooth {\underline{open}} embedded self-shrinkers  in $\RR^{n+1}$
with $H \geq 0$, polynomial volume growth, and $|A|$ bounded are
 isometric products of a round sphere and a linear subspace (i.e. $\SS^k\times \RR^{n-k}\subset \RR^{n+1}$).    We will show  that Huisken's classification holds even without the $|A|$ bound which will be crucial for our applications:

 \begin{Thm}	\label{t:huisken}
 $\SS^k\times \RR^{n-k}$ are the only smooth complete embedded self-shrinkers without boundary, with polynomial volume growth, and   $H \geq 0$ in $\RR^{n+1}$.
  \end{Thm}

  The $\SS^k$ factor in Theorem \ref{t:huisken} is round and has radius $\sqrt{2k}$; we allow the possibilities of a hyperplane (i.e., $k=0$) or a sphere ($n-k = 0$).

\subsection{Conventions and notation}
Using the definition
\eqr{e:defH},  $H$ is $2/R$ on the sphere of radius $R$ in $\RR^3$ and $H$ is $1/R$ on the cylinder of radius $R$.   If $e_i$ is an orthonormal frame for $\Sigma$ and $\nn$ is a unit normal, the coefficients of the second fundamental form are defined to be
\begin{equation}
	a_{ij} = \langle \nabla_{e_i} e_j , \nn \rangle \, .
\end{equation}
In particular, we have
\begin{equation}	\label{e:aij}
	\nabla_{e_i} \nn  = - a_{ij} e_j \, .
\end{equation}
Since $\langle \nabla_{\nn} \nn , \nn \rangle =0$, we have that $H = \langle \nabla_{e_i} \nn , e_i \rangle = - a_{ii}$ where by convention we are summing over repeated indices.

When $L$ is a differential operator, we will say that $u$ is an eigenfunction with eigenvalue $\mu$ if $L \, u = - \mu \, u$ and $u$ is not identically zero.  The Laplacian $\Delta$ is defined to be the divergence of the gradient; thus, on $\RR^n$, it is given by $\Delta u = \sum_{i=1}^n u_{x_i x_i}$.  With this convention, the eigenvalues of the Laplacian are non-negative on any closed manifold.

The parabolic distance between two pairs of points $(x_1,t_1)$ and $(x_2,t_2)$ in space-time is denoted by $d_P((x_1,t_1),(x_2,t_2))$ and equal to $\max \{|x_1-x_2|,\sqrt{|t_1-t_2|}\}$.

Finally, there is a second, and related, way to study the asymptotic structure near a singularity.  In this case, one considers a new flow which combines the mean curvature flow with a continuous rescaling process $\tilde{M}_s = \frac{1}{\sqrt{-t } } \, M_t$ where $s(t)= - \log (-t)$ is a  reparameterization of time.  The resulting flow is called the {\emph{rescaled mean curvature flow}}{\footnote{Our definition of the rescaled flow differs slightly from Huisken's definition in \cite{H3} since Huisken's rescaling fixes the time $t= - \frac{1}{2}$ slice rather than the $t= -1$.}} and a one-parameter family $\tilde{M}_s$   satisfies it if
\begin{equation}
	\left( \partial_s x \right)^{\perp} =   - H \, \nn + \frac{x}{2}  \, ,
\end{equation}
where $x$ is the position vector.  In this case, we are rescaling about $x=0$ and $t=0$, but translations in space and time give similar equations for rescaling about other points.   It is not hard to see that the fixed points of the rescaled mean curvature flow are precisely the self-shrinkers. The result of Huisken, \cite{H1}, about extinction in round points, can be reformulated in terms rescaled MCF flow as stating that the  rescaled MCF of a convex hypersurface exists for all time and converges asymptotically to the round sphere.   Obviously, this implies that the tangent flows at the singularity are round spheres.

\tableofcontents

\section{The $F$ functional and Huisken's monotonicity formula}

We will need to recall Huisken's monotonicity formula (see \cite{H3}, \cite{E1},
\cite{E2}).
To do this, first
define the non-negative function $\Phi$ on $\RR^{n+1} \times (-\infty,0)$ by
 \begin{equation}
\Phi  (x,t)  = [-4\pi t]^{-\frac{n}{2}}
\,\e^{\frac{|x|^2}{4t}}\, ,
\end{equation}
and then set
$\Phi_{(x_0,t_0)} (x,t) = \Phi (x-x_0, t-t_0)$.
  If $M_t$ is a solution to
the MCF and $u$ is a $C^2$ function, then
\begin{equation}    \label{e:huisken}
\frac{d}{dt} \int_{M_t} u \, \Phi_{(x_0,t_0)}=
-\int_{M_t} \left| H\nn- \frac{(x-x_0)^{\perp}
}{2\,(t_0-t)}\right|^2 \,u\,\Phi_{(x_0,t_0)} + \int_{M_t} \left[
u_t - \Delta u \right] \, \Phi
 \, .
\end{equation}
When $u$ is identically one, we get the monotonicity formula
\begin{equation}    \label{e:huisken2}
\frac{d}{dt} \int_{M_t}   \Phi_{(x_0,t_0)}=
-\int_{M_t} \left| H\nn-\frac{(x-x_0)^{\perp}
}{2\,(t_0-t)}\right|^2 \,\Phi_{(x_0,t_0)}
 \, .
\end{equation}
Huisken's density is the limit of $\int_{M_t} \Phi_{x_0,t_0}$ as $t\to t_0$.
That is,
\begin{equation}
\Theta_{x_0,t_0}=\lim_{t\to t_0} \int_{M_t}\Phi_{x_0,t_0}\, ;
\end{equation}
this limit exists by the monotonicity \eqr{e:huisken2} and the density is non-negative as the integrand $\Phi_{x_0,t_0}$ is non-negative.    Note also that it easily follows from the monotonicity that Huisken's density is upper semi-continuous
in the following sense:  Given $x_0$, $t_0$, and $\epsilon>0$, there exists $\delta>0$ such that if $d_P((x,t),(x_0,t_0))<\delta$, then
\begin{equation}
\Theta_{x,t}<\Theta_{x_0,t_0}+\epsilon\, .
\end{equation}

The $F$ functional and Huisken's weighted volume are related by the following identity
\begin{equation}
F_{x_0,t_0}(M_{-1})=(4\pi t_0)^{-n/2}\int_{M_{-1}}\e^{-\frac{|x-x_0|^2}{4t_0}}=\int_{M_{-1}}   \Phi_{(x_0,t_0)}(x,0)\, .
\end{equation}

\begin{Rem}
The   density was defined so that it is one on a hyperplane.
In \cite{St}, A. Stone calculated Huisken's density, and thus the $F$ functional, on spheres and cylinders.  On spheres, the density is $4/\e \approx 1.47$ and on cylinders it is $\sqrt{2\pi/\e} \approx 1.52$.
\end{Rem}

We will use two simple properties of the $F_{x_0 , t_0} $ functional defined in \eqr{e:Ft0}.  The first is that if we scale $\Sigma$ by some $\alpha > 0$ about the origin, then the scaled hypersurface $\alpha \, \Sigma$ satisfies
\begin{equation}	\label{e:scaleF}
		F_{0 , \alpha^2 \, t_0} (\alpha \, \Sigma)  = F_{0 , t_0} (\Sigma)  \, .
\end{equation}
Similarly, we get a corresponding equality for rescalings about an arbitrary $x_0 \in \RR^{n+1}$.
The second property that we will need is that if the one-parameter family of hypersurfaces $M_t$ flows by mean curvature and $t > s$, then
Huisken's monotonicity formula  \eqr{e:huisken2} gives
\begin{equation}	\label{e:huiskenF}
		F_{x_0 , t_0} ( M_t)  \leq  F_{x_0 , t_0 + (t-s)} (M_s)  \, .
\end{equation}

Recall that we defined the entropy $\lambda = \lambda (\Sigma)$ of $\Sigma$ in \eqr{e:lamb} as the supremum  of the $F_{x_0,t_0}$ functionals
\begin{equation}	
	\lambda =   \sup_{x_0,t_0} \, \, F_{x_0 , t_0} (\Sigma)   \, .
\end{equation}
We conclude that the entropy is monotone under the mean curvature flow:

\begin{Lem}	\label{l:entromono}
The entropy $\lambda$  is non-increasing in $t$ if the hypersurfaces $M_t$ flow by mean curvature or by the rescaled mean curvature flow.
\end{Lem}

\begin{proof}
The first claim follows immediately from \eqr{e:huiskenF}, while the second claim follows from  \eqr{e:huiskenF} and \eqr{e:scaleF}.
\end{proof}

%We also have that the entropy is lower semi-continuous:

%\begin{Cor}
%($\lambda$ is lower semi-continuous.)
%\end{Cor}

%\begin{proof}
%\end{proof}

\section{Self-shrinkers and self-shrinking MCF}	\label{s:selfs}

A hypersurface is said to be a {\emph{self-shrinker}} if it is the time $t=-1$ slice{\footnote{In \cite{H3}, Huisken defines self-shrinkers to be the time $t= - \frac{1}{2}$ slice of a self-shrinking MCF; consequently, Huisken gets that $H = \langle x , \nn \rangle$.}} of a self-shrinking MCF that disappears at $(0,0)$, i.e., of a MCF satisfying  $M_t = \sqrt{-t} \, M_{-1}$.  Thanks to the next lemma we will later use slight abuse of notation and identify a self-shrinker with the corresponding self-shrinking MCF and sometimes we will also simply think of a self-shrinker as a hypersurface satisfying the following equation for the mean curvature and the normal
\begin{equation}
H=\frac{\langle x,\nn\rangle}{2}\, .\label{e:selfsSigma}
\end{equation}

 \begin{Lem}	\label{l:ss}
If a hypersurface $\Sigma$ satisfies  \eqr{e:selfsSigma}, then $M_t=\sqrt{-t}\,\Sigma$ flows by mean curvature and
\begin{equation} \label{e:selfs}
H_{M_t}=-\frac{\langle  x,\nn_{M_t}\rangle}{2t}\, .
\end{equation}

Conversely if a one-parameter family of hypersurfaces $M_t$ flows by mean curvature, then $M_t$ satisfies  $M_t = \sqrt{-t} \, M_{-1}$
 if and only if
  $M_{t}$ satisfies \eqr{e:selfs}.
\end{Lem}

\begin{proof}
If $\Sigma$ is a hypersurface that satisfies that $H=\frac{\langle x,\nn\rangle}{2}$, then we set $M_t=\sqrt{-t}\,\Sigma$ and $x(p,t)=\sqrt{-t}\,p$ for $p\in \Sigma$.  It follows that $\nn_{M_t}(x(p,t))=\nn_{\Sigma}(p)$, $H_{M_t}(x(p,t))=\frac{H_{\Sigma}(p)}{\sqrt{-t}}$, and $\partial_t x=-\frac{p}{2\sqrt{-t}}$.  Thus, $(\partial_t x)^{\perp}=-\frac{\langle p,\nn\rangle}{2\sqrt{-t}}=-H_{M_t}(x(p,t))$.  This proves that $M_t$ flows by mean curvature and shows \eqr{e:selfs}.

On the other hand, suppose that $M_t$ is a one-parameter family of hypersurfaces flowing by mean curvature.
A computation shows that
\begin{equation}
(-t)^{\frac{3}{2}} \, \partial_t\left( \frac{x}{\sqrt{-t}}\right)=-t\, \partial_t x +\frac{x}{2} \, .
\end{equation}
If $\frac{M_t}{\sqrt {-t}}=M_{-1}$, then
\begin{equation}
0=(-t)^{\frac{3}{2}} \, \langle \partial_t\left( \frac{x}{\sqrt{-t}}\right),\nn_{M_{-1}}\rangle=-t\,\langle  \partial_t x,\nn_{M_{-1}}\rangle +\frac{1}{2}\langle x,\nn_{M_{-1}}\rangle\, .
\end{equation}
Hence, since $M_t$ flow by the mean curvature it follows that
\begin{equation}
H_{M_{-1}}=-\langle \partial_t x,\nn_{M_{-1}}\rangle = \frac{\langle x , \nn_{M_{-1}} \rangle }{2} \, .
\end{equation}
The equation for $H_{M_t}$ for general $t$ follows by scaling.

Finally, if $M_t$ flows  by mean curvature and satisfies \eqr{e:selfs}, then, by the first part of the
lemma, it follows that $N_t=\sqrt{-t}M_{-1}$ also flows by the mean curvature and has the same
initial condition as $M_t$; thus $M_t=N_t$ for $t\geq -1$.
\end{proof}	

\begin{Cor}
If the hypersurfaces $M_t$ flow by mean curvature, then $M_t$ is a self-similar shrinking MCF if and only $\int_{M_t}\Phi$ is constant.
\end{Cor}

\begin{proof}
This follows by combining Lemma \ref{l:ss} with \eqr{e:huiskenF}.
\end{proof}

Another consequence of Lemma \ref{l:ss} is the following standard corollary:

\begin{Cor}	\label{c:mincone}
If $\Sigma$ is a self-shrinker and $H \equiv 0$, then $\Sigma$ is a minimal cone.  In particular, if $\Sigma$ is also smooth and embedded, then it is a hyperplane through $0$.
\end{Cor}

\begin{proof}
Since $H\equiv 0$, $M_t = \Sigma$ is a static solution of the MCF.  Hence, by
  Lemma \ref{l:ss},  $\sqrt{-t} \, \Sigma = \Sigma$ for all $t< 0$ and, thus, $\Sigma$ is a cone.  The second claim follows since the only smooth cone through $0$ is a hyperplane.
\end{proof}

\subsection{Volume bounds and a priori bounds for blow ups}

Let $M_t\subset \RR^{n+1}$ be a MCF with initial hypersurface $M_0$ where $M_0$ is smooth closed and embedded.

\begin{Lem}  \label{l:volb}
If $M_0$ and $M_t$ are as above and $T>0$ is given, then there exists a constant $V = V(\Vol (M_0),T)> 0$ so that for all $r>0$, all $x_0\in \RR^{n+1}$, and all $t\geq T$
\begin{equation}
\Vol (B_r(x_0)\cap M_{t})\leq V\,r^n\, .
\end{equation}
\end{Lem}

\begin{proof}
For any $t_0>t$ to be chosen later,  Huisken's monotonicity formula gives
\begin{align}
[4\pi (t_0-t)]^{-\frac{n}{2}}\,  \e^{-\frac{1}{4}}\,   \Vol (B_{\sqrt{t_0-t}}(x_0)\cap M_t)
 \leq &
[4\pi (t_0-t)]^{-\frac{n}{2}}     \int_{B_{\sqrt{t_0-t}}(x_0)\cap M_t}                 \e^{\frac{|x-x_0|^2}{4(t-t_0)}}\notag\\
 \leq &
\int_{M_t}\Phi_{x_0,t_0}(\cdot,t)\leq  \int_{M_0}\Phi_{x_0,t_0}(\cdot,0)\\
 =  &  [4\pi t_0]^{-\frac{n}{2}}     \int_{M_0}                                                         \e^{-\frac{|x-x_0|^2}{4t_0}}
\leq (4\pi T)^{-\frac{n}{2}}\, \Vol (M_0)\, .\notag
\end{align}
Hence,
\begin{equation}
\Vol (B_{\sqrt{t_0-t}}(x_0)\cap M_t)\leq \e^{\frac{1}{4}}\,\left(\frac{(t_0-t)}{T}\right)^{\frac{n}{2}}\, \Vol (M_0)\, .
\end{equation}
Setting $t_0=t+r^2$ and $V = \e^{\frac{1}{4}} \, \Vol \, (M_0) \, T^{ - \frac{n}{2}}$, the claim now easily follows.
\end{proof}

Given a MCF $M_t$ as above a {\it{limit flow}} (or a {\it{blow up flow}}), $\cB_t$, is defined as follows (cf., for instance, page 675--676 of \cite{W2} and chapter 7 of \cite{I2}).  Let $(x_j,t_j)$ be a sequence of points in space-time and $c_j$ a sequence of positive numbers with $c_j\to \infty$.  A {\it{limit flow}} is a Brakke limit of the sequence of rescaled flows $t\to c_j\,(M_{c^{-2}_j\,t+t_j}-x_j)$.  Such limits exist by Brakke's compactness theorem, \cite{B}, and the a priori volume bound in Lemma \ref{l:volb}.  Note that when all $x_j$ are equal to a common point $x_0$ and the $t_j$'s are equal to a common time $t_0$, then the limit flow is a tangent flow at $(x_0,t_0)$ (translated to the origin in space-time).

\begin{Cor}
If $M_0$ and $M_t$ are as above, then there exists a constant $V > 0$ so that for any limit flow $\cB_t$, all $r>0$, $x_0\in \RR^{n+1}$, and all $t$
\begin{equation}
\Vol (B_r(x_0)\cap \cB_{t})\leq V\,r^n\, .
\end{equation}
\end{Cor}

\begin{proof}
This follows immediately from the lemma after noting that, since the initial hypersurface $M_0$ is smooth, the original
 flow remains smooth for a short time and, thus, any blow up near (in time) the initial hypersurface is a static plane.  Thus, to prove the corollary, we may assume that the blow up is for times $t_j \geq T$ for some $T > 0$; the corollary now follows from the lemma.
\end{proof}

If $\Sigma$ is a hypersurface in $\RR^{n+1}$, then we say that $\Sigma$ has polynomial volume growth if there are constants $C$ and $d$ so that for all $r \geq 1$
\begin{equation}	\label{e:poly}
	\Vol \, (B_r(0) \cap \Sigma) \leq C \, r^d \, .
\end{equation}

Note that, by the corollary, any time-slice of a blow up has polynomial volume growth.

\section{A variational characterization of self-shrinkers}	\label{s:one}

In Lemma \ref{l:ss} in the previous section, we saw that a hypersurface is a self-shrinker is equivalent to that:
\begin{itemize}
\item  It satisfies the second order differential equation  ${H} = \,  \frac{\langle x , \nn \rangle }{2}$.
\end{itemize}

In this section, we will see that there are two other equivalent characterizations for  a hypersurface $\Sigma$ to be a self-shrinker:
\begin{itemize}
\item It is a minimal hypersurface for the conformally changed metric $g_{ij} =
\e^{-\frac{|x|^2}{2n}} \, \delta_{ij}$ on $\RR^{n+1}$.
\item It is a  critical points for   the $F_{x_0,t_0}$ functionals with respect to variations in all three parameters, i.e., critical for variations of $\Sigma$, $x_0$, and $t_0$.
\end{itemize}
The equivalence of these is proven in Lemma \ref{l:varl0} and Proposition \ref{p:critall}, respectively.

Unless explicitly stated otherwise, the results of the rest of the paper hold for hypersurfaces
    in $\RR^{n+1}$ for any $n$.

\subsection{The first variation of $F_{x_0 , t_0}$}

We will say that the one-parameter family of smooth embedded hypersurfaces $\Sigma_s \subset \RR^{n+1}$ is a variation of $\Sigma$ if $\Sigma_s$ is given by a one-parameter family of embeddings $X_s : \Sigma \to \RR^{n+1}$ with $X_0$ equal to the identity.  We will call the vector field $\frac{\partial X_s}{\partial s} \big|_{s=0}$ the variation vector field.

 \begin{Lem}	\label{l:varl0}
Let $\Sigma_s \subset \RR^{n+1}$ be a variation of $\Sigma$ with variation vector field
$\Sigma_0' = f \, \nn$.   If $x_s$ and $t_s$ are variations of $x_0$ and $t_0$ with   $x_0' = y$ and
 $t_{0}' = h$, then $\frac{\partial }{\partial s} \,  \, \left( F_{ x_s , t_s} (\Sigma_s) \right)$ is
\begin{equation}	\label{e:Fprime0}
	  (4\pi \, t_0)^{-\frac{n}{2}} \,  \int_{\Sigma}  \left[ f \, \left( H - \frac{ \langle x-x_0 ,  \nn \rangle}{2t_0}
		\right)
		+ h \, \left( \frac{|x-x_0|^2}{4t_0^2} - \frac{n}{2t_0} \right)  +
	  \frac{  \langle x-x_0 , y \rangle }{2t_0}    \right] \, \e^{\frac{-|x-x_0|^2}{4t_0}} \, d\mu \, .
\end{equation}
\end{Lem}

\begin{proof}
 From the first variation formula (for area), we know that
 \begin{equation}	\label{e:fhdm}
 	(d\mu)' = f \, H \, d \mu \, .
 \end{equation}
 The $s$ derivative of the weight  $\e^{-|x-x_s|^2/(4t_s)}$ will have three separate terms coming from the variation of the surface, the variation of $x_s$, and the variation of $t_s$.  Using, respectively, that $\nabla  |x-x_s|^2  = 2\, (x-x_s)$,
 \begin{equation}
	\partial_{t_s}  \log \left[ (4\pi t_s)^{-n/2} \,   \e^{-\frac{|x-x_s|^2}{4t_s}} \right] = \frac{-n}{2t_s} + \frac{|x-x_s|^2}{4t_s^2}
\end{equation}
and  $\partial_{x_s}  |x-x_s|^2  = 2\, (x_s-x)$, we get that
the derivative of  $\log \, \left[ (4\pi \, t_s)^{-\frac{n}{2}} \, \e^{-|x-x_s|^2/(4t_s)} \right]$ at $s=0$ is given by
 \begin{equation}	\label{e:logder}
 	  - \frac{ f}{2t_0} \, \langle x-x_0 , \nn \rangle  + h \, \left( \frac{|x-x_0|^2}{4t_0^2} - \frac{n}{2t_0} \right) +
	  \frac{ 1}{2t_0} \, \langle x-x_0 , y \rangle \, .
 \end{equation}
 Combining this with \eqr{e:fhdm} gives \eqr{e:Fprime0}.
\end{proof}

 \subsection{Critical points of $F_{x_0,t_0}$ are self-shrinkers}

 We will say that $\Sigma$ is a critical point for $F_{x_0 , t_0} $ if it is critical with respect to all normal variations in $\Sigma$ and all variations in $x_0$  and   $t_0$.  Strictly speaking, it is the triplet $(\Sigma, x_0 , t_0)$ that is a critical point of $F$, but we will refer   to $\Sigma$ as a critical point of $F_{x_0 , t_0}$.
 The next proposition shows that $\Sigma$ is a critical point for $F_{x_0 , t_0} $ if and only if it is the time $-t_0$ slice of a self-shrinking solution of the mean curvature flow that becomes extinct at the point $x_0$ and time $0$.

 \begin{Pro}	\label{p:critall}
$\Sigma$ is a critical point for $F_{x_0 , t_0} $ if and only if  $H = \frac{ \langle x-x_0 ,  \nn \rangle}{2t_0}$.
\end{Pro}

\begin{proof}
It suffices to show that if $H = \frac{ \langle x-x_0 ,  \nn \rangle}{2t_0}$, then the last two terms in \eqr{e:Fprime0} vanish for every $h$ and every $y$.  Obviously it is enough to show this for $x_0=0$ and $t_0=1$ and this follows from Lemma \ref{l:rightsign}   below (or more precisely \eqr{e:rsa} and \eqr{e:mass} below).
\end{proof}

\subsection{Identities at self-shrinkers}  In this subsection,
 $\Sigma \subset \RR^{n+1}$ is a smooth embedded hypersurface; $\Delta$, $\dv$, and $\nabla$ are the (submanifold) Laplacian, divergence, and gradient, respectively, on $\Sigma$.

 We will need below that the linear operator
 \begin{equation}	\label{e:divww}
       \cL\,v=\Delta \, v -\frac{1}{2}\,  \langle  x , \nabla v  \rangle
	= \e^{\frac{|x|^2}{4}} \,   \dv \, \left( \e^{\frac{-|x|^2}{4}} \,  \nabla v \right)
\end{equation}
  is self-adjoint in a weighted $L^2$ space.   That is, the following lemma:

  \begin{Lem}   \label{l:self}
  If $\Sigma \subset \RR^{n+1}$ is a hypersurface, $u$ is a $C^1$ function with compact support,
  and $v$ is a $C^2$ function, then
  \begin{equation}
\int_{\Sigma} u \, (\cL\,v) \, \e^{\frac{-|x|^2}{4}}  =
	- \int_{\Sigma}  \langle \nabla v , \nabla u  \rangle \, \e^{\frac{-|x|^2}{4}} \, .	\label{e:selfadjoint0}
\end{equation}
  \end{Lem}

  \begin{proof}
The lemma follows immediately from Stokes' theorem and \eqr{e:divww}.
\end{proof}

The next corollary is an easy extension of Lemma \ref{l:self} used later to justify computations when $\Sigma$ is not closed.

  \begin{Cor}   \label{c:self}
  Suppose that $\Sigma \subset \RR^{n+1}$ is a complete hypersurface without boundary.
   If $u , v$ are $C^2$ functions with
   \begin{equation}	\label{e:inl2}
   	\int_{\Sigma} \left( |u \, \nabla v| + |\nabla u| \, |\nabla v| + | u \, \cL\, v| \right) \, \e^{\frac{-|x|^2}{4}}  < \infty \, ,
   \end{equation}
   then we get
  \begin{equation}	
  \int_{\Sigma} u \, (\cL\, v) \, \e^{\frac{-|x|^2}{4}}  =
	- \int_{\Sigma}  \langle \nabla v , \nabla u  \rangle \, \e^{\frac{-|x|^2}{4}} \, .	\label{e:selfadjoint}
\end{equation}
  \end{Cor}

  \begin{proof}
   Within this proof, we will use square brackets $\left[ \cdot \right]$ to denote weighted integrals
   \begin{equation}
    \left[ f \right] = \int_{\Sigma} f \, \e^{\frac{-|x|^2}{4}} \, .
   \end{equation}
  Given any $\phi$ that is $C^1$ with compact support, we can apply Lemma \ref{l:self} to $\phi \, u$ and $v$ to get
 \begin{equation}	\label{e:phite}
\left[ \phi \, u \, \cL\, v\right]
    =
	- \left[  \phi \,   \langle \nabla v , \nabla u  \rangle  \right]
    - \left[  u \,   \langle \nabla v , \nabla \phi \rangle  \right] \, .
\end{equation}
Next, we apply this with $\phi = \phi_j$ where $\phi_j$ is one on the intrinsic ball of radius $j$ in $\Sigma$ centered at a fixed point and $\phi_j$ cuts off linearly to zero between $B_j$ and $B_{j+1}$.  Since $|\phi_j|$ and $|\nabla \phi_j |$ are bounded by one, $\phi_j \to 1$, and $|\nabla \phi_j| \to 0$, the dominated convergence theorem (which applies because of \eqr{e:inl2})
gives that
\begin{align}
	\left[ \phi_j \, u \, \cL\,v\right]  & \to
	\left[  u \, \cL\,v\right]   \, , \\
	\left[  \phi_j \,   \langle \nabla v , \nabla u  \rangle  \right] &\to \left[    \langle \nabla v , \nabla u  \rangle  \right] \, , \\
	\left[  u \,   \langle \nabla v , \nabla \phi \rangle  \right]  &\to 0 \, .
\end{align}
Substituting this into \eqr{e:phite} gives the corollary.
\end{proof}

We will use Corollary \ref{c:self} in the proof of Lemma \ref{l:rightsign} below.  To keep things short, we will say that a function $u$ is ``in the weighted $W^{2,2}$ space'' if
\begin{equation}	\label{e:wted22}
	\int_{\Sigma} \left( |u|^2 + |\nabla u|^2 + |\cL\, u |^2 \right) \,  \e^{\frac{-|x|^2}{4}}  < \infty \, .
\end{equation}
The point is that if $u$ and $v$ are both in the weighted $W^{2,2}$ space, then \eqr{e:inl2} is satisfies and we can apply Corollary \ref{c:self} to $u$ and $v$.  In particular, since the right side of \eqr{e:selfadjoint} is symmetric in $u$ and $v$, it follows that
 \begin{equation}	\label{e:selfc2}
\int_{\Sigma} u \, (\cL\, v)\, \e^{\frac{-|x|^2}{4}}  =
\int_{\Sigma} v \, (\cL\, u) \, \e^{\frac{-|x|^2}{4}}   \, .
\end{equation}

The next lemma applies the self-adjoint operator  $\cL$ to natural geometric quantities on a self-shrinker;   $x_i$ is the $i$-th component of the position vector $x$, i.e., $x_i = \langle x , \partial_i \rangle$.

 \begin{Lem}	\label{l:opquant}
 If $\Sigma \subset \RR^{n+1}$ is a hypersurface with  $H = \frac{ \langle x , \nn \rangle}{2}$, then
 \begin{align}
 	\cL\, x_i &= - \frac{1}{2} x_i \, ,  \label{e:deltaxi}  \\
 	\cL\,|x|^2 &= 2n - |x|^2 \, .   \label{e:operx2}
 \end{align}
 \end{Lem}

 \begin{proof}
  Since $\Delta \, x = - H \, \nn $ on any hypersurface and
  $H = \frac{ \langle x , \nn \rangle}{2}$ on a self-shrinker, we have
  \begin{equation}
 	\Delta x_i = \langle - H \nn , \partial_i \rangle = \frac{-1}{2} \, \langle  x^{\perp} , \partial_i \rangle =
		 \frac{-1}{2} \, \langle  x  , \partial_i \rangle +  \frac{1}{2} \, \langle  x  ,  (\partial_i )^T \rangle \, ,
 \end{equation}
 giving \eqr{e:deltaxi}.  Combining this with
  $\nabla |x|^2 = 2 \, x^{T}$ where $x^T$ is the tangential projection of $x$ gives
  \begin{equation}  \label{e:deltax2}
 	\Delta |x|^2 = 2 \, \langle \Delta x , x \rangle + 2\, |\nabla x|^2 =
		- \left| x^{\perp} \right|^2 + 2n \, .
\end{equation}
 The second claim \eqr{e:operx2}  follows since $\frac{1}{2} \, \langle  x , \nabla |x|^2 \rangle = \langle  x ,    x^T \rangle = |x^T|^2$.
 \end{proof}

 In the rest of this section,  we will always assume that $\Sigma$ is smooth, complete, $\partial \Sigma = \emptyset$,   $\Sigma$ has polynomial volume growth, and    $H = \frac{ \langle x , \nn \rangle}{2}$.

 Equations \eqr{e:rsa} and \eqr{e:mass} in the next lemma are used in the proof of Proposition \ref{p:critall} whereas \eqr{e:rs}, \eqr{e:mass2}, and the
 corollary that follow the next lemma are first used in the next section when we compute the second
 variation of the $F$ functional at a critical point.

\begin{Lem}	\label{l:rightsign}
If $\Sigma$ is complete, $\partial \Sigma = \emptyset$,   $\Sigma$ has polynomial volume growth, and    $H = \frac{ \langle x , \nn \rangle}{2}$, then
\begin{align}	\label{e:rsa}
   \int_{\Sigma}  \left(   |x|^2 - 2n \right)  \, \e^{\frac{-|x|^2}{4}}  &= 0 \, , \\
     \int_{\Sigma} x \, \e^{\frac{-|x|^2}{4}}  &= 0 =  \int_{\Sigma} x \, |x|^2 \e^{\frac{-|x|^2}{4}}   \, , \label{e:mass} \\
	    \int_{\Sigma}  \left(  |x|^4 - 2n (2n+4) + 16 \, H^2  \right) \, \e^{\frac{-|x|^2}{4}} & = 0 \, . \label{e:rs}
\end{align}
 Finally, if $w \in \RR^{n+1}$ is a constant vector, then
  \begin{equation}	\label{e:mass2}
 	\int_{\Sigma} \langle x , w \rangle^2 \, \e^{\frac{-|x|^2}{4}}  = 2 \, \int_{\Sigma} \left|   w^T \right|^2 \, \e^{\frac{-|x|^2}{4}} \, .
 \end{equation}
\end{Lem}

 \begin{proof}
 In  this proof,  we will use
 square brackets $\left[ \cdot \right]$ to denote  weighted integrals
 \begin{equation}
 	\left[ f \right] =       \int_{\Sigma}   f \,   \e^{\frac{-|x|^2}{4}}  \, .
 \end{equation}
 We will repeatedly use that the constant functions, the function $x_i$, the function $\langle x , w \rangle$, and the function $|x|^2$ are all in the weighted $W^{2,2}$ space and, thus, that we can apply (the self-adjointness)
 Corollary \ref{c:self} to any pair of these functions.

 To get the first claim \eqr{e:rsa}, use $u=1$ and $v= |x|^2$ in \eqr{e:selfc2} to get
 \begin{equation}
 	0 = \left[ 1 \, \cL |x|^2 \right] = \left[ 2n - |x|^2 \right] \, ,
 \end{equation}
 where the last equality used that $\cL |x|^2 = 2n - |x|^2$ by \eqr{e:operx2} in Lemma \ref{l:opquant}.  The first equality in \eqr{e:mass} follows similarly by taking $u=1$
 and $v = x_i$ and using that $\cL x_i = - \frac{1}{2} \, x_i $ by \eqr{e:deltaxi}.

 To get the second equality in \eqr{e:mass}, argue similarly
 with $u= x_i$ and $v = |x|^2$ to get
  \begin{equation}
  	- \frac{1}{2} \left[ x_i \, |x|^2 \right] = \left[ |x|^2 \, \cL x_i \right] = \left[ x_i \, \cL |x|^2 \right] = \left[ x_i \, (2n - |x|^2) \right] =
		-   \left[ x_i \, |x|^2 \right]  \, ,
 \end{equation}
 where  the last equality used that $\left[ x_i \right] = 0$.  Since the constants in front of the $\left[ x_i \, |x|^2 \right] $ terms differ, we get that
$ \left[ x_i \, |x|^2 \right] = 0$ as claimed.

 To get \eqr{e:rs}, apply \eqr{e:selfadjoint} with $u = v = |x|^2$ and \eqr{e:operx2} to get
 \begin{equation}
 		 2n \, \left[   |x|^2 \right] - \left[ |x|^4  \right]  =
		  \left[ |x|^2 \, \cL |x|^2 \right] = - \left[ \left| \nabla |x|^2 \right|^2 \right] = - 4 \, \left[ \left| x^T \right|^2 \right]
		    \, .
 \end{equation}
 This gives \eqr{e:rs} since $\left| x^T \right|^2 =   |x|^2 -   \left| x^{\perp} \right|^2    =   |x|^2 - 4 \, H^2$ and $\left[ |x|^2 \right] = \left[ 2n \right]$.
Finally,
  the last claim \eqr{e:mass2} follows from \eqr{e:selfadjoint} with $u = v = \langle x , w \rangle$ since $\cL\, \langle x , w \rangle =
  - \frac{1}{2}  \langle x , w \rangle$ and $\nabla   \langle x , w \rangle = w^T$.
   \end{proof}

\begin{Cor}	\label{c:rs}
If $\Sigma$ is as in Lemma \ref{l:rightsign}, then
\begin{equation}
 \int_{\Sigma}  \left[  \left( \frac{|x|^2}{4} - \frac{n}{2} \right)^2 - \frac{n}{2}
		\, \right] \, \e^{\frac{-|x|^2}{4}}   =  -    \int_{\Sigma} H^2
 \,   \e^{\frac{-|x|^2}{4}}  \, .
 \end{equation}
\end{Cor}

\begin{proof}
As in the proof of Lemma \ref{l:rightsign}, square brackets $\left[ \cdot \right]$ will denote  weighted integrals.
Squaring things out gives
\begin{align}
 4 \,   \left[  \left( \frac{|x|^2}{4} - \frac{n}{2} \right)^2 - \frac{n}{2}
		\, \right] & = \left[    \frac{|x|^4}{4} -  n \, |x|^2   +  n^2 - 2 \, n \right]  \notag \\
		&= \left[  n (n+2) - 4 \, H^2  -  n (2n) + n^2 - 2n \right]  = - 4 \, \left[ H^2 \right] \, ,
\end{align}
where the second equality used Lemma \ref{l:rightsign}.
\end{proof}

\section{The second variation of $F_{x_0,t_0}$ and the operator $L_{x_0,t_0}$} \label{s:two}

In this section, we will calculate the second variation formula of the $F$ functional for simultaneous variations in all three parameters $\Sigma$, $x_0$, and $t_0$.  The most important case is when $\Sigma$ is a critical point where we will use our calculation to formulate a notion of stability.

As we saw in Section \ref{s:one},    critical points of the $F$ functional are the same as critical points for $F_{x_0,t_0}$ where $x_0$ and $t_0$ are fixed and we vary the hypersurface alone and these are easily seen to characterize self-shrinkers and are equivalent to being a
 minimal hypersurface in $\RR^{n+1}$  in a conformally changed metric.   In Section \ref{s:three} we will see that
 it turns out that as minimal hypersurfaces in the conformally changed metric, they are unstable in the usual sense of minimal surfaces (cf. chapter $1$ in \cite{CM4}).
  Equivalently,
  every critical point of the $F$ functional is unstable if you fix $x_0$ and $t_0$ and vary $\Sigma$ alone.

  In spite of this, it is still possible to formulate a natural version of stability for the full $F$ functionals.  The point is that the apparent instability comes from translating   the hypersurface in space and time and  our notion of stability will account for this.

 \subsection{The general second variation formula}

 In the next theorem, $\Sigma_s$ will be a one-parameter family of hypersurfaces where the $s$ derivative $\partial_s \Sigma_s$ will be assumed to be normal to $\Sigma_s$ for each $s$.  Namely,
 we assume that  $\partial_s \Sigma_s = f \, \nn$ where $f$ is a function and $\nn$ is the unit normal to $\Sigma_s$ (and we have suppressed the $s$ dependence).
 This allows us to apply Lemma \ref{l:varl0} for each $s$ and then differentiate with respect to $s$.

  \begin{Thm}	\label{t:secvar0}
 If $\Sigma_s$ is a normal variation of $\Sigma$, $x_s$ is a variation of $x_0$,  and $t_s$ is a variation of $t_0$ with
 \begin{align}
 	\partial_s \big|_{s=0} \, \Sigma_s = f \nn ,  \,  \partial_s  \big|_{s=0} x_s = y ,   {\text{ and }} \partial_s  \big|_{s=0} t_s = h \, ,  \\
	\partial_{s} \big|_{s=0} \, f = f' ,  \,  \partial_{ss}  \big|_{s=0} x_s = y' ,   {\text{ and }} \partial_{ss}  \big|_{s=0} t_s = h' \, ,
\end{align}
then setting $F'' = \partial_{ss} \big|_{s=0} \, \left( F_{x_s , t_s} (\Sigma_s) \right)$ gives
 \begin{align}	\label{e:secvar20}
	F'' &=    (4\pi t_0)^{-n/2} \, \int_{\Sigma}
	 \left[ - f \,     L_{x_0,t_0} \, f
		  +  f\,h\,\frac{ \langle x - x_0,  \nn \rangle}{t_0^2}    -   h^2 \, \frac{|x-x_0|^2 - nt_0}{2t_0^3}
		\right. \notag \\
		&\quad + f \, \frac{ \langle y , \nn \rangle}{t_0}  - \frac{|y|^2}{2t_0}  - h \, \frac{\langle x-x_0 , y \rangle}{t_0^2} \notag \\
		&\quad +  \left( f \, \left( H - \frac{ \langle x-x_0 ,  \nn \rangle}{2t_0} \right) +
		h \, \left( \frac{|x - x_0|^2}{4 t_0^2} - \frac{n}{2t_0} \right) +   \langle  \frac{ x -x_0}{2t_0}  , y \rangle \right)^2
		   \\
		 &\quad +   \left. f' \, \left( H - \frac{ \langle x-x_0 ,  \nn \rangle}{2t_0}
		\right)
		+ h' \, \left( \frac{|x-x_0|^2}{4t_0^2} - \frac{n}{2t_0} \right)  +
	  \frac{  \langle x-x_0 , y' \rangle }{2t_0}    \right]
		 \, \e^{\frac{-|x-x_0|^2}{4t_0}}  \, d\mu  \, ,\notag
\end{align}
where $L_{x_0,t_0} = \Delta + |A|^2 - \frac{1}{2t_0} \, \langle x - x_0, \nabla (\cdot) \rangle + \frac{1}{2t_0} $ is a (non-symmetric) second order operator.
\end{Thm}

\begin{proof}
Within this proof, we will use square brackets $\left[ \cdot \right]_{x_s,t_s}$ to denote  weighted integrals, i.e.,
 \begin{equation}
 	\left[ g \right]_{x_s,t_s} =  (4\pi t_s)^{- \frac{n}{2} } \,    \int g \,   \e^{\frac{-|x-x_s|^2}{4t_s}} \, d \mu  \, .
 \end{equation}
Letting primes denote derivatives with respect to $s$ at $s=0$, differentiating  \eqr{e:Fprime0} gives
\begin{align}	
	F'' &=  \left[
		 f \, \left( H - \frac{ \langle x - x_s ,  \nn \rangle}{2t_s}
		\right)'
		+ h \, \left( \frac{|x - x_s|^2}{4t_s^2} - \frac{n}{2t_s} \right)'
		+
	     \langle \left( \frac{ x-x_s}{2t_s} \right)' , y \rangle
		\right.   \notag \\
		&\quad +   \left( f \, \left( H - \frac{ \langle x-x_0 ,  \nn \rangle}{2t_0} \right) +
		h \, \left( \frac{|x-x_0|^2}{4t_0^2} - \frac{n}{2t_0} \right) +     \langle \frac{x -x_0}{2t_0} , y \rangle \right)^2
		   \\
		&\quad +
		\left. f' \, \left( H - \frac{ \langle x-x_0 ,  \nn \rangle}{2t_0}
		\right)
		+ h' \, \left( \frac{|x-x_0|^2}{4t_0^2} - \frac{n}{2t_0} \right)  +
	  \frac{  \langle x-x_0 , y' \rangle }{2t_0}    \right]_{x_0,t_0}
		\, ,\notag
\end{align}
where we used \eqr{e:logder} to compute the term on the second line.
By \eqr{e:big4} and \eqr{e:Nip} in Appendix \ref{s:AC}, we have that
\begin{align}
	H' &= - \Delta f - |A|^2 \, f \, , \\
	\nn' &= - \nabla f  \, .		\label{e:fromap}
\end{align}
Since $x' = f \, \nn$,   we get
\begin{equation}	\label{e:xxsp}
	\langle x - x_s , \nn \rangle ' = f -  \langle y , \nn \rangle -  \langle x - x_0 ,  \nabla f \rangle  \, .
\end{equation}
 Hence, since $\left( t_s^{-1} \right)' = -h\, t_0^{-2}$ and $(x_s)' = y$, we have
\begin{align}	
	\left( H - \frac{ \langle x-x_s ,  \nn \rangle}{2t_s}
		\right)'  &=  - \Delta f - |A|^2 \, f  -   \frac{ f - \langle y , \nn \rangle -  \langle x - x_0 ,  \nabla f \rangle }{2t_0} + \frac{h}{2 \, t_0^2} \, \langle x - x_0 , \nn \rangle
		  \notag \\
			&=  -  L_{x_0,t_0} \, f + \frac{h}{2 \, t_0^2} \, \langle x - x_0 , \nn \rangle
		+ \frac{\langle y , \nn \rangle}{2 t_0}
		 \, ,
\end{align}	
where the second equality used the definition of $L_{x_0,t_0}$.   The second term is given by
\begin{align}
	\left( \frac{|x - x_s|^2}{4t_s^2} - \frac{n}{2t_s} \right)'  &= - h \, \frac{|x-x_0|^2}{2t_0^3} +
	\frac{ \langle x - x_0, f \, \nn \rangle}{ 2t_0^2}  -
	\frac{  \langle x - x_0 , y \rangle}{2t_0^2}
	+ \frac{n}{2t_0^2 } \, h \notag  \\
	&=  f \, \frac{ \langle x - x_0,  \nn \rangle}{2t_0^2}
	- h \, \frac{|x-x_0|^2 - nt_0}{2t_0^3}   -
	\frac{\langle x - x_0 , y \rangle}{2t_0^2} \,   .
\end{align}
 For the third term, observe that
 \begin{equation}
 	\left( \frac{ x-x_s}{2t_s} \right)' = \frac{ f \nn - y - h (x-x_0)/t_0}{2t_0}  \, .
 \end{equation}
Combining  these gives the theorem.
  \end{proof}

 \subsection{The second variation at a critical point}

In this subsection, we specialize our calculations from the previous section to the case where $\Sigma$ is a critical point.

In the next theorem and throughout the remainder of this paper we let $L= L_{0,1}$ be the (non-symmetric) second order operator
\begin{equation}
L = \cL+|A|^2+\frac{1}{2}=\Delta + |A|^2 - \frac{1}{2} \, \langle x , \nabla (\cdot) \rangle + \frac{1}{2}\, .
\end{equation}

  \begin{Thm}	\label{t:secvar}
 Suppose that $\Sigma$ is complete, $\partial \Sigma = \emptyset$,   $\Sigma$ has polynomial volume growth, and    $\Sigma$ is a critical point for $F_{0 ,1}$.
  If $\Sigma_s$ is a normal variation of $\Sigma$, $x_s$, $t_s$ are  variations with $x_0 = 0$ and $t_0 = 1$, and
 \begin{equation}
 	\partial_s \big|_{s=0} \, \Sigma_s = f\nn ,  \,  \partial_s  \big|_{s=0} x_s = y ,   {\text{ and }} \partial_s  \big|_{s=0} t_s = h \, ,
\end{equation}
then setting $F'' = \partial_{ss} \big|_{s=0} \, \left( F_{x_s , t_s} (\Sigma_s) \right)$ gives
 \begin{equation}	\label{e:secvar2}
	F''=    (4\pi)^{-n/2} \, \int_{\Sigma}
	 \left( - f \,     Lf
		  + 2f\,h\,H   -   h^2 \, H^2
		+ f \, \langle y , \nn \rangle  - \frac{\langle y , \nn \rangle^2}{2}
		\right)
		 \, \e^{\frac{-|x|^2}{4}}  \, d\mu  \,  .
\end{equation}
\end{Thm}

\begin{proof}
Since $\Sigma$ is a critical point for $F_{0,1}$, we have by \eqr{e:Fprime0} that
\begin{equation}	\label{e:215}
	H = \frac{ \langle x , \nn \rangle}{2} \, .
\end{equation}
 Within this proof, we will use square brackets $\left[ \cdot \right]$ to denote  weighted integrals:
 \begin{equation}
 	\left[ g \right] =  (4\pi)^{- \frac{n}{2} } \,    \int_{\Sigma}  g \,   \e^{\frac{-|x|^2}{4}} \, d \mu  \, .
 \end{equation}
It follows from Lemma \ref{l:rightsign}  that
\begin{equation}	\label{e:11012)}
	     \left[ \frac{|x|^2}{4} - \frac{n}{2} \right] = 0  {\text{ and }}
	\left[  x   \right] = \left[ x \, |x|^2 \right] = 0
	 \, .
\end{equation}
Theorem \ref{t:secvar0} (with $x_0 = 0$ and $t_0 = 1$) gives
\begin{align}	\label{e:Fprime20}
	F'' &=  \left[
		- f \, L \, f
		 + f \, h \,  \langle x , \nn \rangle  - h^2 \frac{|x|^2 - n}{2} + f \,  \langle y , \nn \rangle - \frac{|y|^2}{2} - h \,   \langle x , y \rangle
		\right]   \notag \\
		&\quad + \left[ \, \left(
		h \, \left( \frac{|x|^2}{4} - \frac{n}{2} \right) +   \frac{ 1}{2} \, \langle x  , y \rangle \right)^2
		\right]
		\, ,
\end{align}
where we used \eqr{e:215} and \eqr{e:11012)} to see that a number of the terms from Theorem \ref{t:secvar0} vanish.
Squaring out the last term and using \eqr{e:215} and   \eqr{e:11012)} gives
\begin{align}	\label{e:Fprime21}
	F'' &=  \left[
		- f \, L \, f
		 + 2 \, f \, h \,  H  -   \frac{n}{2} \, h^2 + f \,  \langle y , \nn \rangle - \frac{|y|^2}{2}
		\right]   \notag \\
		&\quad + \left[
		h^2  \left( \frac{|x|^2}{4} - \frac{n}{2} \right)^2  + \frac{1}{4} \, \langle x , y \rangle^2  \right]
		+ h \, \left[  \left( \frac{|x|^2}{4} - \frac{n}{2} \right) \, \langle x , y \rangle \right]
		\, .
\end{align}

Using the last two equalities in \eqr{e:11012)}, we see that the last term in \eqr{e:Fprime21} vanishes and thus
\begin{align}
	F'' &=     \left[ - f \,    Lf
		 + 2f\,h\,H + h^2 \,  \left( \frac{|x|^2}{4} - \frac{n}{2} \right)^2 - \frac{n \, h^2}{2}
		+ f \, \langle y , \nn \rangle  - \frac{|y|^2}{2} + \frac{ 1}{4} \, \langle x  , y \rangle^2
		\right]  \notag \\
		&= \left[ - f \,   Lf
		 + 2f\,h\,H   -   h^2 \, H^2
		+ f \, \langle y , \nn \rangle  - \frac{\left| y^{\perp} \right|^2}{2}
		\right]
		 \, ,
\end{align}
where the second equality used Corollary \ref{c:rs} and \eqr{e:mass2}.
  \end{proof}

\subsection{$F$-stable self-shrinkers}

We will say that a critical point $\Sigma$ for $F_{0,1}$ is {\emph{$F$-stable}} if for every normal variation $f \, \nn$ there exist variations of $x_0$ and $t_0$ that make  $F'' \geq 0$.

Roughly speaking, $\Sigma$ is stable when the only way to decrease the $F_{0,1}$ functional is to translate in space or time.  It is easy to see that the sphere of radius $\sqrt{2n}$ is $F$-stable in $\RR^{n+1}$:

\begin{Lem}	\label{l:sn}
The $n$-sphere of radius $\sqrt{2n}$ in $\RR^{n+1}$ is   $F$-stable.
\end{Lem}

\begin{proof}
Note that
  $x^T = 0$,  $A$ is $1/\sqrt{2n}$ times the metric, and  $L = \Delta + 1$.  Therefore, by Theorem \ref{t:secvar},
  the lemma will follow from showing that
given an arbitrary normal variation $f \nn$, there exist $h \in \RR$ and $y \in \RR^{n+1}$ so that
  \begin{equation}	\label{e:secvarsn2}
	  \int_{\SS^n}
	 \left[ - f \, \left(   \Delta f + f
		\right) + \sqrt{2n} \, f\,h    -   \frac{n}{2} \, h^2
		+ f \, \langle y , \nn \rangle  - \frac{\langle y , \nn \rangle^2}{2}
		\right]    \geq 0  \, .
\end{equation}
Recall that the eigenvalues of the Laplacian{\footnote{See, e.g., ($14$) on page $35$ of Chavel, \cite{Ca}.}} on the $n$-sphere of radius one are given by $k^2 + (n-1) \, k$ for $k= 0 , 1, \dots$ with $0$ corresponding to the constant function and the first non-zero eigenvalue $n$ corresponding to the restrictions of the linear functions in $\RR^{n+1}$.  It follows that the eigenvalues of $\Delta$ on the sphere of radius $\sqrt{2n}$ are given by
\begin{equation}
	\mu_k = \frac{k^2 + (n-1) \, k}{2n} \, ,
\end{equation}
with $\mu_0 = 0$ corresponding to the constant functions and $\mu_1 = \frac{1}{2}$ corresponding to the linear functions.  Let $E$ be the space of $W^{1,2}$ functions that are orthogonal to constants and linear functions; equivalently, $E$ is the span of all the eigenfunctions for $\mu_k$ for all $k \geq 2$.
Therefore, we can choose $a \in \RR$ and $z \in \RR^{n+1}$ so that
\begin{equation}
	f_0 \equiv f - a - \langle z , \nn \rangle  \in E \, .
\end{equation}
Using the orthogonality of the different eigenspaces, we get that
\begin{align}	\label{e:p1}
	\int_{\SS^n}  -f \, (\Delta f +  f) &\geq (\mu_2- 1) \, \int_{\SS^n} f_0^2 + (\mu_1 - 1) \, \int_{\SS^n} \langle z , \nn \rangle^2 + (\mu_0 - 1) \, \int_{\SS^n} a^2 \notag \\
	&= \frac{1}{n} \, \int_{\SS^n} f_0^2  - \frac{1}{2} \, \int_{\SS^n} \langle z , \nn \rangle^2   -  \int_{\SS^n} a^2 \, .
\end{align}
Again using the orthogonality of different eigenspaces, we get
  \begin{equation}	 	\label{e:p2}
	  \int_{\SS^n}
	 \left[ \sqrt{2n} \, f\,h
		+ f \, \langle y , \nn \rangle
		\right]    =    \int_{\SS^n}
	 \left[ \sqrt{2n} \, a \,h
		+  \langle z , \nn \rangle \, \langle y , \nn \rangle
		\right]
		  \, .
\end{equation}
Combining \eqr{e:p1} and \eqr{e:p2}, we get that the left hand side of \eqr{e:secvarsn2} is greater than or equal to
  \begin{align}	 	\label{e:p3}
	 & \int_{\SS^n}  \left[ \frac{f_0^2 }{n}  - \frac{1}{2}   \langle z , \nn \rangle^2   -   a^2+
	  \sqrt{2n} \, a \,h    -   \frac{n}{2} \, h^2
		+  \langle z , \nn \rangle \, \langle y , \nn \rangle   - \frac{\langle y , \nn \rangle^2}{2}
		\right]   \notag \\
		& \quad =
		 \int_{\SS^n}  \left[ \frac{f_0^2 }{n}  -
		 \frac{1}{2} \left(  \langle z , \nn \rangle  - \langle y , \nn \rangle \right)^2   -
		 \left( a - \frac{\sqrt{n} \, h }{\sqrt{2}} \right)^2 		\right]
		  \, .
\end{align}
This can be made non-negative by choosing $y = z$ and $h = \frac{\sqrt{2} \, a }{\sqrt{n}} $.

\end{proof}

  We will prove the following converse:

\begin{Thm}		\label{t:linearstable}
The $n$-sphere of radius $\sqrt{2n}$ is the only (smooth) embedded closed $F$-stable hypersurface in  $\RR^{n+1}$ for any $n \geq 2$.
\end{Thm}

In the non-compact case, we show:

\begin{Thm}		\label{t:linearstable2}
The $n$-plane is the only complete (smooth) embedded non-compact $F$-stable hypersurface in
 $\RR^{n+1}$  with no boundary
  and   polynomial volume growth.
  \end{Thm}

We will prove Theorem \ref{t:linearstable} in Section \ref{s:sect6} and then prove Theorem \ref{t:linearstable2} in Section \ref{s:noncompactspec}.
 Theorem \ref{t:liketo} from the introduction follows from combining Theorems \ref{t:linearstable} and \ref{t:linearstable2}.

\section{Eigenfunctions and eigenvalues of $L$}	\label{s:three}

Throughout this section, $\Sigma \subset \RR^{n+1}$ is a smooth hypersurface without boundary, $H = \frac{\langle x , \nn \rangle}{2}$ and
\begin{equation}
	L = \Delta + |A|^2 + \frac{1}{2} - \frac{1}{2} \, \langle x , \nabla (\cdot) \rangle
	\end{equation}
	 is as in Theorem \ref{t:secvar}.

In this section we will show that $H$ and the translations are eigenfunctions of $L$.  We will see that the eigenvalue corresponding to $H$ is $-1$, whereas the eigenvalue corresponding to the translations (in space) is $-\frac{1}{2}$.  We will also see that in a weighted space these eigenfunctions are orthogonal.    In fact, orthogonality is a direct consequence of that they are eigenfunctions corresponding to different eigenvalues and that in a weighted space the operator $L$ is self-adjoint.  Namely, we show the following theorem:

\begin{Thm}	\label{t:spectral}
The mean curvature $H$ and the normal part $\langle v , \nn \rangle$ of a constant vector field $v $  are eigenfunctions of $L$ with
\begin{equation}	\label{e:spec1}
	LH =  H  {\text{ and }} L \langle v , \nn \rangle =  \frac{1}{2} \,  \langle v , \nn \rangle \, .
\end{equation}
Moreover, if $\Sigma$ is compact, then $L$ is self-adjoint in the weighted space and
\begin{equation}	\label{e:selfad}
	- \int_{\Sigma} (u_1 \, L u_2) \, \e^{\frac{-|x|^2}{4}}   =  \int_{\Sigma} \left( \langle \nabla u_1 , \nabla u_2 \rangle - |A|^2 \, u_1 \, u_2 - \frac{1}{2} \, u_1 \, u_2  \right) \,
		\e^{\frac{-|x|^2}{4}}  \, .
\end{equation}
\end{Thm}

Once we have this, then we can use a theorem of Huisken together with the usual characterization
 of the first eigenvalue to conclude that if $\Sigma$ is a closed $F$-stable self-shrinker, then $\Sigma$ must be  a sphere.

\subsection{$H$ and $\langle v , \nn \rangle$ are eigenfunctions of $L$}

\begin{Lem}		\label{l:H}
$L \, H = H$ and $L \, \langle v , \nn \rangle =  \frac{1}{2} \,  \langle v , \nn \rangle$ for
any constant vector field $v$.
\end{Lem}

\begin{proof}
Letting $e_i$ be an orthonormal frame for $\Sigma$ and
differentiating $H = \frac{1}{2} \, \langle x , \nn \rangle$ as in theorem $4.1$ of \cite{H3} (which has a different normalization) gives
\begin{equation}	\label{e:jc2}
	2 \, \nabla_{e_i} \, H= \langle e_i , \nn \rangle  + \langle x , \nabla_{e_i} \nn \rangle = - a_{ij} \, \langle x , e_j \rangle \, ,
\end{equation}
where the first equality used that $\nabla_e x = e$ for any vector $e$ and the second equality used \eqr{e:aij}.
Working at a point $p$ and choosing the frame $e_i$ so that $\nabla_{e_i}^T e_j (p) = 0$, differentiating again gives at $p$ that
\begin{equation}	\label{e:jc3}
	2 \, \nabla_{e_k} \, \nabla_{e_i} \, H=
	 - a_{ij,k} \, \langle x , e_j \rangle
	  - a_{ij} \, \langle e_k , e_j \rangle
	 - a_{ij} \, \langle x , \nabla_{e_k} e_j \rangle   \, .
\end{equation}
Using that $\nabla_{e_i}^T e_j = 0$ at $p$, the last term in \eqr{e:jc3} can be rewritten as
\begin{equation}	\label{e:forlast}
	 - a_{ij} \, \langle x , \nabla_{e_k} e_j \rangle=
	  -a_{ij} \, \langle x,  a_{kj} \, \nn \rangle
	 =-2 \, Ha_{ij}a_{kj}
	 \, .
\end{equation}
Taking the trace of \eqr{e:jc3}, using Codazzi's equation (i.e., $a_{ij,i} = a_{ii,j}= -H_j$; see, e.g., (B.5) in \cite{Si}), and substituting \eqr{e:forlast}
gives at $p$ that
\begin{equation}	\label{e:jc4a}
	  2 \, \Delta \, H  = \langle x , \nabla H \rangle + H - 2 \, |A|^2 \, H
	    \, ,
\end{equation}
and, thus,
\begin{equation}	\label{e:jc4}
	 L \, H \equiv \Delta \, H + |A|^2 \, H + \frac{1}{2} \, H - \frac{1}{2}
	  \langle x , \nabla H \rangle   = H
	    \, .
\end{equation}
Since $p$ is arbitrary and \eqr{e:jc4} is independent of the frame, we conclude that $L \, H = H$.

Fix a constant vector   $v \in \RR^{n+1}$ and set $f= \langle v , \nn \rangle$.  We have
\begin{equation}	\label{e:miracle}
	\nabla_{e_i} f = \langle v , \nabla_{e_i} \nn \rangle = - a_{ij} \, \langle v , e_j \rangle \, .
\end{equation}
Working again at a fixed point $p$ and choosing the frame as above,
differentiating again and using the Codazzi equation gives at $p$ that
\begin{equation}
	\nabla_{e_k} \nabla_{e_i} f =  - a_{ik,j} \, \langle v , e_j \rangle
	- a_{ij} \, \langle v , a_{jk} \, \nn \rangle\, .
\end{equation}
Taking the trace gives
\begin{equation}	\label{e:notyet}
	\Delta f =   \langle v , \nabla H \rangle
	- |A|^2 \, f \, .
\end{equation}
(Note that we have not used the self-shrinker equation for $H$ in \eqr{e:notyet}, but we will use it next.)  From \eqr{e:jc2}, we know that
\begin{equation}
	\langle v , \nabla H \rangle = - \frac{1}{2} \, \langle x , e_j \rangle \,  a_{ij} \, \langle v , e_i  \rangle
	= \frac{1}{2} \langle x  , \nabla f \rangle \, ,
\end{equation}
where the last equality used \eqr{e:miracle}.  Substituting this back into \eqr{e:notyet}, we finally get that $L \, f =
\Delta \, f +   |A|^2 \, f + \frac{1}{2}   \, f - \frac{1}{2} \, \langle x , \nabla f \rangle =
\frac{1}{2} \, f $.
\end{proof}

\begin{proof}
(of Theorem \ref{t:spectral}).
The two claims in \eqr{e:spec1} are proven in Lemma \ref{l:H}.  Self-adjointness of $L$ in the weighted space follows immediately from
Lemma \ref{l:self}.
\end{proof}

\subsection{The eigenvalues of $L$}

Recall that we will say that $\mu \in \RR$ is an eigenvalue of
$L$ if there is a (not identically zero) function $u$ with
$L \, u = - \mu \, u $.
 The operator $L$
 is not self-adjoint with respect to the $L^2$-inner product because of the first order term,
 but it is self-adjoint with respect to the weighted inner product by Theorem \ref{t:spectral}.    In particular, standard spectral theory{\footnote{These facts are proven in a similar setting (Dirichlet eigenvalues for divergence operators on smooth, bounded open sets in $\RR^n$) in theorems $1$ and $2$ on pages $335$ and $336$, respectively, in Evans, \cite{Ev}.  The proof uses only linear elliptic theory and self-adjointness and carries over to this setting with standard modifications.}} gives the following corollary:

\begin{Cor}	\label{t:evans}
If $\Sigma \subset \RR^{n+1}$ is a smooth closed hypersurface, then
\begin{enumerate}
\item $L$ has real eigenvalues $\mu_1 < \mu_2 \leq \dots$ with $\mu_k \to \infty$.
\item There is an orthonormal basis $\{ u_k \}$ for the weighted $L^2$ space with $L \, u_k = - \mu_k \, u_k$.
\item The lowest eigenvalue $\mu_1$ is characterized by
\begin{equation}	\label{e:lam0}
	\mu_1 = \, \inf_{f} \,  \, \frac{    \int_{\Sigma} \left(  |\nabla f |^2 - |A|^2 \, f^2
	- \frac{1}{2} \, f^2 \right) \, \e^{- \frac{|x|^2}{4} } }{   \int_{\Sigma} f^2 \, \e^{- \frac{|x|^2}{4} } } \, ,
\end{equation}
where the infimum is taken over smooth functions $f$ (that are not identically zero).
\item Any eigenfunction for $\mu_1$ does not change sign and, consequently, if $u$ is any weak solution of $L u = - \mu_1 u$ then $u$ is a constant multiple of $u_1$.
\end{enumerate}
\end{Cor}

\begin{Rem}
There is a corresponding result when $\Sigma$ is not closed.  In this case, we choose a
smooth  open subset   $\Omega \subset \Sigma$  with compact closure and restrict to functions with compact support in $\Omega$.  In this case, (1) and (2) are unchanged, the infimum in (3) is over smooth functions with compact support in $\Omega$, and (4) is for Dirichlet eigenfunctions (i.e., ones that vanish on $\partial \Omega$).
\end{Rem}

\section{Spheres are the only closed $F$-stable self-shrinkers}	\label{s:sect6}

We now have all of the tools that we need to prove that spheres are the only smooth, embedded closed self-shrinkers that are  $F$-stable (this is Theorem \ref{t:linearstable}).

\begin{proof}
(of Theorem \ref{t:linearstable}.)
Since $\Sigma$ is closed, the mean curvature $H$ cannot vanish identically.
By Huisken's classification, \cite{H3},
it suffices to show that $H$ does not change sign (by the Harnack inequality, it is then positive everywhere).
We will argue by contradiction, so suppose  that
 $H$ changes sign.

 In order to show that $\Sigma$ is unstable, we must find a function $f$ on $\Sigma$ so that $F'' < 0$ no matter what values of $h$ and $y$ that we choose.  Given a variation $f \nn$ of $\Sigma$, a variation $y$ of $x_0$, and a variation $h$ of $t_0$,
  Theorem \ref{t:secvar} gives
  \begin{equation}	\label{e:secvar2a}
	F''=    (4\pi)^{-n/2} \, \int_{\Sigma}
	 \left[ - f \,    Lf
		 + 2f\,h\,H   -   h^2 \, H^2
		+ f \, \langle y , \nn \rangle  - \frac{\langle y , \nn \rangle^2}{2}
		\right]
		 \, \e^{\frac{-|x|^2}{4}}   \, .
\end{equation}

 We know from Theorem \ref{t:spectral} that $H$ is an eigenfunction for $L$ with eigenvalue $-1$.  However, since $H$ changes sign,  (4) in  Theorem \ref{t:evans} implies that $-1$ is not the smallest eigenvalue for $L$.  Thus, we get a positive function $f $ with
$Lf = \mu \, f$ where  $\mu > 1$.
Since $L$ is self-adjoint in the weighted $L^2$ space by \eqr{e:selfad}, we see that $f$
must be orthogonal to the eigenspaces with different eigenvalues (this is (2) in Theorem \ref{t:evans}).
Therefore, by Theorem \ref{t:spectral}, $f$ is orthogonal to $H$ and the translations, i.e., for any
 $y \in \RR^{n+1}$ we have
\begin{equation}
	0 = \int_{\Sigma} f \, H \, \e^{\frac{-|x|^2}{4}} =  \int_{\Sigma} f \, \langle y , \nn \rangle \, \e^{\frac{-|x|^2}{4}} \, .
\end{equation}
Substituting this into \eqr{e:secvar2a} gives
  \begin{equation}	\label{e:secvar2aa}
	 (4\pi)^{n/2} \, F''=    \int_{\Sigma}
	 \left[ - f \,   \mu \, f
		     -   h^2 \, H^2
		  - \frac{\langle y , \nn \rangle^2}{2}
		\right]
		 \, \e^{\frac{-|x|^2}{4}}  \leq -  \mu \, \int_{\Sigma} f^2 \, \e^{\frac{-|x|^2}{4}}
		  \, .
\end{equation}
Since $\mu > 1$ is positive, it follows that $F'' < 0$ for every choice of $h$ and $y$ as desired.
\end{proof}

\section{Entropy}

 The entropy $\lambda (\Sigma)$  of a hypersurface $\Sigma$ is defined to be
 \begin{equation}
 	\lambda (\Sigma) = \sup_{x_0 \in \RR^{n+1} , t_0 > 0} \, \, F_{x_0 , t_0} (\Sigma) \, .
 \end{equation}
Even though the supremum is over non-compact  space-time,
  we will see that this supremum is achieved on two important classes of hypersurfaces: compact hypersurfaces with $\lambda>1$ and self-shrinkers with polynomial volume growth.

\subsection{The entropy is achieved for compact hypersurfaces with $\lambda > 1$}

We will need a few simple properties of the $F$ functionals that we summarize in the following lemma:

\begin{Lem}	\label{l:basicF}
If $\Sigma \subset \RR^{n+1}$ is a smooth complete embedded hypersurface without boundary and with polynomial volume growth, then
\begin{enumerate}
\item $F_{x_0,t_0} (\Sigma) $ is a smooth function of $x_0$ and $t_0$ on $\RR^{n+1} \times (0,\infty)$.

\item  Given any $t_0 > 0$ and any  $x_0$, we have
$
	\partial_{t_0} F_{x_0,t_0} (\Sigma)
\geq
 -    \frac{ \lambda (\Sigma) }{4} \,   \sup_{\Sigma} H^2$.
 \item For each $x_0$,   $\lim_{t_0 \to 0} F_{x_0,t_0} (\Sigma)$ is  $1$ if $x_0 \in \Sigma$ and is $0$ otherwise.
\item If $\Sigma$ is closed, then $\lambda (\Sigma) < \infty$.
\end{enumerate}
 \end{Lem}

 \begin{proof}
 Property (1) follows immediately from differentiating under the integral sign and the polynomial volume growth.

  Since the hypotheses and conclusion in (2)   are invariant under translation, it suffices to do the case where $x_0 = 0$.
The first variation formula
(Lemma \ref{l:varl0}) gives
\begin{equation}
	\partial_{t_0} F_{0,t_0} (\Sigma) = (4\pi t_0)^{ - \frac{n}{2} } \,
	\int_{\Sigma}  \frac{|x|^2 - 2nt_0}{4 t_0^2} \, \e^{-\frac{|x|^2}{4t_0}} \, .
\end{equation}
Since $\Delta |x|^2 = 2n - 2 \, H \, \langle x , \nn \rangle$ and $\Delta \, \e^f = \e^f \, \left( \Delta f + |\nabla f|^2 \right)$, we have
\begin{align}
	 \e^{\frac{|x|^2}{4t_0}} \, \, \Delta \, \e^{-\frac{|x|^2}{4t_0}} &=
	 \frac{|x^T|^2}{4 \, t_0^2} - \frac{2n}{4\, t_0} + \frac{H \, \langle x , \nn \rangle}{2 \, t_0}
	 =   \frac{|x|^2- 2n\, t_0}{4 \, t_0^2} -  \frac{|x^{\perp}|^2}{4\,  t_0^2}  + \frac{H \, \langle x , \nn \rangle}{2t_0}  \notag \\
	 &\leq \frac{|x|^2- 2n\, t_0}{4\,  t_0^2}  + \frac{H^2}{4} \, ,
	\end{align}
where the inequality used $2ab \leq a^2 + b^2$.  Since $\Sigma$ has polynomial volume growth and
the vector field $\nabla \e^{-\frac{|x|^2}{4t_0}}$ decays exponentially,    Stokes' theorem gives
\begin{equation}	\label{e:from77}
	\partial_{t_0} F_{0,t_0} (\Sigma) \geq -  (4\pi t_0)^{ - \frac{n}{2} } \,
	\int_{\Sigma}  \frac{H^2}{4} \, \e^{-\frac{|x|^2}{4t_0}}  \geq - \frac{1}{4} \, F_{x_0 , t_0} (\Sigma)
		 \, \sup_{\Sigma} H^2 \geq  - \frac{\lambda (\Sigma) }{4} \, \sup_{\Sigma} H^2
	\, ,
\end{equation}
where the last inequality used that, by definition, $\lambda (\Sigma) \geq F_{x_0 , t_0} (\Sigma)$.

Property (3) is a standard consequence of the fact that smooth hypersurfaces are approximated by a hyperplane on small scales (the function $(4\pi t_0)^{ - \frac{n}{2} } \,
	  \e^{-\frac{|x-x_0|^2}{4t_0}} $ is a heat kernel on a hyperplane through $x_0$ and has integral one on the hyperplane independent of $t_0$).
	
	  To see property (4), observe that since $\Sigma$ is smooth and closed there is a constant $V>0$ so that for every $x_0$ and for every radius $R>0$
 \begin{equation}	\label{e:Vvol}
 	\Vol \, (B_R(x_0) \cap \Sigma) \leq V \, R^n \, .
\end{equation}
Property (4)  follows easily from this.

 \end{proof}

  \begin{Lem}	\label{l:cpt}
If $\Sigma \subset \RR^{n+1}$ is a smooth closed embedded hypersurface and $\lambda (\Sigma) > 1$, then there exist $ x_0 \in \RR^{n+1}$ and $t_0 > 0$ so that $\lambda = F_{x_0 , t_0} (\Sigma)$.
\end{Lem}

\begin{proof}
 For each fixed $t_0 > 0$, it is easy to see that $\lim_{|x_0| \to \infty} \, F_{x_0,t_0} (\Sigma) = 0$ by the exponential decay of the weight function together with the compactness of $\Sigma$.    In particular, for each fixed $t_0 > 0$, the maximum of   $F_{x_0,t_0} (\Sigma)$ is achieved at some $x_0$.  Moreover, the first variation formula (Lemma \ref{l:varl0}) shows that this maximum occurs when the weighted integral of $(x-x_0)$ vanishes, but this can only occur when $x_0$ lies in the convex hull of $\Sigma$.   It remains to take the supremum of these maxima as we vary $t_0$, i.e., we must show that the $F$ functionals are strictly less than $\lambda (\Sigma) $ when $t_0$ goes to either zero or infinity.

It is easy to see that  $F_{x_0,t_0} (\Sigma)$ goes to zero (uniformly in $x_0$) as $t_0 \to \infty$ since
\begin{equation}
	F_{x_0 , t_0} (\Sigma) \leq (4\pi \, t_0)^{ - \frac{n}{2} } \, \Vol (\Sigma) \, .
\end{equation}
Fix an $\epsilon > 0$ with $1 + 3 \, \epsilon < \lambda (\Sigma)$.
 Given any $x$, property (3) in Lemma \ref{l:basicF} gives a $t_x > 0$ so that
 $F_{x,t_0} (\Sigma) < 1 + \epsilon$ for all $t_0 \leq t_x$.
 Using the continuity (in $x$) of $F_{x,t_x} (\Sigma)$ and the compactness of $\Sigma$, we can choose a finite collection of points $x_i$, positive numbers $t_i < \frac{\epsilon}{C_2}$ (with $C_2$ from (2) in Lemma \ref{l:basicF}), and radii $r_i > 0$ so that
 \begin{itemize}
 \item $ \Sigma \subset \cup_i B_{r_i} (x_i)$.
 \item For every $x \in B_{2r_i}(x_i)$, we have  $F_{x,t_i} (\Sigma) < 1 + \epsilon$.
 \end{itemize}
 If we let $\bar{r}$ be the minimum of the $r_i$'s and let $\bar{t}$ be the minimum of the $t_i$'s, then property (2) in Lemma \ref{l:basicF} gives   for any $x_0$ in the $\bar{r}$ tubular neighborhood of $\Sigma$ and any $t_0 \leq \bar{t}$ that
 \begin{equation}
 	F_{x_0 , t_0} (\Sigma) \leq  F_{x_0 , t_{i(x_0)}} (\Sigma) + C_2 \, (t_{i(x_0)} - t_0) \leq
	1+ 2 \epsilon  \, ,
 \end{equation}
 where $i(x_0)$ satisfies $x_0 \subset B_{2r_{i(x_0)}} (x_{i(x_0)})$.  It follows that $F_{x_0 , t_0} (\Sigma)$ is strictly less than $\lambda$ for $t_0$ sufficiently small and the lemma follows.
\end{proof}

\subsection{The entropy is achieved for  self-shrinkers with polynomial volume growth}

The next lemma shows that, for a self-shrinker $\Sigma$ with polynomial volume growth, the function $(x_0,t_0) \to F_{x_0,t_0} (\Sigma)$ has  a strict (global) maximum at $x_0 =0 , t_0 = 1$, unless  $\Sigma$ splits off a line isometrically.  Most of the proof is concerned with handling the case where $\Sigma$ is non-compact; when $\Sigma$ is closed, the lemma is an easy consequence of Huisken's monotonicity for the associated self-similar flow (cf. section $8$ of \cite{W4}).

\begin{Lem}	\label{l:strict}
Suppose that $\Sigma$ is a smooth complete embedded self-shrinker with $\partial \Sigma = \emptyset$,   polynomial volume growth, and $\Sigma$  does not split off a line  isometrically.  Given $\epsilon > 0$, there exists $\delta > 0$ so
\begin{equation}
	\sup \, \{ F_{x_0 , t_0}(\Sigma) \, | \, |x_0| + |\log t_0| > \epsilon \} \, < \, \lambda - \delta \, .
\end{equation}
\end{Lem}

\begin{proof}
We will prove first that the function $(x_0,t_0) \to F_{x_0 , t_0}(\Sigma)$ has a strict local maximum at $(0,1)$.  We do this by  showing that $(0,1)$ is a critical point (in fact, the only critical point) and then showing that the second derivative at $(0,1)$ is strictly negative.
The bulk of the proof is then devoted to showing that the function is decreasing along a family of paths that start at $(0,1)$ and run off to infinity (or time goes to zero) and whose union is all of space-time.

Since $\Sigma$ is fixed and has polynomial volume growth, we may think of
  $F_{x_0 , t_0}(\Sigma)$ as a smooth function of $x_0$ and $t_0$.  Since $\Sigma$ is a self-shrinker, it follows from Proposition
  \ref{p:critall} that the $\RR^{n+1} \times \RR$-gradient of this function vanishes at $x_0 =0$ and $t_0 =1$.

  The second variation formula (Theorem \ref{t:secvar}) with $\Sigma_s \equiv \Sigma$ gives that the second derivative of $F_{x_s , t_s}(\Sigma)$  at $s=0$ along the path $x_s = s\, y, \, t_s = 1 + s \, h$ is
 \begin{equation}	\label{e:secvar27}
	  - (4\pi)^{-n/2} \, \int_{\Sigma}
	 \left(  h^2 \, H^2
		+   \frac{\left| y^{\perp} \right|^2}{2}
		\right)
		 \, \e^{\frac{-|x|^2}{4}}  \, d\mu  \,  .
\end{equation}
This expression is clearly non-positive.  Moreover, the second term vanishes only when $y$ is everywhere tangent to $\Sigma$; this can happen only when $y=0$ since $\Sigma$ does not split off a line.  Likewise, the first term can vanish only when $h=0$ or when $\Sigma$ is a minimal cone; the latter does not occur here since hyperplanes are the only smooth minimal cones.
  We conclude that $F_{x_0 , t_0}(\Sigma)$ has a strict local maximum at $x_0 = 0$ and $t_0 = 1$.
Thus, the lemma will follow from the following claim:

{\bf{Claim}}: If we
 fix $y \in \RR^{n+1}$  and $a \in \RR$  and then define   $g(s)$   by
\begin{equation}
	g(s) = F_{s y , 1 + a \, s^2} (\Sigma) \, ,
\end{equation}
then $g'(s)  \leq  0$ for all $s > 0$ with $1+as^2 > 0$ (the second condition is automatic for $a\geq0$).

The rest of the proof will be devoted to proving the claim.

\vskip2mm
{\bf{Preliminaries: weighted spaces and self-adjoint operators}}:
We will use square brackets $\left[ \cdot \right]_s$ to denote the weighted integral
 $\left[ f \right]_s = (4\pi \, t_s)^{ - \frac{n}{2} } \, \int_{\Sigma} f \, \e^{- \frac{|x-x_s|^2}{4\, t_s} } \, d\mu$
with $x_s  = sy$ and $t_s = 1 + a \, s^2 $ and define the operator $\cL_s$
 \begin{equation}	\label{e:divcLs}
       \cL_s v = \Delta v  -  \langle \frac{x - x_s}{2t_s} \, , \nabla v \rangle
	= \e^{\frac{|x-x_s|^2}{4t_s}} \,   \dv \, \left( \e^{\frac{-|x-x_s|^2}{4t_s}} \,  \nabla v \right) \, .
\end{equation}
 A slight variation of the proof of Corollary \ref{c:self} gives that
\begin{equation}	\label{e:varcself}
	- \left[ u \, \cL_s v \right]_s = \left[ \langle \nabla u , \nabla v \rangle \right]_s
\end{equation}
whenever $u$ and $v$ are $C^2$ functions with
$\left[ |u \, \nabla v| + |\nabla u| \, |\nabla v| + |u \, \cL_s v | \right]_s < \infty$.

 To apply this, we will compute $\cL_s$ on $\langle x   , y \rangle$  and  $|x|^2$   (cf. Lemma
\ref{l:opquant} where we did this for $s=0$).   First, since $\nabla \langle x   , y \rangle = y^T$ and $\Delta \langle x , y \rangle = - H \,
\langle \nn , y \rangle=
- \frac{ \langle x , \nn \rangle}{2} \, \langle \nn , y \rangle = - \frac{ \langle x , y^{\perp} \rangle}{2} $, we get
\begin{equation}	\label{e:cls1}
	2 \, \cL_s \, \langle x  , y \rangle  = - \langle x , y^{\perp} \rangle   - \langle \frac{x-x_s}{t_s} , y^T \rangle
	\, .
\end{equation}
Similarly, using   $\nabla  \, \left| x \right|^2 = 2 x^T$ and  $\Delta |x|^2 = 2n -  \left| x^{\perp} \right|^2$ by
 \eqr{e:deltax2}  gives
 \begin{equation}    \label{e:deltax02}
   \cL_s \, \left| x \right|^2  = 2n - \left| x^{\perp} \right|^2  - \langle \frac{ x-x_s}{t_s} , x^T \rangle    \, .
\end{equation}

It follows from the polynomial volume growth and \eqr {e:cls1} that we can
apply \eqr{e:varcself} with $u=1$ and $v = \langle x  , y \rangle$ to get
\begin{equation}    \label{e:secli3}
	\left[   \langle (x - x_s)^T , y^T \rangle  \right]_s =  - t_s \, \left[
	\langle x^{\perp} , y^{\perp} \rangle \right]_s
\, .
\end{equation}
Similarly, the polynomial volume growth and \eqr{e:deltax02} allow us to apply  \eqr{e:varcself} with $u=1$ and $v = |x|^2$ to get
  \begin{equation}   \label{e:firsttt}
   \left[  \langle (x -x_s)^T , x^T \rangle \right]_s =  t_s \, \left[ 2n - \left| x^{\perp} \right|^2
   \right]_s
  \, .
  \end{equation}

{\bf{The derivative of $g$}}:
The first variation formula (Lemma \ref{l:varl0}) gives
\begin{equation}    \label{e:gps}
	g'(s) = \left[  t_s'  \, \left( \frac{|x - x_s|^2}{4t_s^2} - \frac{n}{2t_s} \right) +
 \frac{\langle x-x_s , y \rangle}{2t_s} \right]_s \, .
\end{equation}
It is useful to set $z= x-x_s$ so that $x = z + sy$.  Using     \eqr{e:secli3}, we   compute the weighted $L^2$ inner product of
$y^T$ and $z^T$:
\begin{equation}	\label{e:zyT}
	\left[ \frac{\langle z^T , y^T \rangle}{t_s} \right]_s =
	-  \left[
	\langle (z+ sy)^{\perp} , y^{\perp} \rangle \right]_s
	=
	 - \left[ \langle z^{\perp} , y^{\perp} \rangle + s \, \left| y^{\perp} \right|^2 \right]_s \, ,
\end{equation}
Similarly,  we can rewrite   \eqr{e:firsttt} as
\begin{equation}
	 \left[ \frac{\left| z^T \right|^2}{t_s} + s\,  \frac{\langle z^T ,  y^T \rangle}{t_s} \right]_s =   \left[ 2n - \left| z^{\perp} \right|^2
	 - 2 \, s \, \langle  z^{\perp} , y^{\perp} \rangle
	 - s^2 \left|  y^{\perp} \right|^2
   \right]_s   \, .
\end{equation}
Using \eqr{e:zyT} to evaluate the $\langle z^T ,  y^T \rangle$ term on the left gives
\begin{align}	\label{e:zT2}
	\left[ \frac{ \left| z^T  \right|^2}{t_s}  \right]_s &=   s \, \left[ \langle z^{\perp} , y^{\perp} \rangle + s \, \left| y^{\perp} \right|^2 \right]_s
	+  \left[ 2n - \left| z^{\perp} \right|^2
	 - 2 \, s \, \langle  z^{\perp} , y^{\perp} \rangle
	 - s^2 \left|  y^{\perp} \right|^2
   \right]_s
	\notag \\
	&= \left[ 2n - \left| z^{\perp} \right|^2 - s \, \langle y^{\perp} , z^{\perp} \rangle \right]_s \, .
\end{align}
Using \eqr{e:zT2},
the first  term in \eqr{e:gps} can be written as
\begin{align}    \label{e:gps2t}
 t_s' \left[     \frac{|z|^2}{4t_s^2} - \frac{n}{2t_s}  \right]_s & = \frac{ t_s' }{4\, t_s} \,
 \left[    \frac{|z^{T}|^2}{t_s} +  \frac{|z^{\perp}|^2}{t_s } -  2n  \right]_s\,   \notag \\
 & = \frac{ t_s' }{4\, t_s} \,
 \left[    \left( \frac{1}{t_s} - 1 \right) \,   |z^{\perp}|^2 -  s\, \langle y^{\perp} , z^{\perp} \rangle  \right]_s \, .
 \end{align}
 Similarly, using  \eqr{e:zyT}, we can rewrite the second term in \eqr{e:gps} as
 \begin{equation}    \label{e:gps2ta}
	  \left[
 \frac{\langle z , y \rangle}{2t_s} \right]_s  =  \frac{1}{2} \, \left[
 \frac{\langle z^T , y^T \rangle}{ t_s} +
 \frac{\langle z^{\perp} , y^{\perp} \rangle}{ t_s} \right]_s
 =    \frac{1}{2} \, \left[  \left( \frac{1}{t_s} - 1 \right) \, \langle z^{\perp} , y^{\perp} \rangle
 - s \, \left| y^{\perp} \right|^2 \right]_s \, .
 \end{equation}
 Since we set $t_s = 1 + a s^2$,  we have $t_s' = 2 a \, s$ and
\begin{equation}
	-\frac{t_s'}{4t_s} \, s =  \frac{1}{2} \, \left( \frac{1}{t_s} - 1 \right) = \frac{-2a \, s^2 }{4t_s}\, ,
\end{equation}
so the cross-terms (i.e., the $\langle z^{\perp} , y^{\perp} \rangle$ terms) are equal in
\eqr{e:gps2t} and \eqr{e:gps2ta}.
This gives
\begin{align}    \label{e:gpsagain}
	4 \, t_s \, g'(s) &= t_s' \,  \left( \frac{1}{t_s} - 1 \right) \,  \left[   |z^{\perp}|^2    \right]_s -
	4\,  a \, s^2 \left[ \langle z^{\perp} , y^{\perp} \rangle \right]_s -
	2 \, s \, t_s \, \left[ \left| y^{\perp} \right|^2 \right]_s  \notag \\
	&=  - \frac{2\, s}{t_s} \,
	 \left[  a^2 \, s^2 \,    |z^{\perp}|^2    +
	2\,  a \, s \, t_s \,    \langle z^{\perp} , y^{\perp} \rangle  +
	 t_s^2 \,   \left| y^{\perp} \right|^2 \right]_s    \\
	 &= - \frac{2\, s}{t_s} \,
	 \left[  \left|     a  \, s  \,    z^{\perp} +
	 t_s y^{\perp} \right|^2 \right]_s
	\, .  \notag
\end{align}
This is clearly non-positive for all   $a \in \RR$, $y \in \RR^{n+1}$, and $s > 0$ with $t_s >0$.
\end{proof}

\subsection{The equivalence of $F$-stability and entropy-stability}	\label{s:deform}

We can now prove
 that $F$-stability and entropy-stability are equivalent   for self-shrinkers that do not split off a line isometrically.
Within the proof, we will need the following elementary lemma:

\begin{Lem}		\label{l:wtHL}
There is a constant $C_n$ depending only on $n$ so that if
 $\Sigma \subset \RR^{n+1}$ is a complete hypersurface, then for  any $x_0 \in \RR^{n+1}$ and $t_0 > 0$
\begin{equation}	\label{e:tzf1}
	    (4\pi t_0)^{ - \frac{n}{2} } \,
	\int_{\Sigma}   |x|^2  \, \e^{-\frac{|x - x_0|^2}{4t_0}}  \leq 2 \,
	   \lambda (\Sigma) \, \left(  C_n \, t_0 + |x_0|^2 \right)  \, .
\end{equation}
\end{Lem}

\begin{proof}
We can assume that $\lambda (\Sigma ) < \infty$ since there is otherwise nothing to prove.
The translation and dilation $x \to y = \frac{x-x_0}{\sqrt{t_0}}$ takes $\Sigma$ to   $\tilde{\Sigma}$ with $d \mu_{\Sigma} = t_0^{ \frac{n}{2}}  \, d \mu_{\tilde{\Sigma}}$ so that
\begin{align}	\label{e:tzf1b}
	          (4\pi t_0)^{ - \frac{n}{2} } \,
	\int_{\Sigma}   |x|^2  \, \e^{-\frac{|x - x_0|^2}{4t_0}} \, d \mu_{\Sigma}  &  =
	    (4\pi )^{ - \frac{n}{2} }  \, \int_{\tilde{\Sigma}}   \left|\sqrt{t_0} \,  y + x_0
	    \right|^2  \, \e^{-\frac{|y|^2}{4}} \,   d \mu_{\tilde{\Sigma}} \notag \\
	&\leq 2 \, |x_0|^2 \, \lambda (\tilde{\Sigma}) + 2 \, t_0 \,     (4\pi )^{ - \frac{n}{2} }  \,  \int_{\tilde{\Sigma}}    |y|^2   \, \e^{-\frac{|y|^2}{4}}    \, d \mu_{\tilde{\Sigma}}
	 \, .
\end{align}
To bound the second term by the entropy, note that
\begin{equation}	\label{e:volentb}
	\left( 4 \pi \, R^2 \right)^{- \frac{n}{2}} \, \e^{ - \frac{1}{4} } \, \Vol \, (B_R \cap \tilde{\Sigma})
	\leq  F_{0, R^2} (\tilde{\Sigma}) \leq \lambda (\tilde{\Sigma}) \, .
\end{equation}
Using this volume bound and chopping the integral up into annuli $B_{j+1} \setminus B_j$ gives
\begin{equation}	\label{e:got2t}
	 (4\pi )^{ - \frac{n}{2} }  \, \int_{\tilde{\Sigma}}    |y|^2   \, \e^{-\frac{|y|^2}{4}}  \, d \mu_{\tilde{\Sigma}} \leq   \e^{ \frac{1}{4} } \, \lambda (\tilde{\Sigma}) \,
	\sum_{j=0}^{\infty}
	\left(   \e^{ - \frac{j^2}{4} } \, (j+1)^{n+2} \right) = C_n \,    \lambda (\tilde{\Sigma}) \,  ,
\end{equation}
where $C_n$ depends only on $n$.
The lemma follows from these bounds and the invariance of the entropy under dilations and translations that gives
$\lambda (\Sigma) = \lambda (\tilde{\Sigma})$.
\end{proof}

\begin{proof}
(of Theorem \ref{t:nonlin1c}).
   Assume  that $\Sigma$ is not $F$-stable and, thus, there is a
  one-parameter normal variation $\Sigma_s$  for $s \in
[-2\, \epsilon ,   2 \, \epsilon ]$ with $\Sigma_0 = \Sigma$   so that:
\begin{enumerate}
\item[(V1)] For each $s$, the variation vector field is given by a function $f_{\Sigma_s}$ times the normal $\nn_{\Sigma_s}$ where every $f_{\Sigma_s}$ is supported in a fixed compact subset of $\RR^{n+1}$.
\item[(V2)] For any variations $x_s$ and $t_s$ with $x_0 = 0$ and $t_0 = 1$, we get that
\begin{equation}
	\partial_{ss} \big|_{s=0} \, F_{x_s , t_s} (\Sigma_s) < 0 \, .
\end{equation}
\end{enumerate}
We will use this to prove that $\Sigma$ is also entropy-unstable.

\vskip2mm
\noindent
{\bf{Setting up the proof}}:
Define a function $G: \RR^{n+1} \times \RR^{+} \times [-2\, \epsilon ,   2 \, \epsilon ] \to \RR^{+} $ by
\begin{equation}
	G(x_0 , t_0 , s) = F_{x_0 , t_0} \, \left( \Sigma_s \right) \, .
\end{equation}
We will show that there exists some $\epsilon_1 > 0$ so that if $s \ne 0$ and $|s| \leq \epsilon_1$, then
\begin{equation}	\label{e:goal}
	\lambda (\Sigma_s) \equiv \sup_{x_0 , t_0} \, G(x_0 , t_0 , s) < G(0,1,0) = \lambda (\Sigma) \, ,
\end{equation}
and this will give the theorem with $\tilde{\Sigma}$ equal to $\Sigma_s$ for any   $s  \ne 0$ in $ (-\epsilon_1 ,\epsilon_1)$; by taking $s >0$ small enough, we can arrange that $\tilde{\Sigma}$ is as close as we like to $\Sigma_0 = \Sigma$.

The remainder of the proof is devoted to establishing \eqr{e:goal}.  The key points will be:
\begin{enumerate}
\item   $G$ has a strict local maximum at $(0,1,0)$.
\item  The restriction of $G$ to $\Sigma_0$, i.e., $G(x_0 , t_0, 0)$, has a strict global maximum at $(0,1)$.
\item	 $\left| \partial_s G \right|$ is uniformly bounded on compact sets.
\item  $G(x_0,t_0 , s)$ is strictly less than $G(0,1,0)$ whenever  $|x_0|$ is sufficiently large.
\item $G(x_0,t_0 , s)$ is strictly less than $G(0,1,0)$ whenever $\left| \log t_0 \right|$ is sufficiently large.
\end{enumerate}

\vskip2mm
\noindent
{\bf{The proof of \eqr{e:goal} assuming (1)--(5)}}:
We will divide into three separate regions depending on the size of $|x_0|^2 + ( \log t_0 )^2$.

First,
it follows from steps (4) and (5) that there is some $R > 0$ so that \eqr{e:goal} holds for every $s$ whenever
\begin{equation}
	x_0^2 + ( \log t_0 )^2 > R^2 \, .
\end{equation}

Second, as long as $s$ is small, step (1) implies that \eqr{e:goal} holds when $x_0^2 + ( \log t_0 )^2$ is sufficiently small.

Finally,  in the intermediate region where $x_0^2 + ( \log t_0 )^2$  is bounded from above and bounded uniformly away from zero, step (2) says that $G$ is strictly less than $\lambda (\Sigma)$ at $s=0$ and step (3) says that the $s$ derivative of $G$ is uniformly bounded.  Hence,  there exists some $\epsilon_3 > 0$ so that
$G(x_0 , t_0 , s)$ is strictly less than $\lambda (\Sigma)$ whenever $(x_0 , t_0)$ is in the intermediate region as long as  $|s| \leq \epsilon_3$.

This completes the proof of \eqr{e:goal} assuming (1)--(5).

\vskip2mm
\noindent
{\bf{The proof of  (1)--(5)}}:

\vskip2mm
\noindent
{\bf{Step (1)}}: Since $\Sigma$ is a self-shrinker, it follows from Proposition
  \ref{p:critall}
  that $\nabla G$ vanishes at $(0,1,0)$.   Given $y \in \RR^{n+1}$ , $a \in \RR$,  and $b \in \RR \setminus \{ 0 \}$, the second derivative of $G( sy ,  1 + a s ,   b s)$  at $s=0$  is just $b^2 \, \partial_{ss} \big|_{s=0} \, F_{x_s , t_s} (\Sigma_s) $ with $x_s = s \frac{y}{b}$ and $t_s = 1 + \frac{a}{b} \, s$.  In particular,  it is negative by (V2).  On the other hand, by
  Theorem \ref{t:secvar}, the second derivative of  $G( sy ,  1 + a s ,   0)$  at $s=0$ is given by
    \begin{equation}	\label{e:secvar27a}
	   (4\pi)^{-n/2} \, \int_{\Sigma}
	 \left( -
	 a^2 \, H^2
		-   \frac{\left| y^{\perp} \right|^2}{2}
		\right)
		 \, \e^{\frac{-|x|^2}{4}}  \, d\mu  \,  .
\end{equation}
Since $\Sigma$ does not split off a line, it is not a hyperplane and, thus, $H$ cannot be zero everywhere; similarly, $\left| y^{\perp} \right|$ can only vanish identically when $y=0$.  We conclude that
   the Hessian of $G$ at $(0,1,0)$ is negative definite.  It follows that $G$ has a strict local maximum at $(0,1,0)$, so there exists $\epsilon_2 \in (0,\epsilon)$ so that
   \begin{equation}	\label{e:localmax}
   	G(x_0,t_0 , s) < G(0,1,0) {\text{ if }} 0 < x_0^2 + (\log t_0)^2 + s^2 < \epsilon_2^2 \, .
   \end{equation}

\vskip2mm
\noindent
{\bf{Step (2)}}:
If   we restrict $G$ to $\Sigma_0$, then
  Lemma \ref{l:strict} gives that $G(\cdot , \cdot , 0)$ has a strict global maximum at $(0,1,0)$.  It follows that $\lambda (\Sigma) = G(0,1,0)$ and
  there exists $\delta > 0$ so that
   \begin{equation}	\label{e:globalmax}
   	G(x_0,t_0 , 0) < G(0,1,0) - \delta {\text{ for all $x_0, \, t_0$ with }} \frac{\epsilon_2^2}{4}  < x_0^2 + (\log t_0)^2   \, .
   \end{equation}

\vskip2mm
\noindent
{\bf{Step (3)}}:
Using the first variation formula (Lemma \ref{l:varl0}),
we get that
\begin{equation}	\label{e:sG}
	  \partial_s \, G(x_0,t_0,s) =
 (4\pi \, t_0)^{-\frac{n}{2}} \,  \int_{\Sigma_s}    f_{\Sigma_s} \, \left( H_{\Sigma_s} -
 \frac{ \langle x-x_0 ,  \nn_{\Sigma_s} \rangle}{2t_0}
		\right)
		\, \e^{\frac{-|x-x_0|^2}{4t_0}} \, d\mu_{\Sigma_s}  \, ,
\end{equation}
where  $H_{\Sigma_s}$ is the mean curvature of $\Sigma_s$ and
$ d\mu_{\Sigma_s}$ is the volume measure on $\Sigma_s$.  Since $f_{\Sigma_s}$ is smooth and
 compactly supported and $\Sigma$ has polynomial volume growth,
$\partial_s G$ is evidently continuous in all three variables $x_0$, $t_0$, and $s$.  We conclude that $\partial_s G$ is uniformly bounded on compact sets.

\vskip2mm
\noindent
{\bf{Step (4)}}:  Since $\Sigma_s$ is a $C^2$ family of  surfaces that vary only in a compact set and $s$ is in the compact interval $[-\epsilon , \epsilon]$, there are uniform upper bounds
\begin{equation}
	\Vol (\Sigma_s \setminus \Sigma) \leq C_V {\text{ and }} \max_{\Sigma_s \setminus \Sigma} |x| \leq C_x \, .
\end{equation}
It follows that for $|x_0| > C_x + R$ with $R > 0$ we have
\begin{align}	\label{e:gxts0}
	G(x_0 , t_0 , s) &\leq  (4\pi \, t_0)^{ - \frac{n}{2}  } \, \left( \int_{\Sigma_s \setminus \Sigma} \e^{ - \frac{|x-x_0|^2}{4t_0} } +
	  \int_{ \Sigma} \e^{ - \frac{|x-x_0|^2}{4t_0} }  \right) \notag \\
	& \leq C_V \, (4\pi \, t_0)^{ - \frac{n}{2}  } \,  \e^{ - \frac{R^2}{4t_0} } + G(x_0 , t_0 , 0)\, .
\end{align}
Therefore, if we define   functions $u_R (r) = r^{ - \frac{n}{2}  } \,  \e^{ - \frac{R^2}{4r} }$ for $r>0$, then
for $|x_0| > C_x + R$ we have
\begin{equation}
	G(x_0 , t_0 , s) \leq C_V \, (4\pi)^{ - \frac{n}{2}  } \,   u_R(t_0) + G(x_0 , t_0 , 0) \, ,
\end{equation}
so step (4) follows from step (2) and showing that $\sup_{ r> 0} \, u_R (r)$ converges to zero as $R \to \infty$.  To see this, note that $u_R(r) = R^{-n} \, u_1 \left( \frac{r}{R^2} \right)$ and
$u_1$ is bounded since it is continuous and goes to zero as $r$ goes to either zero
or  infinity.

\vskip2mm
\noindent
{\bf{Step (5)}}:
The easy case  when $t_0$ is large follows  from step (2) since we always have
\begin{equation}	
	G(x_0 , t_0 , s) \leq   (4\pi \, t_0)^{ - \frac{n}{2}  } \, \Vol (\Sigma_s \setminus \Sigma )
	+ G(x_0 , t_0 , 0) \leq
	C_V \, (4\pi \, t_0)^{ - \frac{n}{2}  } + G(x_0 , t_0 , 0)  \, .
\end{equation}
This estimate also immediately gives a uniform bound for $G(x_0 , t_0 , s)$ for all $t_0 \geq 1$:
\begin{equation}		\label{e:Gt01}
	\sup_{t_0 \geq 1} \, \, G(x_0 , t_0 , s) \leq
	C_V \, (4\pi )^{ - \frac{n}{2}  } + \lambda (\Sigma)  \, .
\end{equation}

The   case
  when $t_0$ is close to zero will follow from using
   that $\lim_{t_0 \to 0} G(x_0 , t_0 , s) \leq 1$ for each fixed $(x_0 , s)$  together with a covering argument in $(x_0 , s)$ (using compactness), and then showing that once $G(x_0 , t_0 , s)$ is small for some $t_0 > 0$, then it is also small for all   $t \leq t_0$.
  We will make this precise next.

 We begin by noting
  that
(by compactness) there is a uniform constant $C_H$ with
\begin{equation}	\label{e:CH}
	\sup_{|s| \leq \epsilon } \, \, \left( \sup_{\Sigma_s \setminus \Sigma} \, \left| H_{\Sigma_s} \right|  \right) \leq C_H \, .
\end{equation}
Arguing as in \eqr{e:from77}, we have
\begin{align}	\label{e:tzerosmall}
	\partial_{t_0} \, G(x_0 , t_0 , s)
	&\geq   -    (4\pi t_0)^{ - \frac{n}{2} } \,
	\int_{\Sigma}  \frac{H_{\Sigma}^2}{4} \, \e^{-\frac{|x - x_0|^2}{4t_0}}
	 -    (4\pi t_0)^{ - \frac{n}{2} } \,
	\int_{\Sigma_s \setminus \Sigma}  \frac{H_{\Sigma_s}^2}{4} \, \e^{-\frac{|x - x_0|^2}{4t_0}}
	   \notag \\
	   &\geq -   \lambda (\Sigma) \, \left( |x_0|^2 + C_n \, t_0 \right) - \frac{C_H^2}{4} \,  G(x_0 , t_0 , s) \, ,
\end{align}
where $C_n$ is a constant depending only on $n$,  the second inequality used Lemma \ref{l:wtHL}
and the self-shrinker equation to bound the first term,  and the second inequality used \eqr{e:CH} to bound the second term.

By step (4), there exists some $R_0 > 0$ so that if $|x_0| \geq R_0$, then $G(x_0 , t_0 , s)$ is strictly less than $\lambda (\Sigma)$.  Note that
\eqr{e:tzerosmall} gives a constant $C_0$ so that if $|x_0| \leq R_0 + 2$ and $t_0 \leq 1$, then
\begin{equation}	\label{e:tzerosmall0}
	\partial_{t_0} \, \left[ \e^{ \frac{C_H^2 \, t_0 }{4}  } \, G(x_0 , t_0 , s) \right] \geq -   C_0 \, .
\end{equation}
The point is that $C_0$ does not depend on $x_0$, $s$, or $t_0$.  One immediate consequence of this
together with \eqr{e:Gt01} is that there is a uniform upper bound  $\lambda (\Sigma_s)  \leq
\bar{\lambda}$.

It follows from step (2) above together with (3) in Lemma \ref{l:basicF}
that $G(0,1,0) = \lambda (\Sigma) > 1$, so we can
fix   $\alpha > 0$ with
\begin{equation}
	1 + 3 \, \alpha < \lambda (\Sigma) \, .
\end{equation}
Define a constant $t_{\alpha} > 0$ by
$(1+ \alpha) \, \e^{ \frac{ C_H^2 \, t_{\alpha} }{4} } + C_0 \, t_{\alpha} = 1 + 2 \, \alpha$ so that for $t \in (0, t_{\alpha} ]$
\begin{equation}	\label{e:talpha}
	(1+ \alpha) \, \e^{ \frac{ C_H^2 \, t  }{4} } + C_0 \, t  \leq 1 + 2 \, \alpha \, .
\end{equation}

 Given any $x \in \RR^{n+1}$ and $s \in [-\epsilon , \epsilon]$, property (3) in Lemma \ref{l:basicF} gives some $t_{x,s} > 0$ so that
 $G(x,t_0,s) < 1 + \alpha$ for all $t_0 \leq t_{x,s}$.
 Using the continuity (in both $x$ and $s$) of $G$ and the compactness of each $\overline{B_{R_0 + 1}}$ and $[-\epsilon , \epsilon]$, we can choose a finite collection of pairs $(x_i ,s_i) \in \overline{B_{R_0+ 1}}   \times [-\epsilon , \epsilon]$, positive numbers $t_i <
  \min \, \{ t_\alpha  , 1 \}$, and radii $r_i \in (0,1) $ so that
 \begin{itemize}
 \item  For each $s \in [-\epsilon , \epsilon]$ and $x \in \overline{B_{R_0+1}}$, there exists $i (x,s)$ so that
 $x \in  B_{r_{i(x,s)}} (x_{i(x,s)})$ and $|s-s_{i(x,s)}| < r_{i(x,s)}$.
 \item For each  $i$,   we have  $
 G(x, t_i , s) < 1 + \alpha$  for every $x \in B_{2r_i}(x_i)$ and $s$ with $|s-s_i| < r_i$.
 \end{itemize}
 If we let  $\bar{t}$ be the minimum of the $t_i$'s, then
 \eqr{e:tzerosmall0} gives   for any $s \in [-\epsilon , \epsilon]$, any $x \in \overline{B_{R_0}}$, and any $t_0 \leq \bar{t}$ that
 \begin{align}	\label{e:bart}
 	G(x , t_0 , s) &\leq \e^{ \frac{C_H^2 \, t_0}{4} } \, G(x , t_0 , s) \leq  \e^{ \frac{C_H^2 \, t_{i(x,s)}}{4} } \,  G(x , t_{i(x,s)}, s)   + C_0 \, (t_{i(x,s)} - t_0)  \notag \\
	&\leq
	1+ 2\alpha \, ,
 \end{align}
 where the last inequality used \eqr{e:talpha}, $i(x,s)$ satisfies $|s-s_{i(x,s)}| < r_i$, and $x \subset B_{2r_{i(x,s)}} (x_{i(x,s)})$.   It follows that $G(x , t_0 , s)$ is strictly less than $\lambda(\Sigma)$ for all $t_0 \leq  \tilde{t}$.

\end{proof}

Combining Theorem \ref{t:linearstable} and Theorem  \ref{t:nonlin1c}, we get that the sphere is the only closed entropy-stable self-shrinker:

 \begin{Cor}	\label{t:nonlin1csphere}
Suppose that $\Sigma \subset \RR^{n+1}$ is a  smooth closed embedded self-shrinker,  but is not a round sphere.  Then $\Sigma$ can be perturbed  to an arbitrarily close hypersurface $\tilde{\Sigma}$ where the entropy is strictly less.  In  particular,   $\Sigma$ (or a translation of it) cannot arise as a tangent flow to the MCF starting from $\tilde{\Sigma}$.
\end{Cor}

\begin{proof}
Since $\Sigma$ is a closed self-shrinker, but is not the round sphere, Theorem \ref{t:linearstable} implies that $\Sigma$ is $F$-unstable.  Clearly, $\Sigma$ does not split off a line, so Theorem  \ref{t:nonlin1c} gives that it is entropy-unstable.
\end{proof}

\section{Self-shrinkers with bounded diameter}	\label{s:8}

In this section, we will give a short direct proof of Theorem \ref{c:grayson}.  This theorem will follow immediately
from the main theorem in \cite{CM1},  but the direct proof in this section  serves to illustrate the key ideas in a simpler case.  The results in this section are for surfaces in $\RR^3$; most of the ingredients in the proof work for all dimensions, but  the following compactness theorem from \cite{CM3}  is proven only for surfaces in $\RR^3$:

\begin{Thm}	\label{c:cpt}
(\cite{CM3}).
 Given a non-negative integer $g$ and a positive constant $V$,
the space  of smooth complete embedded self-shrinkers $\Sigma$ with genus at most $g$,  $\partial \Sigma = \emptyset$, and
the scale-invariant area bound
 $\Area \, \left( B_{R}(x_0) \cap \Sigma \right) \leq V \, R^2$  for all $x_0 \in \RR^3$ and all $R > 0$ is compact.   Namely, any sequence  has a subsequence that converges in the
 topology of  $C^k$ convergence on compact subsets for any $k \geq 2$.
\end{Thm}

Throughout this section, we will fix a non-negative integer $g$ and a positive constant
 $\bar{\lambda}$; these will be bounds for the genus and entropy which will in applications come from bounds on the initial surface in a MCF.
 We will not say so explicitly below, but all constants in this section will be allowed to depend on $g$ and $\bar{\lambda}$.

 Given a constant $D > 0$,  let $S_D = S_{g,\bar{\lambda},D}$ denote the space of all smooth closed embedded self-shrinkers in $\RR^3$
 with genus at most $g$,
 entropy at most $\bar{\lambda}$, and diameter at most $D$.

As a corollary of Theorem \ref{c:cpt}, we get a smooth compactness theorem for $S_{D}$:

\begin{Cor}	\label{c:cpt2}
 For each fixed  $D$, the space $S_{D}$
 is compact.   Namely, any sequence in $S_{D}$   has a subsequence that converges uniformly in the $C^k$ topology (any fixed $k$)
 to a surface in $S_{D}$.
\end{Cor}

\begin{proof}
Given $\Sigma \in S_D$,
the bound on the entropy gives the scale-invariant area bound
\begin{equation}
	(4\pi)^{-1} \, \e^{ - \frac{1}{4} } \,
\frac{ \Area \, \left( B_{R}(x_0) \cap \Sigma \right) }{R^2}	 \leq   (4\pi \, R^2 )^{-1} \int_{B_R(x_0) \cap \Sigma} \e^{ - \frac{|x-x_0|^2}{4 R^2} } \leq
	F_{x_0 , R^2} (\Sigma) \leq \bar{\lambda} \, .
\end{equation}
Therefore,
Corollary \ref{c:cpt} gives a subsequence that converges on compact subsets to a (possibly non-compact) limiting self-shrinker.  However, the limit must have diameter at most $D$ (if not, one of the surfaces sufficiently far out in the subsequence would have diameter greater than $D$), so it is closed and also in $S_{D}$.
\end{proof}

The next result is a corollary of the compactness theorem for self-shrinkers together with the result that the round sphere is the only closed entropy-stable self-shrinker.

\begin{Cor}		\label{c:nl}
Given $D >0$, there exists $\epsilon > 0$ so that if $\Sigma \in S_{D}$ is not  the round sphere,  then there is a graph $\Gamma$ over $\Sigma$ with
$\lambda ( \Gamma) < \lambda (\Sigma) -  \epsilon$.
\end{Cor}

\begin{proof}
Given any  $\Sigma_1 \in S_{D}$ that is not the round sphere,  Corollary  \ref{t:nonlin1csphere} gives a graph  $\tilde{\Sigma}_1$ over $\Sigma_1$ and a constant $\epsilon_1 > 0$ so that
\begin{equation}
		\lambda ( \tilde{\Sigma}_1 ) \leq \lambda (\Sigma_1) - \epsilon_1 \, .
\end{equation}
It follows from the first variation formula (Lemma \ref{l:varl0}) that  $F_{0,1}(\Sigma_s)$ is continuous in $s$ when $\Sigma_s$ varies continuously in a $C^2$ neighborhood of $\Sigma_1$. We can therefore choose a neighborhood $ {U}_1$ of $\Sigma_1$ in the $C^2$ topology so that
\begin{equation}
	\min_{\Sigma \in  {U}_1} \, \lambda (\Sigma) \geq
	\min_{\Sigma \in  {U}_1} \, F_{0,1}(\Sigma) \geq F_{0,1} (\Sigma_1)  - \frac{\epsilon_1}{2}
	= \lambda (\Sigma_1) - \frac{\epsilon_1}{2} \, ,
\end{equation}
where the last equality used  Lemma \ref{l:strict}.
After possibly  shrinking the neighborhood $U_1$ of $\Sigma_1$, we can assume that each pair of surfaces in $U_1$ are graphs over each other.

We observe next  $\SS^2$ is an isolated point in $S_{D}$ with respect to the $C^2$ topology since any sufficiently close self-shinker is still closed and mean convex and, thus, a round $\SS^2$ by Huisken's classification, \cite{H3}.
 In particular,
since $S_{D}$ is compact in the $C^2$ topology by Corollary \ref{c:cpt2}, so is $S_{D} \setminus
\{ \SS^2 \}$.  Therefore,
we can choose a finite set of self-shrinkers $\Sigma_i \in  S_{D} \setminus   \{ \SS^2 \}$, graphs $\tilde{\Sigma}_i$ over $\Sigma_i$, neighborhoods $U_i$ of $\Sigma_i$, and constants $\epsilon_i > 0$ so that:
\begin{itemize}
\item  $S_{D} \setminus \{ \SS^2 \} \subset \cup U_i$ \, .
\item If
$\Sigma \in U_i \cap S_{D} $, then
$\lambda ( \tilde{\Sigma}_i ) \leq \lambda (\Sigma) - \epsilon_1/2$ and $\tilde{\Sigma}_i$ is a graph over $\Sigma$.
\end{itemize}
The corollary follows with $\epsilon = \frac{1}{2} \, \min \epsilon_i$.
\end{proof}

\subsection{MCF's that disappear in compact points generically do so in round points}

We will need the following consequence of   Brakke's theorem, \cite{B}, or White, \cite{W5}:

\begin{Lem}	\label{l:brakke}
Suppose that $M_t$ is a MCF of smooth closed surfaces in $\RR^3$ for $t < 0$
and $\Sigma_0$ is a closed smooth embedded self-shrinker equal to the   $t=-1$ time-slice of a multiplicity one tangent flow to  $M_t$ at $(0,0)$. Then
  we can choose a sequence $s_j > 0$ with $s_j \to 0$ so that
\begin{equation}
  \frac{1}{\sqrt{s_j}} \, M_{- s_j}  {\text{ is a graph over }} \Sigma_0  {\text{ of a function }} u_j
   {\text{ with }}  ||u_j||_{C^{2}} \to 0  \, .
\end{equation}
\end{Lem}

\begin{proof}
The main point is to show that Huisken's density is sufficiently close to one on all sufficiently small scales for the rescalings of $M_t$.  Once we have this, then Brakke's theorem will give uniform estimates and the lemma will follow easily.

  Fix $\epsilon > 0$ small (to be given by theorem $3.1$ in \cite{W5}).  Since $\Sigma_0$ is a smooth closed embedded surface, (2) and (3) in Lemma \ref{l:basicF}  give some $\bar{r} > 0$ so that
  \begin{equation}	\label{e:choosebarr}
  	\sup_{ t_0 \leq \bar{r}} \, \left( \sup_{x_0 \in \RR^3} \, F_{x_0 , t_0} (\Sigma_0) \right) < 1 + \epsilon \, .
  \end{equation}

    The definition of tangent flow (see, e.g., Lemma $8$ in \cite{I1} or \cite{W4}) gives a sequence $s_j > 0$ with $s_j \to 0$ so that  the rescaled flows
  $M^j_t = \frac{1}{\sqrt{s_j}} \, M_{ \frac{t}{s_j}}$ converge to the multiplicity-one  flow $\sqrt{-t} \, \Sigma_0$.    Let $M^j_{-1} = \frac{1}{\sqrt{s_j}} \, M_{ \frac{-1}{s_j}}$ be the time $t=-1$ slice of the $j$-th rescaled flow.  We can assume that the $M^j_{-1}$'s converge to $\Sigma_0$ as Radon measures and with respect to Hausdorff distance (see $7.1$ in \cite{I2}).

 We will use the convergence together with \eqr{e:choosebarr} to get uniform bounds for the $F$ functionals on the $M^j_{-1} $'s.  To do this, define a sequence of functions $g_j$ by
 \begin{equation}
 	g_j (x_0 , t_0) = F_{x_0 , t_0} (M^j_{-1} )  \, .
 \end{equation}
 We will only consider the $g_j$'s on the region   $\overline{B} \times [\bar{r}/3 , \bar{r} ]$ where $B \subset \RR^3$ is a fixed ball that contains $\Sigma_0$ and all of the $M^j_{-1} $'s.
 It follows from the first variation formula (Lemma \ref{l:varl0}) that the $g_j$'s are uniformly Lipschitz in this region with
 \begin{equation}	\label{e:gjlip}
 	\sup_{ \overline{B} \times [\bar{r}/3 , \bar{r} ]} \, \, \left| \nabla_{x_0 , t_0} \, g_j \right| \leq C \, ,
 \end{equation}
 where $C$ depends on $\bar{r}$, the radius of the ball $B$, and the scale-invariant local area bounds for the
  $M^j_{-1} $'s which are uniformly bounded by
Lemma \ref{l:volb}.
Since the $M^j_{-1} $'s converge to $\Sigma_0$ as Radon measures and $\Sigma_0$ satisfies
\eqr{e:choosebarr}, it follows from \eqr{e:choosebarr} that
\begin{equation}
	\lim_{j \to \infty} \, g_j (x_0 , t_0) < 1 + \epsilon {\text{ for each fixed }}
	(x_0 , t_0) \in B \times [\bar{r}/3 , \bar{r} ] \, .
\end{equation}
Combining this with \eqr{e:gjlip} and the compactness of $\overline{B} \times [\bar{r}/3 , \bar{r} ]$, there exists some $\bar{j}$ sufficiently large so that for all $j > \bar{j}$ we have
\begin{equation}	\label{e:densityb1}
	\sup_{ \overline{B} \times [\bar{r}/3 , \bar{r} ]} \, F_{x_0 , t_0} (M^j_{-1} ) =
	\sup_{ \overline{B} \times [\bar{r}/3 , \bar{r} ]} \, g_j(x_0 , t_0) < 1 + 2\, \epsilon \, .
\end{equation}
  By  \eqr{e:huiskenF} (which used Huisken's monotonicity formula), we get for every $t> -1$ that
\begin{equation}	\label{e:huiskenFa}
		F_{x_0 , t_0} ( M^j_t)  \leq  F_{x_0 , t_0 + (t+1)} (M^j_{-1})  \, .
\end{equation}
Hence, if $t \in (-1 + \frac{\bar{r}}{3} , - 1 +   \frac{2 \bar{r}}{3})$,   $t_0 \leq
 \frac{\bar{r}}{3}$, and $j > \bar{j}$, then \eqr{e:densityb1} and \eqr{e:huiskenFa} together yield
 \begin{equation}
 	F_{x_0 , t_0} ( M^j_t)  < 1 + 2\, \epsilon \, .
 \end{equation}
This is precisely what is needed to apply
theorem $3.1$ in \cite{W5} to get  uniform $C^{2,\alpha}$ bounds on $M^j_t$ for all
 $t \in (-1 + \frac{4 \bar{r}}{9} , - 1 +   \frac{5 \bar{r}}{9})$  for some $\alpha \in (0, 1)$.  We can slightly change the $s_j$'s so that we instead have uniform $C^{2,\alpha}$ bounds on $M^j_{-1}$.

Finally, observe that if $\Sigma_j$ is a sequence of closed surfaces converging to a closed surface $\Sigma_0$ as Radon measures and both the $\Sigma_j$'s and $\Sigma_0$ satisfy uniform $C^{2,\alpha}$ bounds, then the $\Sigma_j$'s must converge to $\Sigma_0$ uniformly in  $C^{2}$.
\end{proof}

\begin{proof}
(of Theorem \ref{c:grayson}).
We will construct a piece-wise MCF with a finite number of discontinuities that eventually becomes extinct in a round point.   This comes from doing a smooth jump just before a  (non-round) singular time, where we replace a time slice of the flow by a graph over it, and the crucial point is to show that there is a positive constant $\epsilon > 0$ so that the entropy   decreases by at least $\epsilon$ after each replacement.  This will come from Corollary \ref{c:nl}.  We repeat this until we get to a singular point where every tangent flow consists of shrinking spheres.

Let $M^0_t$ denote the MCF with initial surface $M_0$ and let $t^{sing}_0$ be the first singular time.  By assumption, all of the tangent flows at $t^{sing}_0$ are smooth, have multiplicity one, and correspond to compact self-shrinkers with diameter at most $D$.  In particular, since the tangent flows are compact, there is only one singular point $x_0 \in \RR^3$.
  Let $\Sigma_0$ be a self-shrinker equal to the $t=-1$ time-slice of a multiplicity one tangent flow at the singularity.
  Since $\Sigma_0$ is assumed to be smooth, closed, embedded and multiplicity one,
     Lemma \ref{l:brakke} gives a  sequence $s_j > 0$ with $s_j \to 0$ so that
\begin{equation}	\label{e:cvgst0}
  \frac{1}{\sqrt{s_j}} \, \left( M^0_{t^{sing}_0 - s_j} - x_0 \right)  {\text{ is a graph over }} \Sigma_0  {\text{ of a function }} u_j
   {\text{ with }}  ||u_j||_{C^{2}} \to 0  \, .
\end{equation}

There are two possibilities.  First, if $\Sigma_0$ is the round sphere for every tangent flow at $x_0$, then \eqr{e:cvgst0} implies that
$M^0_{t^{sing}_0 - s_j}$ is converging to a round sphere for every sequence $s_j \to 0$.

Suppose instead that there is at least one tangent flow so that
$\Sigma_0$ is not the round sphere.  Then,
we can apply Corollary \ref{c:nl} to get a graph $\Gamma_0$ over $\Sigma_0$ with
$	\lambda (\Gamma_0) < \lambda (\Sigma_0) - \epsilon $
where $\epsilon > 0$ is a fixed constant given by the corollary.  When $j$ is sufficiently large,
$\left( \sqrt{s_j} \Gamma_0 \right) + x_0$ is a graph over $M^0_{t^{sing}_0 - s_j}$ and
\begin{equation}
	\lambda ( \left( \sqrt{s_j} \, \Gamma_0 \right) + x_0 ) = \lambda (\Gamma_0) < \lambda (\Sigma_0) - \epsilon  \leq \lambda  (M^0_{t^{sing}_0 - s_j}) - \epsilon \, ,
\end{equation}
where the first equality used
the scale invariance of entropy (by \eqr{e:scaleF}) and the last inequality used
the monotonicity of entropy under MCF (Lemma \ref{l:entromono}).  We then replace $M^0_{t^{sing}_0 - s_j}$ with $\left( \sqrt{s_j} \Gamma_0 \right) + x_0$ and restart the flow.
Since the entropy goes down by a uniform constant $\epsilon > 0$ at each replacement and is non-increasing under MCF, this can only occur a finite number of times and the flow eventually reaches a singularity where all tangent flows are round.
\end{proof}

\section{Non-compact self-shrinkers}	\label{s:noncompactspec}

In this section, suppose that $\Sigma \subset \RR^{n+1}$ is a complete non-compact hypersurface without boundary that satisfies $H = \frac{\langle  x , \nn \rangle}{ 2}$    and  has polynomial volume growth.
 Throughout, $ L= \Delta + |A|^2 + \frac{1}{2} - \frac{1}{2} \, x \cdot \nabla$ will be operator from  the second variation formula for the $F_{0,1}$ functional.

 \subsection{The  spectrum of $L$ when $\Sigma$ is non-compact}

Since $\Sigma$ is non-compact,  there may not be a lowest eigenvalue for $L$.  However, we can still define the bottom of the spectrum (which we still call $\mu_1$) by
\begin{equation}	\label{e:lam1}
	\mu_1 = \, \inf_{f} \,  \, \frac{    \int_{\Sigma} \left(  |\nabla f |^2 - |A|^2 \, f^2
	- \frac{1}{2} \, f^2 \right) \, \e^{- \frac{|x|^2}{4} } }{   \int_{\Sigma} f^2 \, \e^{- \frac{|x|^2}{4} } } \, ,
\end{equation}
where the infimum is taken over smooth functions $f$ with compact support.   Since $\Sigma$ is non-compact,   we must allow the possibility that $\mu_1 = - \infty$.   Clearly, we have that $\mu_1 \leq \mu_1 (\Omega)$ for any subdomain $\Omega \subset \Sigma$.

We show next that $\mu_1$ is always negative; this result will not be used here, but it is used in \cite{CM3}.

\begin{Thm}	\label{t:spectral0}
If the hypersurface $\Sigma \subset \RR^{n+1}$ is a smooth complete embedded self-shrinker without boundary and with polynomial volume growth,
then $\mu_1 \leq  - \frac{1}{2}$.
\end{Thm}

\begin{proof}
(of Theorem \ref{t:spectral0}).
   Fix a point $p$ in $\Sigma$ and define a function   $v$ on $\Sigma$ by
\begin{equation}
	v(x)  = \langle \nn (p) , \nn (x) \rangle \, .
\end{equation}
It follows that $v (p) = 1$, $|v| \leq 1$, and, by \eqr{e:spec1},
\begin{equation}
	L \, v = \frac{1}{2} \, v \, .
\end{equation}
When $\Sigma$ is closed, it follows immediately that $\mu_1 \leq - \frac{1}{2}$.  However, if $\Sigma$ is open, then we cannot use $v$ as a test function to get an upper bound for $\mu_1$.
To deal with this,
given any fixed $\mu < - \frac{1}{2}$, we will use $v$ to
  construct a  compactly supported function $u$ with
  \begin{equation}	\label{e:Lstable}
	- \int \left(  u \, L u + \mu \, u^2 \right) \,  \e^{\frac{-|x|^2}{4}} < 0 \, .
  \end{equation}
Given any smooth function $\eta$, we have
\begin{equation}
	L \, (\eta \, v)  = \eta \, L \, v + v \, \cL \, \eta + 2 \langle \nabla \eta , \nabla v
	\rangle
	  = \frac{1}{2} \, \eta \, v  + v \,  \cL \, \eta + 2 \langle \nabla \eta , \nabla v
	\rangle
\, ,
\end{equation}
where  $\cL = \Delta   - \frac{1}{2} \, x \cdot \nabla$ is the operator from Lemma \ref{l:self}.
Taking $\eta$ to have compact support, we get that
\begin{align}	\label{e:notst}
	- \int \eta \, v \, L (\eta \, v) \, \e^{\frac{-|x|^2}{4}} &\leq
	  - \int \left[   \frac{1}{2} \,   \eta^2 \, v^2  +
	\eta \,   v^2  \, \cL \, \eta +
	\frac{1}{2}  \,  \langle \nabla \eta^2 , \nabla v^2
	\rangle  \right] \, \, \e^{\frac{-|x|^2}{4}}  \notag \\
	& =  \int \left[  - \frac{1}{2} \,   \eta^2 \, v^2  +  v^2  \, | \nabla \eta|^2 \right]
	  \, \e^{\frac{-|x|^2}{4}}
	  \, ,
\end{align}
where the  equality came from applying  Lemma \ref{l:self}
to   get
\begin{equation}
	 \int    \frac{1}{2} \,
	 \langle \nabla \eta^2 , \nabla v^2
	\rangle   \, \e^{\frac{-|x|^2}{4}}  = -  \frac{1}{2} \, \int \left( v^2 \, \cL \eta^2 \right) \,  \e^{\frac{-|x|^2}{4}}
	=    -   \int   v^2 \, ( \eta \, \cL \eta + |\nabla \eta|^2  ) \,  \e^{\frac{-|x|^2}{4}}  \, .
\end{equation}
If we take $\eta$ to be identically one on $B_R$ and cut it off linearly to zero on $B_{R+1} \setminus B_R$, then \eqr{e:notst} gives
\begin{equation}	\label{e:contr}
	- \int \left[  \eta \, v \, L (\eta \, v) + \mu \, \eta^2 \, v^2 \right]
	\, \e^{\frac{-|x|^2}{4}} \leq  ( 1 - \mu) \, \int_{\Sigma \setminus B_R} v^2 \, \e^{\frac{-|x|^2}{4}}
	- \left( \frac{1}{2} + \mu \right)  \, \int_{B_R \cap \Sigma}  v^2 \, \e^{\frac{-|x|^2}{4}} \, .
\end{equation}
However, since $|v| \leq 1$ and $\Sigma$ has polynomial volume growth, we know that
\begin{equation}
 \lim_{R \to \infty} \, \, \int_{\Sigma \setminus B_R} v^2 \, \e^{\frac{-|x|^2}{4}} = 0 \, .
\end{equation}
Combining this and $ \left( \frac{1}{2} + \mu \right) > 0$, we see that
  the  right-hand side of \eqr{e:contr} must be negative for all sufficiently large $R$'s.  In particular, when $R$ is large, the function $u = \eta \, v$  satisfies \eqr{e:Lstable}.
  This completes the proof.
\end{proof}

The main result of this subsection is Theorem \ref{c:pde} below which shows that the bottom of the spectrum of $L$ is strictly less than $-1$ if the mean curvature changes sign.

The following standard lemma shows that we get the same $\mu_1$ if we take the infimum over a broader class of functions.

\begin{Lem}	\label{l:infover}
When $\Sigma$ is complete and non-compact, we get the same $\mu_1$ by taking the infimum over Lipschitz functions $f$ satisfying
\begin{equation}	\label{e:allfinte}
	\int_{\Sigma} \left( f^2 + |\nabla f|^2 + |A|^2 \, f^2 \right) \, \e^{- \frac{|x|^2}{4} }    < \infty \, .
\end{equation}
\end{Lem}

\begin{proof}
By standard density arguments, we get the same $\mu_1$ if we take the infimum over Lipschitz functions with compact support.
Let $\phi_j$ be a Lipschitz cutoff function that is identically one on $B_j$ and cuts off to zero linearly between $\partial B_j$ and $\partial B_{j+1}$.  Given any $C^1$ function $f$ (not identically zero) that satisfies \eqr{e:allfinte}, we
 set $f_j = \phi_j \, f$.  It follows from \eqr{e:allfinte} and the monotone convergence theorem that
\begin{align}
	 \int_{\Sigma} \left(  |\nabla f_j |^2 - |A|^2 \, f_j^2
	- \frac{1}{2} \, f_j^2 \right) \, \e^{- \frac{|x|^2}{4} }  &\to    \int_{\Sigma} \left(  |\nabla f |^2 - |A|^2 \, f^2
	- \frac{1}{2} \, f^2 \right) \, \e^{- \frac{|x|^2}{4} }  \, , \\
	\int_{\Sigma} f_j^2 \, \e^{- \frac{|x|^2}{4} }   &\to  \int_{\Sigma} f^2 \, \e^{- \frac{|x|^2}{4} }  \, ,
\end{align}
and the lemma follows from this.
\end{proof}

  The next lemma records integral estimates for   eigenfunctions that are either in the weighted $W^{1,2}$ space or are positive.

    \begin{Lem}	\label{l:L2W12}
  Suppose that $f$ is a $C^2$ function with
 $L \, f = - \mu \, f$ for   $\mu \in \RR$.
 \begin{enumerate}
\item  If    $f$ is in the weighted $W^{1,2}$ space, then  $|A| \, f$ is in the weighted $L^2$ space and
  \begin{equation}	\label{e:L2W12}
  	2 \, \int_{\Sigma}  |A|^2 \, f^2  \, \e^{- \frac{|x|^2}{4} }
	  \leq \int_{\Sigma} \left( (1 - 2 \mu  ) \,  f^2   + 4 \, |\nabla f|^2 \right) \,  \e^{- \frac{|x|^2}{4} }  \, .
  \end{equation}
  \item If $f > 0$ and $\phi$ is in the weighted $W^{1,2}$ space, then
   \begin{equation}	\label{e:Llog2}
  	  \int_{\Sigma}  \phi^2 \, \left(    |A|^2  + |\nabla \log f|^2 \right)   \, \e^{- \frac{|x|^2}{4} }
	  \leq \int_{\Sigma} \left(  4 \, |\nabla \phi  |^2 - 2 \, \mu \, \phi^2 \right) \,  \e^{- \frac{|x|^2}{4} }  \, .
  \end{equation}
  \end{enumerate}
   \end{Lem}

  \begin{proof}
  Within this proof, we will use square brackets $\left[ \cdot \right]$ to denote the weighted integrals
 $\left[ f \right] = \int_{\Sigma} f \, \e^{- \frac{|x|^2}{4} } \, d\mu$.
 We will also use the operator
  $\cL = \Delta   - \frac{1}{2} \, x \cdot \nabla$ from Lemma \ref{l:self}.

 {\bf{Proof of (1)}}:
 If $\phi$ is a smooth function with compact support, then self-adjointness of $\cL = \Delta   - \frac{1}{2} \, x \cdot \nabla$ (Lemma \ref{l:self}) gives
  \begin{equation}	\label{e:cLv}
  	\left[ \langle \nabla \phi^2 , \nabla f^2 \rangle \right] = - \left[ \phi^2 \, \cL f^2 \right] 	=
	- 2 \, \left[ \phi^2 \, \left( |\nabla f|^2  - \left(  |A|^2 + \frac{1}{2} + \mu \right) \,    f^2 \right) \right]\, .
  \end{equation}
  where the last equality used that $\cL f^2 = 2 \, |\nabla f|^2 + 2f \, \cL f$ and
\begin{equation}	\label{e:cLv2}
	\cL f = \left(L -  |A|^2 - \frac{1}{2} \right) \, f = -  \left(  \mu +  |A|^2 + \frac{1}{2} \right) \, f \, .
\end{equation}
 Assume now that $\phi \leq 1$ and $|\nabla \phi| \leq 1$.
Rearranging the terms in \eqr{e:cLv} and
 using Cauchy-Schwarz gives
    \begin{equation}
    	  \left[ \phi^2 \,  (1 + 2 \mu + 2\,  |A|^2 ) \, f^2   \right]
	 =  4 \,  \left[ \phi \, f \,  \langle \nabla \phi  , \nabla f \rangle \right]
    +   2 \,
  	\left[  \phi^2 \,  |\nabla f|^2   \right]   \leq    2 \,
  	\left[  f^2   \right]  + 4 \, \left[ |\nabla f|^2 \right]  \, .
  \end{equation}
  Finally, applying this with $\phi = \phi_j$ where $\phi_j$ is one on $B_j$ and cuts off linearly to zero from $\partial B_j$ to $\partial B_{j+1}$, letting $j \to \infty$, and  using the monotone convergence theorem gives \eqr{e:L2W12}.

   {\bf{Proof of (2)}}:
 Since $f > 0$, $\log f$ is well-defined and we have
  \begin{equation}	\label{e:cLlogf}
	\cL \log f    =
	- |\nabla \log f|^2 + \frac{\Delta f  - \frac{1}{2} \, \langle x , \nabla f \rangle}{f}
	 =   - \mu -  |A|^2 - \frac{1}{2} - |\nabla \log f|^2 \, .
\end{equation}
 Suppose that $\eta$ is a function with compact support.
  As in (1), self-adjointness of $\cL $ (Lemma \ref{l:self}) gives
  \begin{equation}	\label{e:cLlog}
  	\left[ \langle \nabla \eta^2 , \nabla \log f  \rangle \right] = - \left[ \eta^2 \, \cL \log f \right] 	=
	  \left[ \eta^2 \, \left(  \mu +  |A|^2 + \frac{1}{2} + |\nabla \log f|^2 \right) \right] \, .
  \end{equation}
 Combining this with the absorbing inequality
   \begin{equation}	\label{e:cLlog2}
  	  \langle \nabla \eta^2 , \nabla \log f  \rangle   \leq 2 \, | \nabla  \eta |^2 + \frac{1}{2} \, \eta^2 \, |\nabla \log f |^2 \, ,
  \end{equation}
  gives that
    \begin{equation}	\label{e:Llogeta}
  	\left[  \eta^2 \, \left(    |A|^2  + |\nabla \log f|^2 \right)  \right]
	  \leq  \left[  4 \, |\nabla \eta  |^2 - 2 \, \mu \, \eta^2 \right]  \, .
  \end{equation}
  Let   $\eta_j$ be    one on $B_j$ and cut off linearly to zero from $\partial B_j$ to
   $\partial B_{j+1}$.
    Since $\phi$ is in the weighted $W^{1,2}$ space,  applying \eqr{e:Llogeta} with $\eta = \eta_j \, \phi$,
    letting $j \to \infty$,
     and  using the monotone convergence theorem gives  that \eqr{e:Llogeta} also holds with $\eta = \phi$.
     \end{proof}

 We will need the following characterization of the bottom of the spectrum when $\Sigma$ is non-compact that generalizes   (4) in Theorem \ref{t:evans}:

 \begin{Lem}	\label{l:pde}
 If $\mu_1 \ne - \infty$, then
 there is a positive function $u$ on $\Sigma$ with $L \, u = - \mu_1 \, u$.  Furthermore,
 if $v$ is in the weighted $W^{1,2}$ space and $L \, v = - \mu_1 \, v$, then $v = C \, u$ for  $C \in \RR$.
 \end{Lem}

 \begin{proof}
 Within this proof,  square brackets $\left[ \cdot \right]$ denote the weighted integrals
 $\left[ f \right] = \int_{\Sigma} f \, \e^{- \frac{|x|^2}{4} } \, d\mu$.   We will   use the operator
  $\cL = \Delta   - \frac{1}{2} \, x \cdot \nabla$ from Lemma \ref{l:self} and Corollary \ref{c:self}.

{\bf{Part 1: Existence of $u$}}:   Fix a point $p \in \Sigma$ and let $B_k$ denote the intrinsic ball in $\Sigma$ with radius $k$ and center $p$.
 For each $k$, Theorem \ref{t:evans} implies that the lowest Dirichlet eigenvalue $\mu_1 (B_k)$ is a real number and can be represented by a positive Dirichlet eigenfunction $u_k$ on $B_k$ with
 \begin{equation}
 	L u_k = - \mu_1 (B_k) \, u_k
 \end{equation}
 and, after multiplying $u_k$ by a constant, we can assume that $u_k ( p) = 1$.
By the domain monotonicity of eigenvalues{\footnote{This   follows immediately from the variational characterization (2) and uniqueness (4) in Theorem \ref{t:evans}.}}, we know that  $\mu_1 (B_k)$ is a decreasing sequence of real numbers with
$\mu_1 (B_k) > \mu_1$ for each $k$. Furthermore, it is easy to see that
\begin{equation}
	\lim_{k \to \infty} \, \mu_1 (B_k) =  \mu_1 > - \infty \, ,
\end{equation}
where the last inequality is by assumption.
The Harnack inequality gives uniform upper and lower bounds for the $u_k$'s on compact sets.  Combining this with elliptic estimates, we get uniform $C^{2,\alpha}$ bounds on the $u_k$'s on each compact set so that Arzela-Ascoli gives a subsequence that converges uniformly in $C^2$ to a non-negative entire solution $u$ of $L \, u = - \mu_1 \, u$ with $u(p) =1$.  Another application of the Harnack inequality gives that $u$ is positive.

  {\bf{Part 2: Weighted integral estimates}}:  Since $v$ is in the weighted $W^{1,2}$ space, we can apply
part (1) of Lemma \ref{l:L2W12} to $v$ to get that $ |A| \, v$ is in the weighted $L^2$ space.  Using the equation for $\cL v$
\begin{equation}	\label{e:cLva}
	\cL \, v = \left( L -  |A|^2 - \frac{1}{2} \right) \, v = - \left(  \mu_1 +  |A|^2 + \frac{1}{2} \right) \, v \, ,
\end{equation}
it follows that
  \begin{equation}	\label{e:vw22}
  	\left[ |v|^2 + |\nabla v|^2 + |v \, \cL v| \right] < \infty \, .
  \end{equation}
    Since $v$ is in the weighted $W^{1,2}$ space,  we can apply (2) in Lemma \ref{l:L2W12} with $f=u$ and $\phi =  v$ to get that
  $|A| \, v$ and $|\nabla \log u| \, v$ are in the weighted $L^2$ space.  We conclude that
  \begin{equation}	\label{e:logu12}
  	\left[   v^2 \, | \cL \, \log u |  \right]  = \left[   v^2 \, \left| \mu_1 + \frac{1}{2} + |A|^2 + \left| \nabla \log u \right|^2  \right|  \right]  < \infty \, .
  \end{equation}
  Since $2 \, v^2 | \nabla \log u| \leq  v^2 + v^2 \,  | \nabla \log u| ^2$ and $| \nabla v^2| \,  | \nabla \log u| \leq
  |\nabla v|^2 + v^2 \,  | \nabla \log u| ^2$, we conclude
    \begin{equation}	\label{e:logu22}
  	\left[ v^2 | \nabla \log u| + | \nabla v^2| \,  | \nabla \log u|
	+   v^2 \, | \cL \, \log u |  \right]  < \infty \, .
  \end{equation}

  {\bf{Part 3: Uniqueness}}:
 Since $v^2$ and $\log u$ satisfy \eqr{e:logu22},  Corollary \ref{c:self} gives
\begin{equation}	\label{e:fromself1}
	\left[ \langle \nabla v^2 ,  \nabla \log u \rangle \right] = - \left[ v^2 \, \cL \, \log u \right]
	= \left[ v^2 \left( \mu_1 +  |A|^2 + \frac{1}{2} + |\nabla \log u|^2 \right) \right] \, ,
\end{equation}
where the second equality used the equation for $\cL \, \log u$.
  Similarly, \eqr{e:vw22} allows us to apply
Corollary \ref{c:self} to two copies of $v$ and get
\begin{equation}	\label{e:fromself2}
	\left[ |\nabla v|^2 \right] = - \left[ v  \, \cL \,  v \right]
	= \left[ v^2 \left( \mu_1 +  |A|^2 + \frac{1}{2}   \right) \right] \, ,
\end{equation}
where the second equality used the equation \eqr{e:cLva} for $\cL v$.
 Substituting \eqr{e:fromself2} into \eqr{e:fromself1}   and subtracting the left-hand side gives
\begin{equation}	\label{e:fromself3}
	0 = \left[ v^2 \,  |\nabla \log u|^2 - 2  \langle v \, \nabla \log u , \nabla v \rangle + |\nabla v|^2  \right]
	= \left[ \left| v \, \nabla \log u - \nabla v \right|^2    \right] \, .
\end{equation}
We conclude that
\begin{equation}
	v \, \nabla \log u - \nabla v = 0
\end{equation}
 which implies that $\frac{v}{u}$ is constant.
 \end{proof}

We have already seen that the mean curvature $H$ is an eigenfunction of $L$ with eigenvalue $-1$.
The next theorem shows that if $H$ changes, then the bottom of the spectrum $\mu_1$ is strictly less than $-1$.

 \begin{Thm}	\label{c:pde}
 If the mean curvature $H$ changes sign, then $\mu_1 < -1$.
 \end{Thm}

 \begin{proof}
Clearly, we may assume that $\mu_1 \ne - \infty$.

{\bf{Step 1: $H$, $|\nabla H|$, and $|A| \, |H|$ are in the weighted $L^2$ space}}.
The polynomial volume growth of $\Sigma$ and the self-shrinker equation
$H = \frac{ \langle x , \nn \rangle}{2}$ immediately yield that $H$ is in the weighted $L^2$ space, so we must show that $|\nabla H|$ and $|H| \, |A|$ are also
in the weighted $L^2$ space.

Since $\mu_1\ne - \infty$, the first part of Lemma \ref{l:pde} gives a positive function $u$ with $L \, u = - \mu_1 \, u$.
Using this $u$ in Part (2) of Lemma \ref{l:L2W12} gives
\begin{equation}
	 \int_{\Sigma}  \phi^2 \,     |A|^2    \, \e^{- \frac{|x|^2}{4} }
	  \leq \int_{\Sigma} \left(  4 \, |\nabla \phi  |^2 - 2 \, \mu_1 \, \phi^2 \right) \,  \e^{- \frac{|x|^2}{4} }  \, ,
\end{equation}
where $\phi$ is any Lipschitz function with compact support.  If we  take  $\phi$ to be $|x|$ on $B_R$ and cut it off linearly between $\partial B_{R}$ and $\partial B_{2R}$, so that  $|\nabla \phi| \leq 1$ and $\phi^2 \leq |x|^2$, then
we get
\begin{equation}	\label{e:x2a2bd}
	 \int_{B_R \cap \Sigma}      |A|^2  \, |x|^2  \, \e^{- \frac{|x|^2}{4} }
	  \leq \int_{\Sigma} \left(  4 +  2 \, |\mu_1 | \, |x|^2 \right) \,  \e^{- \frac{|x|^2}{4} }  \, .
\end{equation}
The right-hand side of \eqr{e:x2a2bd} is independent of $R$ and is finite
since $\Sigma$ has polynomial volume growth.  Thus, taking $R \to \infty$ and applying the monotone convergence theorem implies that
\begin{equation}	\label{e:x2a2bd2}
	 \int_{ \Sigma}      |A|^2  \, |x|^2  \, \e^{- \frac{|x|^2}{4} }
	  \leq \int_{\Sigma} \left(  4 +  2 \, |\mu_1 | \, |x|^2 \right) \,  \e^{- \frac{|x|^2}{4} } < \infty  \, .
\end{equation}
Using the self-shrinker equation  $H = \frac{ \langle x , \nn \rangle}{2}$, we showed in \eqr{e:jc2} that
$2 \, \nabla_{e_i}  \, H=   - a_{ij} \, \langle x , e_j \rangle$ so that
\begin{equation}
4 \, |\nabla H|^2 \leq |A|^2 \, |x|^2 {\text{ and }} |A|^2 \, H^2 \leq |A|^2 \, |x|^2 \, .
\end{equation}
  Consequently, \eqr{e:x2a2bd2} implies that $|\nabla H|$ and $|A| \, |H|$ are in the weighted $L^2$ space as desired.

{\bf{Step 2: $\mu_1 \leq -1$.}}
We will show this by using $H$ as a test function in the definition of $\mu_1$; this is allowed by Lemma \ref{l:infover} since we showed in step $1$ that $H$, $|\nabla H|$, and $|A| \, |H|$ are in the weighted $L^2$ space.   We will need that,
by Theorem \ref{t:spectral},
\begin{equation}
	 \cL \, H + \left(|A|^2 + \frac{1}{2} \right) \, H = L \, H = H \, ,
\end{equation}
 so that $H \, \cL \, H$ is in the weighted $L^1$ space.  Consequently,
 we can apply
Corollary \ref{c:self} to two copies of $H$ to get
\begin{equation}	\label{e:fromself2a}
	\int_{\Sigma}  |\nabla H|^2  \, \e^{- \frac{|x|^2}{4} }   = - \int_{\Sigma}  ( H \, \cL H ) \, \e^{- \frac{|x|^2}{4} }
	=  \int_{\Sigma} H^2 \left(  \frac{1}{2} - |A|^2  \right)   \, \e^{- \frac{|x|^2}{4} }  \, ,
\end{equation}
 so we get that $\mu_1 \leq - 1$ as claimed.

{\bf{Step 3: $\mu_1 \ne -1$.}}  If we did have that $\mu_1 = -1$, then the positive function $u$ in the first step and $H$ would both be eigenfunctions for $L$ with the same eigenvalue $\mu_1 = -1$.  Since we have shown that $H$ is  in the weighted $W^{1,2}$ space, the second part of Lemma \ref{l:pde} gives that $H$ must be a constant multiple of $u$.  However, this is impossible since $u$ is positive and $H$ changes sign.

 \end{proof}

  \subsection{Constructing unstable variations when $\mu_1 < -1$}

We have shown that if $H$ changes sign, then  $\Sigma$ has $\mu_1 < -1$.  When $\Sigma$ was closed, it followed immediately from this and  the orthogonality of eigenfunctions with different eigenvalues that $\Sigma$ was unstable.  This orthogonality uses an integration by parts  which is  not justified when $\Sigma$ is open.  Instead, we will show that the lowest eigenfunction on a sufficiently large ball is {\emph{almost}} orthogonal to $H$ and the translations.  This will be enough to prove  instability.

    In this subsection,   we will use
  square brackets $\left[ \cdot \right]_X$ to denote a  weighted integral over a set $X \subset \RR^{n+1}$
 \begin{equation}
 	\left[ f \right]_X = \int_{X \cap \Sigma} f \, \e^{- \frac{|x|^2}{4} } \, d\mu_{X \cap \Sigma} \, .
\end{equation}
  We will use this notation when $X$ is a ball $B_r$,
 the boundary $\partial B_r$ of a ball, or an annulus $A_{r_1 , r_2} = B_{r_2} \setminus B_{r_1}$.

     \vskip2mm
 The next lemma shows that two eigenfunctions of $L$ with different eigenvalues
 are almost orthogonal when one of them vanishes on the boundary of a ball, the
 other is small in a neighborhood of the boundary, and the radius is large.

   \begin{Lem}		\label{l:nearbd}
   Suppose that $\mu_1 (B_R) \in [-3/2 , -1)$ for some $R>3$,  $L u = - \mu_1(B_R) \, u$ on $B_R$,   $u$ is zero on $\partial B_R$, and $u$ is not identically zero. If
   $Lv = \mu \, v$ and  $\mu \geq -1$, then
      \begin{equation}	\label{e:nearbd}
   	\frac{  \left| \left[ v \, u   \right]_{B_R} \right| }{ \left[ u^2 \right]^{1/2}_{B_R}}
	  \leq
	  \frac{-4}{\mu_1(B_R) + 1} \, \left(  \frac{  \sqrt{7 \, \Vol (B_R \cap \Sigma)}}{
	 	  \e^{   \frac{(R-3)^2}{8}}} \,  \max_{ B_R} |v|
	+
	 \left[  |\nabla v|^2 \right]^{1/2}_{A_{R-2,R-1}} \right)
	    +   \left[ v^2 \right]^{1/2}_{A_{R-2,R}}\, .
   \end{equation}
      \end{Lem}

   \begin{proof}
 If $v$ also vanished on $\partial B_R$, then an integration by parts argument would show that $u$ and $v$
 are orthogonal.  The key for the integration by parts is that if
  $\cL = \Delta   - \frac{1}{2} \, x \cdot \nabla$ is the operator from Lemma \ref{l:self}, then
   \begin{equation}
   	v \, \cL u - u \, \cL v = v \, L u - u \, L v = (-\mu_1(B_R) + \mu) \, v \, u  \, .
\end{equation}
Since
   $\dv \, \left( (v \, \nabla u - u \nabla v) \, \e^{- \frac{|x|^2}{4} }  \right) = \e^{- \frac{|x|^2}{4} } \, \left( v \, \cL u - u \, \cL  v \right)$,
   Stokes' theorem gives for   $r \leq R $
   \begin{equation}	\label{e:fromstokes}
   	(-\mu_1 (B_R) + \mu) \, \left[ v \, u   \right]_{B_r} = \left[
	  (v \, \partial_r u - u \, \partial_r v) \right]_{\partial B_r} \, .
   \end{equation}
   We cannot control the boundary term when $r=R$, but we will be able to use the co-area formula to bound it for $r$ near $R$ (this is where the term in the big brackets in \eqr{e:nearbd} comes from); the remainder is then bounded by the last term in \eqr{e:nearbd}.

 We will need a $L^1$ bound on $|\nabla u|$ on annuli.  To get this, choose a cutoff function $\phi$ which is
 one on $A_{R-2,R-1}$ and cuts off linearly to be zero on $\partial B_R$ and $\partial B_{R-3}$ and then apply (2) in Lemma \ref{l:L2W12} to get
 \begin{equation}
 	\left[ \frac{ |\nabla u|^2}{u^2} \right]_{A_{R-2,R-1} } \leq   \left[ 7 \right]_{A_{R-3,R}} \leq 7\, \Vol (B_R \cap \Sigma) \, \e^{ - \frac{(R-3)^2}{4}} \, .
 \end{equation}
Combining this with the Cauchy-Schwarz inequality gives
\begin{equation}	\label{e:L1br}
	\left[   |\nabla u|  \right]^2_{A_{R-2,R-1} } = \left[ \frac{ |\nabla u| }{u }  \, u  \right]^2_{A_{R-2,R-1} }
	  \leq 7\, \Vol (B_R \cap \Sigma) \, \e^{ - \frac{(R-3)^2}{4}} \, \left[  u^2  \right]_{B_{R} } \, .
\end{equation}
 If we let $\alpha$   denote the positive number $(  -\mu_1(B_R)  -1)$, then \eqr{e:fromstokes} gives
    \begin{equation}	\label{e:alphae}
   	\alpha \, \left| \left[ v \, u   \right]_{B_r} \right|  \leq  \max_{\partial B_r} |v| \,  \left[ |\nabla u| \right]_{\partial B_r} + \left[ u^2 \right]^{1/2}_{\partial B_r} \,
	  \left[ |\nabla v|^2 \right]^{1/2}_{\partial B_r}  \, .
   \end{equation}
   Given a non-negative $L^1$ function $f$ on $A_{R-2,R-1} \cap \Sigma$, it follows from the coarea formula that
   \begin{equation}
   	\int_{R-2}^{R-1} \, \left[ f \right]_{\partial B_r} \, dr = \left[  f \, |\nabla |x|| \right]_{A_{R-2, R-1}} \leq
		\left[  f   \right]_{A_{R-2, R-1}} \, ,
   \end{equation}
   so we conclude that there is a measurable set $I_{f} \subset [R-2,R-1]$ with Lebesgue measure at most $1/4$ so that if $r \notin I_f$, then
     \begin{equation}
   	  \left[ f \right]_{\partial B_r}    \leq 4 \,
		\left[  f   \right]_{A_{R-2, R-1}} \, .
   \end{equation}
   Applying this with $f$ equal to $|\nabla u|$, $f$ equal to $|\nabla v|^2$, and $f$ equal to $u^2$,
    we conclude that at least one quarter of the  $r$'s in $[R-2 , R-1]$  simultaneously satisfy
    \begin{align}
   	\left[ |\nabla u| \right]_{\partial B_r} &\leq 4 \, \left[ |\nabla u| \right]_{A_{R-2,R-1}} \, , \\
   	\left[ |\nabla v|^2 \right]_{\partial B_r} &\leq 4 \, \left[  |\nabla v|^2 \right]_{A_{R-2,R-1}} \, , \\
	\left[ u^2 \right]_{\partial B_r} &\leq 4 \, \left[  u^2 \right]_{A_{R-2,R-1}} \leq 4 \, \left[ u^2 \right]_{B_R} \, .
   \end{align}
   Fix one such $r$.
   Substituting these into \eqr{e:alphae} and using \eqr{e:L1br} on the $|\nabla u|$ integral
    gives
    \begin{align}
   	 \alpha \, \left| \left[ v \, u   \right]_{B_r} \right|   &\leq
	 4 \,   \max_{ B_R} |v|   \, \left[ |\nabla u| \right]_{A_{R-2,R-1}}
	+ 4 \,
	  \left[  u^2 \right]^{1/2}_{B_R} \, \left[  |\nabla v|^2 \right]^{1/2}_{A_{R-2,R-1}}
	 \notag \\
	 &\leq  4 \, \left(    \sqrt{7 \, \Vol (B_R \cap \Sigma)} \,
	 	  \e^{ - \frac{(R-3)^2}{8}} \,  \max_{ B_R} |v|
	+
	 \left[  |\nabla v|^2 \right]^{1/2}_{A_{R-2,R-1}} \right)  \, \left[  u^2 \right]^{1/2}_{B_R}
	    \end{align}
  The lemma follows from this since
   \begin{equation}
   	 \left| \left[ v \, u   \right]_{B_R} \right|   \leq   \left| \left[ v \, u   \right]_{B_r} \right| +   \left[ |v \, u|  \right]_{A_{R-2,R}}
	 \leq \left| \left[ v \, u   \right]_{B_r} \right| +  \left[  u^2   \right]^{1/2}_{B_r} \,  \left[ v^2 \right]^{1/2}_{A_{R-2,R}} \, .
   \end{equation}
   \end{proof}

   The next lemma gives the desired instability.

   \begin{Lem}	\label{l:ortho}
   If $\mu_1 < -1$, then there exists $\bar{R}$ so that if $R \geq \bar{R}$ and $u$ is a Dirichlet eigenfunction for $\mu_1 (B_R)$, then
   for any $h \in \RR$ and any $y \in \RR^{n+1}$ we have
   \begin{equation}	\label{e:ortho}
   	\left[  -u \, L \, u + 2 u \, h \, H + u \,
	 \langle y , \nn \rangle
		    -   h^2 \, H^2
		  - \frac{\langle y , \nn \rangle^2}{2}  \right]_{B_R} < 0 \, .
   \end{equation}
   \end{Lem}

   \begin{proof}
   Using the Cauchy-Schwartz inequality $ ab \leq \frac{1}{2} \, (a^2 + b^2)$ on the cross-term
  $ u \,
	 \langle y , \nn \rangle $,  the left hand side of \eqr{e:ortho} is bounded from above by
\begin{equation}	     	\label{e:ortho2}
	 \left[ \left( \frac{1}{2} + \mu_1(B_R) \right) \, u^2 + 2 u \, h H
		    -   h^2 \, H^2
		\right]_{B_R} 	  \, .
\end{equation}
We will show that this is  negative when $R$ is   large.
There are two cases depending on $\mu_1$.

   Suppose first that
   $\mu_1 < -\frac{3}{2} $ and choose $\bar{R}$ so that $\mu_1 (B_{\bar{R}}) < - \frac{3}{2}$.
 Given any $R \geq \bar{R}$, then \eqr{e:ortho2} is strictly less than
   \begin{equation}	
	 \left[ -  u^2 + 2 u \, h H
		    -   h^2 \, H^2
		\right]_{B_R}  =
	- \left[ 	 \left(    u  - h H \right)^2
		\right]_{B_R}	  \, ,
\end{equation}
which  gives \eqr{e:ortho} in this case.

Suppose now that $- \frac{3}{2} \leq \mu_1 < -1 $.  In this case, we will take $\bar{R}$ to be the maximum of $\bar{R}_1$ and $\bar{R}_2$, where   $\bar{R}_1$ is chosen so that $\mu(B_R) < -1$ for all $R \geq \bar{R}_1$ and $\bar{R}_2$ will be chosen to handle the cross-term.  Precisely, we will prove that there is some $\bar{R}_2$ so that for all $R \geq \bar{R}_2$
\begin{equation}	\label{e:claimuH}
	  \left[ u \, H \right]^2_{B_R}
	  \leq \frac{1}{2} \, \left[   H^2 \right]_{B_R} \, \left[ u^2 \right]_{B_R} \, .
\end{equation}
It is then easy to see that  \eqr{e:ortho2} is negative when $R \geq \max \{ \bar{R}_1 , \bar{R}_2 \}$.
Namely, using $\mu_1 (B_R) < -1$ and  \eqr{e:claimuH}  in \eqr{e:ortho2} gives
\begin{align}	
	 \left[ \left( \frac{1}{2} + \mu_1(B_R) \right) \, u^2 + 2 u \, h H
		    -   h^2 \, H^2
		\right]_{B_R} 	&<
		  - \frac{1}{2}  \left[ u^2 \right]_{B_R}
		  + \sqrt{2} \, h \, \left[  u^2 \right]^{1/2}_{B_R}  \, \left[ H^2 \right]^{1/2}_{B_R}
		    -   h^2 \, \left[ H^2
		\right]_{B_R}  \notag \\
		&= - \left(  \frac{1}{\sqrt{2}} \,  \left[ u^2 \right]^{1/2}_{B_R} - h \,  \left[ H^2 \right]^{1/2}_{B_R}
		\right)^2
		  \, .
\end{align}

It remains to prove the existence of $\bar{R}_2$ so that the claim \eqr{e:claimuH} holds for all $R \geq \bar{R}_2$.
Clearly, we may  assume that $H$ is not identically zero, so there is some   $H_0 > 0$ with
\begin{equation}	\label{e:H0low}
	 \left[   H^2 \right]_{B_R} \geq H_0 {\text{ for all }} R > H_0^{-1} \, .
\end{equation}
Applying Lemma \ref{l:nearbd} with $v=H$ and $\mu = - 1$ gives that
      \begin{equation}	\label{e:mygf}
	\frac{   \left[ H \, u   \right]^2_{B_R}  }{ \left[ u^2 \right]_{B_R}} \leq
	 2 \, \left( \frac{-4}{\mu_1 (B_R) + 1} \right)^2 \, \left(  \frac{  (7\, C \, R^d )^{1/2} \,  R }{
	 	  \e^{   \frac{(R-3)^2}{8}}}
	+
	R \, \left[  |A|^2   \right]^{1/2}_{A_{R-2,R-1}} \right)^2
	    + 2\,   \left[ R^2 \right]_{A_{R-2,R}}\, ,
   \end{equation}
   where we substituted in the bound  $|H| \leq |x|$, the gradient bound $|\nabla H| \leq |A| \, |x|$ which follows from   \eqr{e:jc2},
    and the polynomial volume bound $\Vol (B_R \cap \Sigma) \leq C \, R^d$ (which holds for some $C$ and $d$ and all $R \geq 1$).  To bound the $|A|$ term,
   choose a cutoff function $\phi$ which is
 one on $A_{R-2,R-1}$ and cuts off linearly to be zero on $\partial B_R$ and $\partial B_{R-3}$ and then apply (2) in Lemma \ref{l:L2W12} to get
 \begin{equation}	\label{2abd}
 	\left[   |A|^2  \right]_{A_{R-2,R-1} } \leq   \left[ 7 \right]_{A_{R-3,R}} \leq 7\, \Vol (B_R \cap \Sigma) \, \e^{ - \frac{(R-3)^2}{4}} \leq 7 \, C \, R^d \, \e^{ - \frac{(R-3)^2}{4}}  \, .
 \end{equation}
 It follows that all of the terms in \eqr{e:mygf} decay exponentially  and, thus, that
 $ \frac{   \left[ H \, u   \right]^2_{B_R}  }{ \left[ u^2 \right]_{B_R}}
$ can be made as small as we like by taking $\bar{R}_2$ sufficiently large.  Since $ \left[   H^2 \right]_{B_R}$ has a uniform positive lower bound by
\eqr{e:H0low}, we can choose $\bar{R}_2$ so that \eqr{e:claimuH} holds for all $R \geq \bar{R}_2$.
\end{proof}

\section{The classification of self-shrinkers with $H\geq 0$}	\label{s:huisken}

In \cite{H3},   Huisken showed  that the only smooth closed  self-shrinkers with non-negative mean curvature in $\RR^{n+1}$ (for $n > 1$) are round spheres (i.e., $\SS^n$).  In the remaining case $n=1$,   Abresch and Langer, \cite{AbLa},  classified all smooth closed self-shrinking curves in $\RR^2$ and showed that the   {\emph{embedded}} ones   are round circles.

In the non-compact case, Huisken showed in \cite{H4} that  the only smooth open embedded self-shrinkers  in $\RR^{n+1}$
with $H \geq 0$, polynomial volume growth, and $|A|$ bounded are
 isometric products of a round sphere and a linear subspace (i.e. $\SS^k\times \RR^{n-k}\subset \RR^{n+1}$).  The assumption of a bound on $|A|$    was  natural in \cite{H4} since Huisken was classifying tangent flows associated to type I singularities and these automatically satisfy such a bound.  However, for our applications to generic flows we do not have an a priori bound on $|A|$ and, thus, we are led to show that
 Huisken's classification holds even without the $|A|$ bound.  This is the next theorem  (Theorem \ref{t:huisken} from the introduction):

 \begin{Thm}	\label{t:huisken2}
 $\SS^k\times \RR^{n-k}$   are the only   smooth complete embedded self-shrinkers without boundary, with polynomial volume growth, and   $H \geq 0$ in $\RR^{n+1}$.
  \end{Thm}

  The $\SS^k$ factor in Theorem \ref{t:huisken2} is round and has radius $\sqrt{2k}$; we allow the possibilities of a hyperplane (i.e., $k=0$) or a sphere ($n-k = 0$).

 \subsection{Two general inequalities  bounding $|\nabla |A||$ by $|\nabla A|$}

 The next lemma holds for all hypersurfaces:

\begin{Lem}	\label{l:csA}
If we fix a point $p$ and choose a frame
$e_i$,
$i=1,\dots,n$, so that $a_{ij}$ is in diagonal form at $p$, that is,  $a_{ij}=\lambda_i\,\delta_{ij}$,
then we have at $p$ that
\begin{equation}        \label{e:csA}
        |\nabla |A| |^2\leq \sum_{i,k} a_{ii,k}^2 \leq \sum_{i,j,k} a_{ij,k}^2 = |\nabla A|^2 \, .
\end{equation}

Moreover,
\begin{equation}  \label{e:linalg}
	 \left(1+\frac{2}{n+1} \right)\,|\nabla |A||^2
         \leq  |\nabla A|^2  + \frac{2n}{n+1} \, |\nabla H|^2  \, .
\end{equation}
\end{Lem}

\begin{proof}
Observe first that, since $a_{ij}$ is symmetric, we may choose $e_i$,
$i=1,\dots,n$, so that $a_{ij}$ is in diagonal form at $p$.
  Since $\nabla |A|^2 = 2 \, |A| \, \nabla |A|$, we have at $p$
\begin{equation}   \label{e:ecss1}
        4\,|A|^2\,|\nabla |A||^2
        =\sum_{k=1}^{n} \left( (\sum_{i,j=1}^{n} a_{ij}^2) _k \right) ^2
 = 4\,\sum_{k=1}^{n} \left( \sum_{i=1}^{n}
        a_{ii,k}\,\lambda_{i} \right) ^2\leq 4\,|A|^2\sum_{i,k=1}^{n} a_{ii,k}^2 \,
        ,
 \end{equation}
where the second  equality used that $a_{ij}$ is in diagonal form at $p$ and the inequality
used  the Cauchy-Schwarz inequality on the vectors $(a_{11,k}, \dots ,
a_{nn,k})$ and $(\lambda_1 , \dots , \lambda_{n})$.   This proves the first inequality in \eqr{e:csA}; the second inequality follows trivially.

To prove \eqr{e:linalg} observe that by   \eqr{e:csA} and the definition $H = - \sum a_{jj}$, we have
\begin{align}   \label{e:ecss3}
        |\nabla |A||^2 &\leq \sum_{i,k=1}^{n} a_{ii,k}^2
        =\sum_{i\ne k} a_{ii,k}^2+\sum_{i=1}^{n} a_{ii,i}^2
         = \sum_{i\ne k} a_{ii,k}^2+\sum_{i=1}^{n}
                \left(H_i +  \sum_{ \{ j \, | \, i\ne j \} } a_{jj,i}\right) ^2  \notag
\\
        &\leq \sum_{i\ne k} a_{ii,k}^2+ n \, \sum_{i=1}^{n}
               \left( H_i^2 +  \sum_{ \{ j \, | \, i\ne j \} } a_{jj,i}^2 \right)
        =n \, |\nabla H|^2 + (n+1)\sum_{i\ne k} a_{ii,k}^2 \\
        &=n \, |\nabla H|^2 + (n+1) \, \sum_{i\ne k} a_{ik,i}^2
        = n \, |\nabla H|^2 + \frac{n+1}{2}\,\left( \sum_{i\ne k} a_{ik,i}^2
        +\sum_{i\ne k} a_{ki,i}^2\right)\, .\notag
\end{align}
The second inequality in \eqr{e:ecss3} used the algebraic fact
$
    \left( \sum_{j=1}^{n} b_j \right)^2 \leq n \, \sum_{j=1}^{n}
    b_j^2 $ and the last two equalities used the Codazzi equations (i.e., $a_{ik,j} = a_{ij,k}$; see, e.g., (B.5) in \cite{Si}).
 Multiplying  \eqr{e:ecss3} by $\frac{2}{n+1}$ and adding this to the first inequality in \eqr{e:csA} gives
\begin{align}   \label{e:ecss4}
        (1+\frac{2}{n+1})\,|\nabla |A||^2
        &\leq \sum_{i,k=1}^{n}a_{ii,k}^2
        +\sum_{i\ne k}a_{ik,i}^2
        +\sum_{i\ne k} a_{ki,i}^2  + \frac{2n}{n+1} \, |\nabla H|^2  \notag
\\
        &\leq \sum_{i,j,k=1}^{n} a_{ij,k}^2 + \frac{2n}{n+1} \, |\nabla H|^2\, ,
\end{align}
which completes the proof.
 \end{proof}

 \subsection{Simons' inequality for $|A|$}
  We will assume throughout the remainder of this section  that $\Sigma \subset \RR^{n+1}$ is a smooth complete embedded self-shrinker without boundary and with polynomial volume growth.

The next lemma computes the  operator $ L= \Delta + |A|^2 + \frac{1}{2} - \frac{1}{2} \, x \cdot \nabla$ on the second fundamental form; this is
analogous to Simons' equation for minimal hypersurfaces (see, e.g., section $2.1$ in \cite{CM4}).

 \begin{Lem}	\label{l:simonsA}
If we extend the operator $L$ to tensors, then
$L \, A = A$.  In particular, if $|A|$ does not vanish, then
\begin{equation}	\label{e:LmodA}
	L \, |A| = |A|  +  \frac{\left| \nabla A \right|^2 - \left| \nabla |A| \right|^2}{ |A|}
	 \geq |A|  \, .
\end{equation}
\end{Lem}

\begin{proof}
 For a general hypersurface (without using the self-shrinker equation),   the Laplacian of $A$ is (see, e.g., lemma B.8 in \cite{Si} where the sign convention for the $a_{ij}$'s is reversed)
 \begin{equation}	\label{e:simsim}
 	  \left( \Delta  A \right)_{ij}  = -    |A|^2 \, a_{ij} - H \, a_{ik} \, a_{jk} - H_{ij} \, .
 \end{equation}
 Here $H_{ij}$ is the $ij$ component of the Hessian of $H$.
 Working at a point $p$ and choosing the frame $e_i$ so that $\nabla_{e_i}^T e_j (p) = 0$,
 it follows from \eqr{e:jc3} and \eqr{e:forlast} that
\begin{equation}	\label{e:jc3forlast}
	2 \,    H_{ij} =
	 - a_{ik,j} \, \langle x , e_k \rangle
	  - a_{ij}
	 - 2 \, Ha_{ik}a_{kj}   \, .
\end{equation}
Substituting this into \eqr{e:simsim} and using Codazzi's equation (i.e., $a_{ik,j} = a_{ij,k}$), we get at $p$
 \begin{equation}	\label{e:simsim2}
 	  \left( \Delta  A \right)_{ij}  = -    |A|^2 \, a_{ij}
	  + a_{ij,k} \, \langle \frac{x}{2} , e_k \rangle
	  + \frac{1}{2} \, a_{ij}
	   \, ,
 \end{equation}
 so that at $p$ we have $L \, A = A$.  Since $p$ is arbitrary, this holds everywhere on $\Sigma$.

To get the second claim, first compute
  (without using the above equation  for $A$) that
\begin{align}
	L \, |A| &= L \, \left( |A|^2 \right)^{1/2} =   \frac{ \Delta |A|^2 }{2 \, |A|} -  \frac{ \left| \nabla |A|^2 \right|^2}{4 \, |A|^3} +
	\left( |A|^2 + \frac{1}{2} \right) \, |A| - \frac{1}{4} \, \langle x , \frac{\nabla |A|^2}{|A|} \rangle \notag \\
	&=   \frac{ \langle A , \Delta  A \rangle + |\nabla A|^2}{  |A|} -  \frac{\left| \nabla |A|^2 \right|^2}{4 \, |A|^3} +
	\left( |A|^2 + \frac{1}{2} \right) \, |A| - \frac{ \langle A ,   \nabla_{\frac{x}{2}}   A   \rangle}{|A|} \\
	&= \frac{ \langle A , L \, A \rangle}{|A|} +  \frac{ \left| \nabla A \right|^2 - \left| \nabla |A| \right|^2}{|A|}
	 \, . \notag
\end{align}
This is well-defined and smooth since $|A| \ne 0$.    Substituting $L \, A = A$  into this and using \eqr{e:csA} gives \eqr{e:LmodA}.
\end{proof}

\subsection{Weighted estimates for $|A|$ and a uniqueness result}

In this subsection,  square brackets $\left[ \cdot \right]$ denote the weighted integrals
 $\left[ f \right] = \int_{\Sigma} f \, \e^{- \frac{|x|^2}{4} } \, d\mu$.   We will   use the operator
  $\cL = \Delta   - \frac{1}{2} \, x \cdot \nabla$ from Lemma \ref{l:self} and Corollary \ref{c:self}.

The next proposition adapts an estimate of Schoen, Simon, and Yau, \cite{ScSiY},  for stable minimal
hypersurfaces to  self-shrinkers with positive mean curvature.  The idea is to carefully play two inequalities off of each other: a stability   type inequality with a cutoff function in terms of $|A|$ and a Simons' type inequality (see subsection $2.3$ in \cite{CM4} for the minimal case).  The stability type inequality in this case will come from that $H > 0$ is an eigenfunction for $L$.

\begin{Pro}	\label{p:ssy}
If $H> 0$, then $\left[ |A|^2 + |A|^4 + |\nabla |A||^2 + |\nabla A|^2 \right] < \infty$.
\end{Pro}

\begin{proof}
First, since $H > 0$,
 $\log H$ is well-defined and Theorem \ref{t:spectral} gives
  \begin{equation}	\label{e:cLlogH}
	\cL \log H  =
	- |\nabla \log H|^2 + \frac{\Delta H  - \frac{1}{2} \, \langle x , \nabla H \rangle}{H}
	 =   \frac{1}{2}  -  |A|^2 - |\nabla \log  H |^2 \, .
\end{equation}
  Given any compactly supported function $\phi$, self-adjointness of $\cL $ (Lemma \ref{l:self}) gives
  \begin{equation}	\label{e:cLloga}
  	\left[ \langle \nabla \phi^2 , \nabla \log H \rangle \right] = - \left[ \phi^2 \, \cL \log H \right] 	=
	  \left[ \phi^2 \, \left(    |A|^2 -  \frac{1}{2} + |\nabla \log H|^2 \right) \right] \, .
  \end{equation}
  Combining this with the  inequality
  $| \langle \nabla \phi^2 , \nabla \log H  \rangle|   \leq   | \nabla  \phi |^2 +
   \phi^2 \, |\nabla \log H |^2$ gives the ``stability inequality''
    \begin{equation}	\label{e:stab1}
	  \left[ \phi^2 \,      |A|^2 \right] \leq \left[ |\nabla \phi|^2 +  \frac{1}{2} \, \phi^2 \right] \, .
  \end{equation}
  We will apply this with $\phi = \eta \, |A|$ where $\eta \geq 0$ has compact support to get
     \begin{align}	\label{e:stab2}
	  \left[ \eta^2 \,      |A|^4 \right]
	  &\leq \left[ \eta^2 \, | \nabla |A| |^2 + 2 \eta \, |A| \,  |\nabla |A|| \, |\nabla \eta|
	+  |A|^2 \, |\nabla \eta |^2  +
	    \frac{1}{2} \,   \eta^2 \, |A|^2 \right]  \notag \\
	    &\leq (1 + \epsilon) \, \left[  \eta^2 \, | \nabla |A| |^2  \right]
	+ \left[  |A|^2 \,  \left(  \left( 1 + \frac{1}{\epsilon} \right) \, |\nabla \eta |^2   +
	    \frac{1}{2} \,  \eta^2  \right) \right]   \, ,
  \end{align}
  where $\epsilon > 0$ is arbitrary and the last inequality used the absorbing inequality
   $2ab \leq \epsilon a^2 + \frac{b^2}{\epsilon}$.

Second, combining the Leibniz rule $\cL f^2 = 2 |\nabla f|^2 + 2 f \, \cL f$ for a general function $f$, the equation
$\cL =  L - |A|^2-\frac{1}{2}$,  the first equality in \eqr{e:LmodA},
 and
\eqr{e:linalg} gives
 the inequality
\begin{align}	\label{e:oneiwant}
\cL\,|A|^2
&= 2 \, |\nabla |A||^2 + 2 \, |A| \, \cL |A|  =
2 \, |\nabla |A||^2 + 2 \, |A| \,  \left( L |A|
 -   |A|^3- \frac{1}{2}  |A| \right) \notag \\
 &=2 |\nabla A |^2 + |A|^2 - 2\,  |A|^4
\notag \\
			&\geq  2 \, \left(1+\frac{2}{n+1} \right)\,|\nabla |A||^2 -
		\frac{4n}{n+1} \, |\nabla H|^2 + |A|^2 - 2\,  |A|^4 \, .
\end{align}
Integrating this against $\frac{1}{2} \, \eta^2$ and using the self-adjointness of $\cL $ (Lemma \ref{l:self}) gives
\begin{equation}	
	- 2 \,  \left[ \langle \eta \, |A| \,  \nabla \eta , \nabla   |A| \rangle \right]
	\geq   \left[   \eta^2 \, \left(1+\frac{2}{n+1} \right)\,|\nabla |A||^2 -
		\frac{2n}{n+1} \, \eta^2 \,  |\nabla H|^2   -    \eta^2 \, |A|^4 \right] \, ,
\end{equation}
where we dropped the $|A|^2 \eta^2$ term that only helped us. Using the absorbing inequality
$2ab \leq \epsilon a^2 + \frac{b^2}{\epsilon}$ gives
\begin{equation}	\label{e:intsimo}
	\left[      \eta^2 \, |A|^4 \right] +  \left[ \frac{2n}{n+1} \, \eta^2 \,  |\nabla H|^2    + \frac{1}{\epsilon} \, |A|^2 \, |\nabla \eta|^2 \right]
	\geq  \left(1+\frac{2}{n+1} - \epsilon \right)\,  \left[   \eta^2 \, |\nabla |A||^2 \right]  \, ,
\end{equation}

We will assume that $|\eta| \leq 1$ and $|\nabla \eta|\leq 1$.
Combining   \eqr{e:stab2} and \eqr{e:intsimo} gives
  \begin{equation}	\label{e:stab22}
	  \left[ \eta^2 \,      |A|^4 \right]
	   	    \leq  \frac{1 + \epsilon}{ 1+\frac{2}{n+1} - \epsilon} \,
		    \left[      \eta^2 \, |A|^4 \right] +  C_{\epsilon} \,  \left[    |\nabla H|^2    +
		     |A|^2   \right]    \, ,
  \end{equation}
  where the constant $C_{\epsilon}$ depends only on $\epsilon$.
  Choosing $\epsilon > 0$ small so that $\frac{1 + \epsilon}{ 1+\frac{2}{n+1} - \epsilon} < 1$ (i.e., $\epsilon < \frac{1}{n+1}$), we can absorb the  $\left[ \eta^2 \,      |A|^4 \right] $ term on the right to get
    \begin{equation}	\label{e:stab23}
	  \left[ \eta^2 \,      |A|^4 \right]
	   	    \leq     C  \,  \left[    |\nabla H|^2    +
		     |A|^2   \right]  \leq  C \, \left[ |A|^2 (1 + |x|^2) \right]
		       \, ,
  \end{equation}
  for some constant $C$ depending only on $n$ where the last inequality used
  the gradient bound $|\nabla H| \leq |A| \, |x|$ which follows from   \eqr{e:jc2}.  Since $H>0$, it follows from
Part (2) of Lemma \ref{l:L2W12} and the polynomial volume growth
that $\left[ |A|^2 (1 + |x|^2) \right] < \infty$, so \eqr{e:stab23} and the monotone convergence theorem  give that
 $ \left[  |A|^4 \right] < \infty$.

The bound $\left[ |\nabla |A||^2 \right] < \infty$ follows immediately from  \eqr{e:intsimo},
  $ \left[  |A|^4 \right] < \infty$,  and the monotone convergence theorem,
  so it only remains to show that $\left[ |\nabla A|^2 \right] < \infty$.  To do this, integrate the  equality in \eqr{e:oneiwant}
		against $\frac{1}{2} \, \eta^2$ and use the self-adjointness of $\cL $ (Lemma \ref{l:self}) to get
\begin{equation}	
	\left[ \eta^2 ( |\nabla A |^2   -    |A|^4 ) \right] \leq
	 2 \,  \left[    |A| \,   | \nabla   |A| | \right]  \leq   \left[    |A|^2 +  | \nabla   |A| |^2  \right] \, .
\end{equation}
Here we used again that $|\eta| \leq 1$ and $|\nabla \eta| \leq 1$.
Since $\left[ |A|^2 + |\nabla |A||^2 + |A|^4 \right] < \infty$, the monotone convergence theorem gives that
  $\left[ |\nabla A|^2 \right] < \infty$, completing the proof.
\end{proof}

\begin{proof}
(of Theorem \ref{t:huisken} = Theorem \ref{t:huisken2}).
Since $H \geq 0$ and  $L H = H$ by Theorem \ref{t:spectral}, the Harnack inequality implies that either $H \equiv 0$ or $H > 0$ everywhere.  In the trivial case where $H \equiv 0$, the self-similar equation implies that $\Sigma$ is a smooth minimal cone and, hence,
  a hyperplane through $0$.  Therefore, we will assume below that $H > 0$.

In the first step, we will prove weighted integral estimates that will be needed to justify various integrations by parts in the second  step.  The second step uses $L H = H$ and $L \, |A| \geq |A|$ (and the estimates from the first step) to show that $|A| = \beta \, H$ for a constant $\beta > 0$.  As in Huisken, \cite{H4}, this  geometric identity is the key for proving the classification; we do this in the third step.

\vskip2mm
{\bf{Step 1: Weighted estimates}}: Since $H > 0$,
Proposition \ref{p:ssy}
gives
\begin{equation}	\label{e:ssy}
	 \left[ |A|^2 + |A|^4 + |\nabla |A||^2 + |\nabla A|^2 \right] < \infty \, .
\end{equation}
Thus, $|A|$ is in the weighted $W^{1,2}$ space, so (2) in Lemma \ref{l:L2W12} gives
$
	\left[ |A|^2 \, | \nabla \log H |^2 \right] < \infty$.
	Furthermore, using that $\cL \log H = \frac{1}{2} - |A|^2 - | \nabla \log H|^2$, we also get that
$\left[ |A|^2 \, |\cL \log H | \right] < \infty$.
  This gives by  the trivial inequality $2ab \leq a^2 + b^2$ the integral estimate
    \begin{equation}	\label{e:logu22H}
  	\left[ |A|^2 | \nabla \log H| + | \nabla |A|^2| \,  | \nabla \log H|
	+   |A|^2 \, | \cL \, \log H |  \right]  < \infty \, .
  \end{equation}

The definitions  of $\cL$ and $L$ and the equality in \eqr{e:LmodA} give
  \begin{equation}
   |A| \, \cL |A| = |A| \, \left( L \, |A| - |A|^2 \, |A| - \frac{1}{2} \, |A| \right) = \frac{1}{2} \, |A|^2 - |A|^4
   + |\nabla A|^2 - |\nabla |A||^2 \, ,
  \end{equation}
  and combining this with \eqr{e:ssy} and  the trivial inequality $2ab \leq a^2 + b^2$ gives
  \begin{equation}	\label{e:modAmodA}
  	\left[ |A| | \nabla |A|| + | \nabla |A||^2
	+   |A|  \, | \cL \,|A| |  \right]  < \infty \, .
  \end{equation}

 {\bf{Step 2: Geometric identities}}:  By
 \eqr{e:logu22H},
  	 we can apply Corollary \ref{c:self} to $|A|^2$ and $\log H$ to get that
\begin{equation}	\label{e:fromself1z}
	\left[ \langle \nabla |A|^2 ,  \nabla \log H \rangle \right] = - \left[ |A|^2 \, \cL \, \log H \right]
	= \left[ |A|^2 \left(   |A|^2 - \frac{1}{2} + |\nabla \log H|^2 \right) \right] \, ,
\end{equation}
where the second equality used the equation for $\cL \, \log  H$.
  Similarly, \eqr{e:modAmodA}  allows us to apply
Corollary \ref{c:self} to two copies of $|A|$ and get
\begin{equation}	\label{e:fromself2z}
	\left[ |\nabla |A||^2 \right] = - \left[ |A| \, \cL \,  |A| \right] =  - \left[ |A| \, \left( L \, |A| - |A|^2 \, |A| - \frac{1}{2} \, |A| \right)  \right]
	\leq \left[      |A|^4 - \frac{1}{2}   \, |A|^2  \right] \, ,
\end{equation}
where the inequality used that $L \, |A| \geq |A|$ by
Lemma \ref{l:simonsA}.
 Substituting \eqr{e:fromself2z} into \eqr{e:fromself1z}   and subtracting the left-hand side gives
\begin{equation}	\label{e:fromself3z}
	0 \geq \left[ |A|^2 \,  |\nabla \log H|^2 - 2  \langle |A| \, \nabla \log H , \nabla |A| \rangle + |\nabla |A||^2  \right]
	= \left[ \left| |A| \, \nabla \log H - \nabla |A| \right|^2    \right] \, .
\end{equation}
We conclude that $ |A| \, \nabla \log H - \nabla |A|  \equiv 0$ and, thus, that
 $|A| = \beta \, H$ for a constant $\beta > 0$.  In particular, $|A|$  satisfies the equation $L |A| = |A|$ and, again by Lemma \ref{l:simonsA}, we get
 \begin{equation}	\label{e:forhui}
 	|\nabla A|^2 = |\nabla |A||^2 \, .
\end{equation}

  {\bf{Step 3: Classifying $\Sigma$ using the geometric identities}}:
 Now that we have obtained the key geometric identities \eqr{e:forhui} and $|A| = \beta \, H$, the rest of the argument follows Huisken's argument
 on pages $187$ and $188$ of \cite{H4}.

 As in  Lemma \ref{l:csA}, fix a point $p \in \Sigma$ and
choose $e_i$,
$i=1,\dots,n$, so that at $p$  we have
$  a_{ij}=\lambda_i\,\delta_{ij}$.   We showed in  \eqr{e:ecss1} that
\begin{equation}   \label{e:ecss1qq}
        |A|^2\,|\nabla |A||^2
 = \sum_{k} \left( \sum_{i}
        a_{ii,k}\,\lambda_{i} \right) ^2\leq |A|^2\sum_{i,k} a_{ii,k}^2
        \leq  |A|^2\sum_{i,j,k}  a_{ij,k}^2 =   |A|^2 \, |\nabla A|^2  \,
        .
 \end{equation}
By \eqr{e:forhui}, we
must have equality in the two  inequalities in \eqr{e:ecss1qq} so we conclude that:
\begin{enumerate}
\item For each $k$, there is a constant $\alpha_k$ so that
$a_{ii,k} = \alpha_k \, \lambda_i {\text{ for every }} i \, .$
\item If $i \ne j$, then $a_{ij,k} = 0$.
\end{enumerate}
Since   $a_{ij,k}$ is fully symmetric in $i$, $j$, and $k$ by the Codazzi equations, (2) implies
\begin{enumerate}
\item[(2')]   $a_{ij,k} = 0$ unless $i=j=k$.
\end{enumerate}
If $\lambda_i \ne 0$ and $j\ne i$, then $0 = a_{ii,j} = \alpha_j \, \lambda_i$ so we must have $\alpha_j = 0$.  In particular, if the rank of $A$ is at least two at $p$, then every $\alpha_j = 0$ and, thus, (1) implies that $\nabla A (p) = 0$.

\vskip2mm
We now consider two separate cases depending on the rank of $A$.

\vskip2mm
{\bf{Case 1: The rank is greater than one}}:
Suppose first that there is some point $p$  where the rank of  $A$ is at least two.
We will show that the rank of $A$ is at least two everywhere.  To see this,
 for each $q \in \Sigma$, let $\lambda_1 (q)$ and $\lambda_2 (q)$ be the
  two eigenvalues of $A(q)$ that are largest in absolute value and define the set
 \begin{equation}
 	\Sigma_2 = \{ q \in \Sigma \, | \, \lambda_1 (q) = \lambda_1 (p) {\text{ and }}
	\lambda_2 (q) = \lambda_2 (p) \} \, .
\end{equation}
The subset $\Sigma_2$ contains $p$ and,
since the $\lambda_i$'s are continuous in $q$, it is automatically closed. Given any point $q \in \Sigma_2$, it follows that   the rank of $A$ is at least $2$ at $q$.
 Since this is an open condition, there is an open set $\Omega$ containing $q$ where the rank of $A$ is at least two.  However, we have already shown that $\nabla A = 0$ on this set and this implies that the eigenvalues of $A$ are constant on $\Omega$.  Thus, $\Omega \subset \Sigma_2$ and we conclude that $\Sigma_2$ is open.  Since $\Sigma$ is connected, we conclude that $\Sigma = \Sigma_2$ and, thus, also that
  $\nabla A = 0$ everywhere on $\Sigma$.

   In theorem $4$ of \cite{L}, Lawson showed that every smooth hypersurface with $\nabla A = 0$ splits isometrically as a product of a sphere and a linear space.  Finally, such a product satisfies the self-shrinker equation only if  the sphere is centered at $0$ and has the appropriate radius.

 \vskip2mm
{\bf{Case 2: The rank is one}}:  Since $H >0$, the remaining case is where the rank of $A$ is exactly one at every point.   It follows that there is a   unit  vector field $V$ so that
\begin{equation}
	A (v,w) = H \, \langle v , V \rangle \, \langle w , V \rangle \, .
\end{equation}
The vector field $V$ is well-defined in a neighborhood of each point, but a priori is globally well-defined only up to $\pm 1$ (it is really the direction that is globally well-defined).

Working at a point $p$ and using the notation of \eqr{e:ecss1qq},  suppose that $\lambda_1 > 0$ (i.e, that $e_1 = V$ at $p$).  It follows from (1) and (2') that $a_{ij,k} = 0$ except possibly when $i=j=k=1$.  This means that at $p$
\begin{equation}	\label{e:nabuAvw}
	\nabla_u A (v,w) = \frac{\nabla_V A(V,V)}{H}  \, \langle u , V \rangle \,  A(v,w)
	\, .
\end{equation}
Since $\nabla A$ and  $\frac{\nabla_V A(V,V)}{H}  \, \langle \cdot , V \rangle \,  A$ are both $(0,3)$-tensors and $p$ is arbitrary, this equality is independent of the choice of frame and, thus, \eqr{e:nabuAvw} holds at every point.
Hence, if $\gamma(t)$ is a unit speed geodesic on $\Sigma$ and $v(t)$ is a parallel vector field tangent to $\Sigma$ along $\gamma$,
then
\begin{equation}
	\partial_t A(v,v) = \nabla_{\gamma'} A (v,v) = \frac{\nabla_V A(V,V)}{H}  \, \langle \gamma' , V \rangle \,  A(v,v)  \, .
\end{equation}
In particular, if $v(0)$ is in the kernel of $A$ at $ \gamma(0)$, then $A(v,v)$ vanishes along $\gamma(t)$ (the proof of this is sometimes called ``Gronwall's lemma'').  In other words, the kernel of $A$ is preserved under parallel transport.  Therefore, if $v(0)$ is in the kernel of $A$ at $ \gamma(0)$, then
\eqr{e:nabuAvw} gives
\begin{equation}
	\langle \nabla^{\RR^{n+1}}_{\gamma'} v    , \nn   \rangle    = A_{\gamma(t)} \,  (\gamma' (t) , v(t)) \equiv 0 \, ,
\end{equation}
so the vector field $v(t)$ is not just parallel along the curve, but it is actually {\underline{constant}}.  It follows that there are $(n-1)$ constant vectors $e_2 , \dots , e_n$ tangent to $\Sigma$ that give a global orthonormal frame  for the kernel of $A$.  Consequently, $\Sigma$ is invariant under the (isometric) translations in the $(n-1)$-plane spanned by $e_2 , \dots , e_n$.
Therefore,  $\Sigma$ is a product of a curve $\tilde{\gamma} \subset \RR^2$  and this   $(n-1)$-dimensional subspace.  The curve  $\tilde{\gamma}$ is a smooth embedded self-shrinking curve in $\RR^2$ with positive curvature (i.e., it is convex) and with ``polynomial length growth''.   Finally, by Corollary \ref{c:curvey} below, $\tilde{\gamma}$ must be a round circle.
\end{proof}

 The next lemma, which was used in the last paragraph of the preceding proof,  shows that any complete self-shrinking curve in $\RR^2$ is either a straight line or it lies in a bounded  set.
 In particular, if it has ``polynomial length growth'', then it is either a straight line or it is closed.  We were unable to find a proof of this in the literature, so we will include one here:

\begin{Lem}		\label{l:curvey}
If $\gamma: \RR \to \RR^2$ satisfies $|\gamma'| =1$ and $H= \frac{1}{2} \, \langle x , \nn \rangle$, then
 either $\gamma$ is a straight line through the origin or $H^2$ is positive and $|x|^2$ is   bounded.
\end{Lem}

\begin{proof}
The geodesic curvature   is   $H = - \langle
\nabla_{\gamma'} \gamma' , \nn \rangle$.
 Since $\nabla_{\gamma'} \nn = H \, \gamma'$,   we get
\begin{equation}	\label{e:Hpgo}
	2 \, H' =  \nabla_{\gamma'} \, \langle x , \nn \rangle
	= \langle \gamma' , \nn \rangle  + \langle x , \nabla_{\gamma'} \nn \rangle = H  \, \langle x , \gamma' \rangle \, .
\end{equation}
Similarly, differentiating $|x|^2$ gives
\begin{equation}	\label{e:x2go}
	\left( |x|^2 \right)' =   \nabla_{\gamma'} \, \langle x , x \rangle = 2 \, \langle \gamma' , x \rangle   \, .
\end{equation}
Combining these two equations, we see that
\begin{equation}
	\e^{\frac{|x|^2}{2} } \, \left( H^2 \, \e^{- \frac{|x|^2}{2} } \right)' =   	  2 \, H \, H' - \frac{1}{2} \, H^2 \, \left( |x|^2 \right)'   	= 	   H^2 \, \langle x , \gamma' \rangle -  H^2 \, \langle x , \gamma' \rangle =
	 0  \, ,
\end{equation}
so the quantity  $ H^2 \, \e^{- \frac{|x|^2}{2} } $
 is constant on $\gamma$; call this constant $E = E(\gamma)$.

 To see why this implies that $H^2$ is bounded, we use that
  $H^2 \leq \frac{|x|^2}{4}$ and, thus,
 \begin{equation}
 	H^2 \, \e^{- 2 \, H^2 } \geq  H^2 \, \e^{- \frac{|x|^2}{2} }  = E (\gamma) \, .
 \end{equation}
 In particular, since  $\lim_{u \to \infty} (u^2 \e^{-2u^2}) = 0$ and we may assume that $E(\gamma) > 0$ (otherwise $H \equiv 0$), we see that $H$
is  bounded.

Suppose that $\gamma$ is not a straight line through the origin.  It follows that $E(\gamma) > 0$ and, thus, that   \begin{equation}
 	\e^{ \frac{|x|^2}{2} } = \frac{H^2}{E(\gamma)}  \, .
 \end{equation}
 It follows   that $H^2> 0$ and, since $H^2$ is bounded by (1), that $|x|^2$ is bounded.
\end{proof}

It follows from Lemma \ref{l:curvey} that any simple complete self-shrinking curve in $\RR^2$ without boundary and with polynomial length growth  must either be a straight line or must be closed and convex.  Combining this with   Gage-Hamilton,   \cite{GH},  gives the following corollary:{\footnote{Gage-Hamilton showed that the curve shortening flow of any simple closed convex curve  becomes extinct at a round point; in particular, the round circle is the only simple closed convex self-shrinker.  This can also be seen by using Hamilton, \cite{Ha1}, or Huisken, \cite{H5}.  Finally,
  one could alternatively use the
Abresch and Langer, \cite{AbLa} classification of simple self-shrinkers; see Andrews, \cite{An}, for yet another  alternative proof that, unlike \cite{AbLa}, is not computer-assisted.}}

\begin{Cor}		\label{c:curvey}
Any simple complete self-shrinking curve in $\RR^2$ without boundary and with polynomial length growth  must either be a straight line or    a round circle.
\end{Cor}

Corollary \ref{c:curvey} has the following immediate consequences for self-shrinkers in $\RR^3$:

\begin{Cor}		\label{c:curvey2}
 Any smooth complete embedded self-shrinker in $\RR^3$ without boundary and with polynomial area growth  that
 splits off a line must either be a plane or    a round cylinder.
\end{Cor}

 \begin{Rem}
Lemma \ref{l:curvey} holds also for immersed self-shrinking curves.  The fact that
the quantity $H^2 \, \e^{- \frac{|x|^2}{2} }$ is constant on $\gamma$ seems not to
have been observed and may be useful in other problems.
\end{Rem}

\section{The classification of stable non-compact self-shrinkers}

 In this section, suppose that $\Sigma \subset \RR^{n+1}$ is a complete non-compact hypersurface without boundary that satisfies $H = \frac{\langle  x , \nn \rangle}{ 2}$    and  has polynomial volume growth.
 Throughout, $ L= \Delta + |A|^2 + \frac{1}{2} - \frac{1}{2} \, x \cdot \nabla$ will be operator from  the second variation formula for the $F_{0,1}$ functional.

  \subsection{The classification of $F$-stable non-compact self-shrinkers}

We can now prove that hyperplanes are the only open $F$-stable self-shrinkers with polynomial volume growth.
To do this, we will show that if $\Sigma$ is not a hyperplane, then there is a compactly supported variation of $\Sigma$ for which $F''$ is negative
no matter which values of $y$ and $h$ that we choose.

\begin{proof}
(of Theorem \ref{t:linearstable2}).
We will consider three separate cases.

\vskip2mm
\noindent {\bf{Case 1: $H$ vanishes identically.}}  It follows that  $\Sigma$ is minimal and, since $H =\frac{ \langle x , \nn \rangle}{2}$,  $\Sigma$ is also a cone.  However, the only (smooth) embedded minimal cones are hyperplanes.

\vskip2mm
\noindent {\bf{Case 2: $H$ does not vanish anywhere.}}
   In this case, Theorem \ref{t:huisken} (our extension of Huisken's classification of mean convex self-shrinkers from \cite{H4}) gives that $\Sigma$ is a product $\RR^k \times \SS^{n-k}$ where the $\SS^{n-k}$ has radius $\sqrt{2(n-k)}$ and $0 < k< n$.  We will see that these are all $F$-unstable.    Observe that   $H = \sqrt{(n-k)}/\sqrt{2}$ and $|A|^2 = \frac{1}{2}$ are  constant.  Let $x_1$ be the coordinate function corresponding to the first coordinate in the $\RR^k$.  Since $x_1$ is harmonic on $\Sigma$ and $\nabla x_1 = \partial_{1}$, we get that
   \begin{align}	\label{e:cLx1}
   	\cL \, x_1 = \Delta  \, x_1     - \frac{1}{2} \langle  x , \nabla x_1 \rangle = - \frac{1}{2} \, x_1 \, , \\
   	L \, x_1 = \cL \, x_1 + \frac{1}{2} \, x_1 + |A|^2 \, x_1  = \frac{1}{2} \, x_1 \, .   \label{e:Lx1}
   \end{align}
   It follows that $|x_1|$, $|\nabla x_1|$, and $|\cL x_1|$ are all in the weighted $L^2$ space
   since $\Sigma$ has polynomial volume growth.  Since the constant functions are also in the weighted $W^{1,2}$ space,  Corollary \ref{c:self} with $u=1$ and $v = x_1$ gives
 \begin{equation}	\label{e:secvar2bc}
	 0 =    \int_{\Sigma}
	 \left[  \cL \, x_1
		\right]
		 \, \e^{\frac{-|x|^2}{4}}
		 = - \int_{\Sigma}
	 \left[  \frac{x_1}{2}
		\right]
		 \, \e^{\frac{-|x|^2}{4}}     \, .
\end{equation}
Suppose next that
   $y \in \RR^{n+1}$ is fixed and $\phi$ is a bounded function on $\RR^k$, so that $\phi (x_1 , \dots , x_k) \, x_1$ is in the weighted $L^1$ space on $\Sigma$.  It follows from Fubini's theorem that
   \begin{align}
   	&\int_{\Sigma} \phi (x_1 , \dots , x_k) \, x_1 \, \langle y , \nn \rangle \,  \e^{\frac{-|x|^2}{4}}  \notag \\
	&\quad \quad =
	\int_{(x_1 , \dots , x_k) \in \RR^k} \, \, \left[ \phi (x_1 , \dots , x_k) \,  x_1 \, \e^{\frac{-(x_1^2 + \dots + x_k^2 + 2(n-k))}{4}} \int_{\SS^{n-k}} \langle y , \nn \rangle \,
	\right]  \, .
   \end{align}
  It is easy to see that
   $\langle y , \nn \rangle$ is independent of $(x_1, \dots , x_k)$ and is an eigenfunction of the Laplacian $\Delta_{\SS^{n-k}}$
   on the $\SS^{n-k}$ factor.  In particular,  $\langle y , \nn \rangle$ integrates to zero on each slice where $(x_1 , \dots , x_k)$ is fixed, so we conclude that
     \begin{equation}	\label{e:fjx1}
   	\int_{\Sigma} \phi (x_1 , \dots , x_k) \,  x_1 \, \langle y , \nn \rangle \,  \e^{\frac{-|x|^2}{4}}  =0  \, .
   \end{equation}
We would like to use $x_1$ as the variation, but it does not have compact support.  To deal with this, let
  $\phi_j$ be a $C^2$ cutoff function on $\RR^k$ that is identically one on $B_j$, cuts off to zero between $\partial B_j$ and $\partial B_{j+2}$,
  has $|\phi_j| \leq 1$,  $|\nabla \phi_j| \leq 2$, and $|\Delta \phi_j| \leq C$ for a fixed constant $C$.
     Therefore, using  $f_j = \phi_j \, x_1$ as a variation,
    Theorem \ref{t:secvar} gives
  \begin{equation}	\label{e:secvar2b}
	 F_{f_j}''  =   (4\pi)^{-n/2} \,  \int_{\Sigma}
	 \left[ - f_j \, L f_j + 2 \frac{\sqrt{n-k}}{\sqrt{2}} \, h \, f_j
		     -     \frac{(n-k) \, h^2}{2}
		   - \frac{\langle y , \nn \rangle^2}{2}
		\right]
		 \, \e^{\frac{-|x|^2}{4}}
		 \, ,
\end{equation}
where we used \eqr{e:fjx1} to get rid of the $f_j \, \langle y , \nn \rangle$ term.  Using \eqr{e:Lx1} and
\eqr{e:secvar2bc},  it follows easily  from the monotone convergence theorem that
  \begin{equation}	\label{e:secvar2bcc}
	 \lim_{j \to \infty}  \int_{\Sigma}
	 \left[ - f_j \, L f_j
		\right]
		 \, \e^{\frac{-|x|^2}{4}}  =    - \int_{\Sigma}
	 \left[   \frac{x_1^2}{2}
		\right]
		 \, \e^{\frac{-|x|^2}{4}}
		 {\text{ and }}
		 \lim_{j \to \infty}  \int_{\Sigma}
	  f_j
 \, \e^{\frac{-|x|^2}{4}}  = 0
		 \, .
\end{equation}
Hence,
for all $j$ sufficiently large, we get
\begin{equation}
  F_{f_j}'' \leq
		  \frac{1}{2} \, \, \,   (4\pi)^{-n/2}   \, \int_{\Sigma}
	\left[ - \frac{x_1^2}{2} \right]
		 \, \e^{\frac{-|x|^2}{4}}  \, .
\end{equation}
 Since this is negative for every $h$ and $y$, we get that $\Sigma$ is unstable as claimed.

\vskip2mm
\noindent {\bf{Case 3: $H$ vanishes somewhere, but not identically.}}     By Theorem \ref{c:pde}, we know that $\mu_1 < -1$.  Therefore, we can apply Lemma \ref{l:ortho} to get some $\bar{R}$ so that if $R \geq \bar{R}$ and $u$ is a Dirichlet eigenfunction for $\mu_1 (B_R)$, then
   for any $h \in \RR$ and any $y \in \RR^{n+1}$ we have
   \begin{equation}	\label{e:u0a}
   	\int_{B_R \cap \Sigma} \left[  -u \, L \, u + 2 u \, h \, H + u \,
	 \langle y , \nn \rangle
		    -   h^2 \, H^2
		  - \frac{\langle y , \nn \rangle^2}{2}  \right] \, \e^{ - \frac{|x|^2}{4} } < 0 \, .
   \end{equation}
   If we fix any $R \geq \bar{R}$ and then use the corresponding $u$ in
  the second variation formula from Theorem \ref{t:secvar}, then we get that $F''$ is negative for every $h$ and $y$, so we get that $\Sigma$ is unstable as claimed.
\end{proof}

\subsection{Non-compact entropy-stable self-shrinkers}

In this section, we will use our characterization of entropy-stable self-shinkers together with our classification of $F$-stable self-shinkers to  prove Theorems \ref{c:nonlin1a} and \ref{t:nonlin1a}.    By Corollary \ref{c:curvey2}, the only embedded self-shrinkers in $\RR^3$ that split off a line are cylinders and planes, so 
Theorem \ref{c:nonlin1a} is an immediate consequence of  part (2) of Theorem \ref{t:nonlin1a}.

Throughout this subsection, $\Sigma \subset \RR^{n+1}$ will be a smooth complete embedded self-shrinker without boundary and with $\lambda (\Sigma) < \infty$.

As mentioned, it suffices to prove Theorem \ref{t:nonlin1a} that classifies entropy-stable self-shrinkers in $\RR^{n+1}$; we do this next.

\begin{proof}
(of Theorem \ref{t:nonlin1a}).
We have already shown that the sphere is the only smooth embedded {\underline{closed}}  entropy-stable self-shrinker, so we may assume that $\Sigma$ is open.

We will first show part (2) of the theorem:
\begin{quote}
If $\Sigma$ is not a round $\SS^n$ and does not split off a line, then $\Sigma$ can be perturbed to a  graph $\tilde{\Sigma}$ of a compactly supported function $u$  over $\Sigma$ with $\lambda (\tilde{\Sigma}) < \lambda (\Sigma)$.  In fact, we can take $u$ to have arbitrarily small $C^m$ norm for any fixed $m$.
\end{quote}
Since $\Sigma$ does not split off a line, this follows immediately from combining Theorem
 \ref{t:nonlin1c} and  Theorem \ref{t:linearstable2}.

Now that we have part (2), the key for proving part  (1) is the following observation:\\
{\bf{Claim}}:
If $\Sigma = \RR^{n-k} \times \Sigma_0$ where $\Sigma_0 \subset \RR^{k+1}$, then
\begin{equation}
	F^{\RR^{n+1}}_{x , t_0} (\Sigma) = F^{\RR^{k+1}}_{x_0 , t_0} (\Sigma_0) \, ,
\end{equation}
 where $x_0$ is the projection of $x$ to $\RR^{k+1}$.

{\bf{Proof of the claim}}:
By induction, it suffices to do the case where $(n-k) = 1$ and $k+1 = n$.  Given a vector $y \in \RR^{n+1}$, set
$y = (y_1 , y_0)$ with $y_1$ and $y_0$ the
  projections of $y$ to  $\RR^1$ and  $\RR^n$, respectively.  Since
  \begin{equation}
  	\int_{\RR} \e^{\frac{-|x_1- y_1|^2}{4t_0}} \, dy_1 =
	(4 \, t_0)^{\frac{1}{2}}  \, \int_{\RR} \e^{ -s^2 } \,  ds = ( 4 \, \pi \, t_0)^{\frac{1}{2}} \, ,
	  \end{equation}
  Fubini's theorem gives
\begin{align}	
	F^{\RR^{n+1}}_{x , t_0} (\Sigma) &=
 	  (4\pi t_0)^{-n/2} \, \int_{\RR \times \Sigma_0} \, \e^{\frac{-|x- y|^2}{4t_0}}  =  (4\pi t_0)^{-n/2} \, \int_{\Sigma_0} \,
	  \left( \int_{\RR} \e^{\frac{-|x_1- y_1|^2}{4t_0}} \, dy_1 \right) \e^{\frac{-|x_0- y_0|^2}{4t_0}}   \notag \\
	  &=   (4\pi t_0)^{-(n-1) /2} \, \int_{\Sigma_0} \,
	  \e^{\frac{-|x_0- y_0|^2}{4t_0}} =  F^{\RR^{n}}_{x_0 , t_0} (\Sigma_0)
	  \, .
\end{align}
This gives the claim.

To prove (1),
suppose that $\Sigma = \RR^{n-k}  \times \Sigma_0$ where $\Sigma_0$ is a self-shrinking hypersurface in $\RR^{k+1}$ and is {\underline{not}} a $k$-dimensional round sphere and does not split off another line.
Part (2) that we just proved gives a graph $\tilde{\Sigma}_0$ of a compactly supported function $u$ on $\Sigma_0$ so that $\lambda_{\RR^{k+1}}  \, (\tilde{\Sigma}_0) < \lambda_{\RR^{k+1}}  \, ( {\Sigma}_0) $.
 Part (1)   follows from combining this with the claim.

\end{proof}

\section{Regularity of $F$-stable self-shrinkers}

Recall that the combined work of Brakke (Allard), \cite{B}, \cite{Al}, Huisken, \cite{H3},   Ilmanen, \cite{I1}, and White, \cite{W4}, yields that, for a MCF in $\RR^{n+1}$ starting at a smooth closed embedded hypersurface, any tangent flow (also past the first singular time) has the property that the time $-1$ slice is an $n$-dimensional $F$-stationary integral varifold\footnote{An $F$-stationary integral varifold is a weakly defined minimal hypersurface in the conformally changed metric on $\RR^{n+1}$} with Euclidean volume growth.   The main result of this section, Theorem \ref{t:mainreg} below, shows that if the regular part of such an $n$-dimensional $F$-stationary integral varifold is orientable and $F$-stable and the singular set has finite $(n-2)$-dimensional Hausdorff measure, then it is smooth (at least for $n\leq 6$) and thus Theorem \ref{t:liketo} applies and shows that such a self-shrinker is either a round sphere or a hyperplane; see Corollary \ref{c:mainreg} below.

\begin{Thm}   \label{t:mainreg}
Let $\Sigma\subset \RR^{n+1}$ be an $n$-dimensional integral varifold that is a time slice of a tangent flow of a MCF starting at a smooth closed embedded hypersurface in $\RR^{n+1}$.  If the regular part of $\Sigma$ is orientable and $F$-stable and the singular set has finite $(n-2)$-dimensional Hausdorff measure, then it corresponds to an embedded, analytic hypersurface away from a closed set of singularities of Hausdorff dimension at most $n-7$, which is absent if $n \leq 6$ and is discrete if $n=7$.
\end{Thm}

Recall that, for a MCF, at any point in space-time tangent flows exist and are Brakke flows.  In particular, the time slices of such a tangent flow are $F$-stationary integral varifolds.

As an immediate corollary of Theorem \ref{t:mainreg} combined with Theorem \ref{t:liketo},  we get:

\begin{Cor}   \label{c:mainreg}
Let $\{\cT_t\}_{t<0}$ be a tangent flow of a MCF starting at a smooth closed embedded hypersurface in $\RR^{n+1}$.  If the regular part of $\cT_{-1}$ is orientable and $F$-stable and the singular set has finite $(n-2)$-dimensional Hausdorff measure and $n\leq 6$, then $\cT_{-1}$ is either a round sphere or a hyperplane.
\end{Cor}

Recall that we defined the operator $L$ in Section \ref{s:two} by
\begin{equation} \label{e:nes1}
L\,v=\e^{\frac{|x|^2}{4}}\,\text{div}\left(\e^{-\frac{|x|^2}{4}}\,\nabla v\right)+|A|^2v+\frac{1}{2}v\, .
\end{equation}

\begin{Lem}	\label{l:not2bad}
Suppose that $\Sigma\subset \RR^{n+1}$ is a $F$-stationary $n$-dimensional integral varifold
and $\Omega \subset \Sigma$ is an open subset of the regular part of $\Sigma$.
If $\mu_1 (L , \Omega) < -\frac{3}{2} $, then $\Omega$ is $F$-unstable.
\end{Lem}

\begin{proof}
If $\mu_1 (L , \Omega) < -\frac{3}{2} $, then we
 get a function $f_0$ with compact support in $\Omega$
 satisfying
  \begin{equation}	
	  -  \int_{\Omega}  \left( f_0 \, L f_0  \right) \, \e^{- \frac{|x|^2}{4} }  <  - \frac{3}{2}  \int_{\Omega}
	   f_0^2 \, \e^{- \frac{|x|^2}{4} }   \, .
  \end{equation}
	  Substituting $f_0$ into the second variation formula from Theorem \ref{t:secvar} gives
  \begin{align}	
	F'' &<   (4\pi)^{-n/2} \, \int_{\Omega}
	 \left[ - \frac{3}{2}  \, f_0^2 + 2 f_0 \, h H + f_0 \,  \langle y , \nn \rangle
		    -   h^2 \, H^2
		  - \frac{\langle y , \nn \rangle^2}{2}
		\right]
		 \, \e^{\frac{-|x|^2}{4}}   \notag \\
		 &= (4\pi)^{-n/2} \, \int_{\Omega}
	 \left[ - \frac{1}{2}  \, \left(f_0  -  \langle y , \nn \rangle \right)^2 -
	 \left(    f_0  - h H \right)^2
		\right]
		 \, \e^{\frac{-|x|^2}{4}}	  \, .
\end{align}
Since $F'' < 0$ no matter which values of $h$ and $y$ that we use, it follows that $\Omega$ is $F$-unstable.
\end{proof}

In the next lemma, $\text{reg}(\Sigma)$ is the regular part of an integral varifold $\Sigma$.  Moreover, in this next lemma, we will assume that $\Vol (B_r(x))\leq V\,r^n$ for all $0<r<1$ and all $x$.  As we have noted before, this is automatically satisfied for tangent flows.

 \begin{Lem}  \label{l:almoststable}
Given $\epsilon>0$,  an integer $n$, and $R$, $V>0$, there exists $r_0=r_0(\epsilon,  n,R,V)>0$ such that the following holds:

Suppose that $\Sigma\subset \RR^{n+1}$ is an $n$-dimensional $F$-stationary integral varifold, $F$-stable on the regular part and $\Vol (B_r(x)\cap \Sigma)\leq V\,r^n$ for all $0<r<1$ and $x\in \RR^{n+1}$, then:
\begin{itemize}
\item  $\Sigma$ is stationary with respect to the metric $g_{ij}=\e^{-\frac{|x|^2}{4}}\,\delta_{ij}$ on $\RR^{n+1}$.
\item For all $x_0\in B_R(0)$ and for all smooth functions $\phi$ with compact support in the regular part of $B_{r_0}(x_0)\cap \Sigma$, we have the stability-type inequality
\begin{equation}  \label{e:stabwi}
\int_{\Sigma}|A_g|_g^2\,\phi^2\,d\Vol_g\leq (1+\epsilon)\,\int_{\Sigma}|\nabla_g \phi|_g^2\,d\Vol_g\, .
\end{equation}
\end{itemize}
Here $A_g$, $\nabla_g$, $|\cdot |_g$, and $d\Vol_g$ are with respect to the metric $g=g_{ij}$ on $\RR^{n+1}$.
\end{Lem}

\begin{proof}
By inspection, the argument  in Section \ref{s:one} that showed that smooth self-shrinkers easily generalizes to show that $\Sigma$ is a stationary varifold in the metric $g_{ij}$ on $\RR^{n+1}$.  We need to show that, in the $g_{ij}$ metric on $\RR^{n+1}$ where the varifold is stationary, the usual second variational operator (or stability operator) satisfies a stability-type inequality, i.e., \eqr{e:stabwi}, on sufficiently small balls.  This will follow from $\mu_1(L,\text{reg}(\Sigma))\geq -\frac{3}{2}$ together with the Sobolev and Cauchy-Schwarz inequalities.

The bulk of the lemma is to establish \eqr{e:stabwi} for $A$, $\nabla=\nabla_{\RR^{n+1}}$, and with respect to the volume on $\Sigma$ coming from thinking of $\Sigma$ as a subset of Euclidean space.  This is because the lemma will follow from this with a slightly worse $\epsilon$ because of the following:

For a metric of the form $g_{ij}=f^2 \, \delta_{ij}$, the Christoffel symbols are
\begin{equation}
\Gamma_{k\ell}^i=\frac{1}{2}\,g^{im}\,(g_{mk,\ell}+g_{m\ell,k}-g_{k\ell,m})
=\frac{f_{\ell}\,\delta_{ik}+f_k\,\delta_{i\ell}-f_i\,\delta_{k\ell}}{f}\, .
\end{equation}
Hence, for vector fields $X$ and $Y$
\begin{align}
\nabla^g_XY&=\nabla_X^{\RR^{n+1}}Y+X_k\,Y_{\ell}\,\Gamma_{k\ell}^i\, \partial_i
=\nabla_X^{\RR^{n+1}}Y+X_k\,Y_{\ell}\,\frac{f_{\ell}\,\delta_{ik}+f_k\,\delta_{i\ell}-f_i\,\delta_{k\ell}}{f}\, \partial_i  \notag \\
&= \nabla_XY + \frac{
\langle Y , \nabla f \rangle \, X + \langle X , \nabla f \rangle \, Y - \langle X , Y \rangle \, \nabla f }{f} \, ,
\end{align}
where all quantities in the second line are computed in the Euclidean metric.  If the vector fields $X$ and $Y$ are tangent to   $\Sigma$ whose (Euclidean) unit normal is $\nn$,  then the second fundamental form $A^g$ in the $g$ metric is given by
\begin{equation}
A^g(X,Y) = f^2 \, \langle  \nabla^g_XY , f^{-1} \, \nn \rangle =  f  \, \langle  \nabla^g_XY ,  \nn \rangle
 = f  \, A(X,Y)  -  \langle X , Y \rangle \, \langle \nabla f ,  \nn \rangle  \, ,
\end{equation}
where the last equality used that $X$ and $Y$ are tangential while $\nn$ is normal.  If $e_i$ is a $\delta_{ij}$-orthonormal frame for $\Sigma$, then $f^{-1} \, e_i$ is a frame in the new metric and, thus,
\begin{equation}
	a^g_{ij} \equiv f^{-2} \, A^g ( e_i , e_j ) = f^{-1} \, a_{ij} - f^{-2} \,  \langle e_i , e_j \rangle \, \langle \nabla f , \nn \rangle \, .
\end{equation}
Squaring and using an absorbing inequality yields for all $\delta>0$
\begin{equation}
|A_g|_g^2\leq (1+\delta)\frac{|A|^2}{f^2}+n \,\left(1+\frac{1}{\delta}\right)\frac{|\nabla f|^2}{f^4}\, .
\end{equation}
Moreover, $d\Vol_g=f^n\,d\Vol$, and $f\,|\nabla_g\phi |_g=|\nabla \phi|$.    Integrating
\begin{align}  \label{e:AgA}
\int_{\Sigma}|A_g|_g^2\,\phi^2\,d\Vol_g&\leq (1+\delta) \,\sup_{|\phi|>0}f^{n-2}\int_{\Sigma}|A|^2\,\phi^2\,d\Vol\\
&+n \,\left(1+\frac{1}{\delta}\right)\sup_{|\phi|>0}(|\nabla f|^2\,f^{n-4})\int_{\Sigma}\phi^2d\Vol\, .\notag
\end{align}
Since $f=\e^{-\frac{|x|^2}{4n}}$ and the support of $\phi$ is contained in $B_{r_0}(x_0)$,  it follows that
\begin{equation}  \label{e:supinf}
\sup_{|\phi|>0}f^{n-2}\leq (1+O_{R,n}(r_0))\inf_{|\phi|>0}f^{n-2}\, ,
\end{equation}
where $O_{R,n}(r_0)\to 0$ as $r_0\to 0$.
The lemma now easily follows from \eqr{e:AgA} and \eqr{e:supinf}
provided we can show \eqr{e:stabwi} for $A$, $\nabla$, $|\cdot|$, $d\Vol$ and we can show that
$\int_{\Sigma} |\phi|^2\,d\Vol$ can be bounded by a small constant times $\int_{\Sigma}|\nabla \phi|^2\,d\Vol$ (the last will follow, after choosing $r_0$ sufficiently small, from the Dirichlet-Poincar\'e inequality that we show below).

In the remainder of the lemma, gradients, second fundamental forms, and volumes are all with respect to the Euclidean metric $\delta_{ij}$ and the metric on $\Sigma$ that it induces.

Observe first that by the Sobolev inequality  (theorem $18.6$ on page $93$ of Simon, \cite{Si}; cf. also \cite{Al}, \cite{MiSi}) for any smooth non-negative function $\psi\in C_0^{\infty}(B_{r_0}(x_0)\cap \Sigma)$ with compact support on the regular part of the varifold (here we also use that since it is $F$-stationary $H=\frac{\langle x,\nn\rangle}{2}$ and that $|x_0|\leq R$)
\begin{align}  \label{e:nes2}
\left(\int_{\Sigma}\psi^{\frac{n}{n-1}}\right)^{\frac{n-1}{n}}&\leq C\,\int_{\Sigma}(|\nabla \psi|+|H|\,\psi)\leq C\,\int_{\Sigma}(|\nabla \psi|+|x|\,\psi)\notag\\
&\leq C\,\int_{\Sigma}(|\nabla \psi|+(R+r_0)\,\psi)\, .
\end{align}
On the other hand, by the H\"older inequality
\begin{equation}
\int_{\Sigma} \psi\leq \Vol (B_{r_0}(x_0)\cap \Sigma)^{\frac{1}{n}}\,\left(\int_{\Sigma} \psi^{\frac{n}{n-1}}\right)^{\frac{n-1}{n}}\, .
\end{equation}
Combining these two inequalities and setting $\psi=|\phi|$ yields
\begin{equation}
\int_{\Sigma} |\phi |\leq C\,  \Vol (B_{r_0}(x_0)\cap \Sigma)^{\frac{1}{n}}\left( \int_{\Sigma} |\nabla \phi|+(R+r_0)\int_{\Sigma} |\phi |\right)\, .
\end{equation}
Combining this with the bound $\Vol (B_{r_0}(x_0)\cap \Sigma)\leq V\,r_0^n$ yields
\begin{equation}
(1-CV^{\frac{1}{n}} r_0\,(R+r_0))\int_{\Sigma} |\phi |\leq C\,V^{\frac{1}{n}} \,r_0\int_{\Sigma}|\nabla \phi |\, .
\end{equation}
Finally, choosing $r_0$ sufficiently small and applying the Cauchy-Schwarz inequality gives the Dirichlet-Poincar\'e inequality for a constant $C$ (here, as elsewhere in this lemma, $C$ denotes a constant, though the actual constant may change from line to line)
\begin{equation}   \label{e:DP}
\int_{\Sigma}|\phi|^2 \leq C\,r_0^2\int_{\Sigma}|\nabla \phi|^2\, .
\end{equation}

Combining \eqr{e:nes1} and Lemma \ref{l:not2bad} yields
\begin{align}
\int_{\Sigma}|A|^2\,\phi^2\,\e^{-\frac{|x|^2}{4}}&
=\int_{\Sigma}\phi\,L\,\phi\,\e^{-\frac{|x|^2}{4}}+\int_{\Sigma}|\nabla \phi|^2\,\e^{\frac{-|x|^2}{4}}
-\frac{1}{2}\int_{\Sigma}\phi^2\,\e^{-\frac{|x|^2}{4}}\notag\\
&\leq -\left(\mu_1+\frac{1}{2}\right)\int_{\Sigma}\phi^2\,\e^{-\frac{|x|^2}{4}}
+ \int_{\Sigma}|\nabla \phi|^2\,\e^{\frac{-|x|^2}{4}}\\
&\leq \int_{\Sigma}\phi^2\,\e^{-\frac{|x|^2}{4}}
+ \int_{\Sigma}|\nabla \phi|^2\,\e^{\frac{-|x|^2}{4}}\, .\notag
\end{align}
Multiplying both sides by $\e^{\frac{|x_0|^2}{4}}$ and using that $||x_0|^2 -|x|^2|\leq r_0(2R+r_0)$ on $B_{r_0}(x_0)$ yields
\begin{align}
\e^{-r_0(R+r_0)}\int_{\Sigma}|A|^2\,\phi^2&\leq \int_{\Sigma}|A|^2\,\phi^2\,\e^{\frac{|x_0|^2-|x|^2}{4}}\notag\\
&\leq \int_{\Sigma}\phi^2\,\e^{\frac{|x_0|^2-|x|^2}{4}}
+ \int_{\Sigma}|\nabla \phi|^2\,\e^{\frac{|x_0|^2-|x|^2}{4}}\\
&\leq \e^{r_0(R+r_0)}\int_{\Sigma}\phi^2
+\e^{r_0(R+r_0)}\int_{\Sigma}|\nabla \phi|^2\, .\notag
\end{align}
Hence,
\begin{equation}
\int_{\Sigma}|A|^2\,\phi^2\leq \e^{2r_0(R+r_0)}\int_{\Sigma}\phi^2+\e^{2r_0(R+r_0)}\int_{\Sigma}|\nabla \phi|^2
	\, .
\end{equation}
The lemma now easily follows from the Dirichlet-Poincar\'e inequality after choosing $r_0$ sufficiently small.
\end{proof}

\begin{proof}
(of Theorem \ref{t:mainreg}).
This is an immediate consequence of \cite{ScSi} and Lemma \ref{l:almoststable}.  Note that the argument in \cite{ScSi} goes through with the slightly weaker stability inequality \eqr{e:stabwi}.
\end{proof}

The proof of Theorem \ref{t:mainreg} relied on the result of Schoen-Simon, \cite{ScSi}.  This result has recently been significantly sharpened by Wickramasekera, \cite{Wi}.  Thus, if we use Wickramasekera's result, then we get a sharper regularity result that we describe below.  Before recalling the theorem of Wickramasekera, we need the
following definition of his:

Fix any $\alpha\in (0,1)$ we say that an $n$-dimesional varifold $\Sigma\subset \RR^{n+1}$ satisfies the $\alpha$-{\it{structural hypothesis}} if no singular point of $\Sigma$ has a neighborhood in which the support of $\Sigma$ is the union of embedded $C^{1, \alpha}$ hypersurfaces with boundary meeting (only) along an $(n-1)$-dimensional embedded $C^{1, \alpha}$ submanifold.

\begin{Pro}	\label{p:wi0}
Let $\Sigma\subset \RR^{n+1}$ be an $n$-dimensional $F$-stationary integral varifold having Euclidean volume growth and orientable, $F$-stable regular part.  Fix any $\alpha \in (0, 1)$, and suppose that $\Sigma$ satisfies the $\alpha$-structural hypothesis.

The varifold then corresponds to an embedded, analytic hypersurface away from a closed set of singularities of Hausdorff dimension at most $n-7$, which is absent if $n \leq 6$ and is discrete if $n=7$.
\end{Pro}

As mentioned,   the proof of this proposition will use a very recent result of Wickramasekera, \cite{Wi}, in place of the result of Schoen-Simon.
 For  convenience of the reader state, we state this next (the proposition is stated for a unit ball in Euclidean space, but holds for all sufficiently small balls in a fixed Riemannian manifold).

\begin{Pro}  \label{p:wi}
(Wickramasekera \cite{Wi}).
Consider an $n$-dimensional stationary integral varifold $\Sigma$  in an open ball in $\RR^{n+1}$ having finite mass and orientable regular part.  Suppose also that for all smooth functions with compact support contained in the regular part of $\Sigma$ we have the stability-type inequality $\int_{\Sigma} |A|^2\,\phi^2\leq (1+\epsilon)\,\int_{\Sigma} |\nabla \phi|^2$ for some sufficiently small $\epsilon>0$.   Fix any $\alpha \in (0, 1)$, and suppose that $\Sigma$ satisfies the $\alpha$-structural hypothesis.

The varifold then corresponds to an embedded, analytic hypersurface away from a closed set of singularities of Hausdorff dimension at most $n-7$, which is absent if $n \leq 6$ and is discrete if $n=7$.
\end{Pro}

\appendix

\section{Calculations for graphs over a hypersurface}		\label{s:AC}

In this appendix, we will calculate various geometric quantities for a one-parameter family $\Sigma_s$ of (normal) graphs over
 a smooth embedded hypersurface $\Sigma \subset \RR^{n+1}$.    To define $\Sigma_s$, we fix a  unit normal $\nn$ on $\Sigma$ and a
    function $u$ on $\Sigma$ and let  $\Sigma_s$ be given by
\begin{equation}	
      F(\cdot , s) : \Sigma \to \RR^{n+1} {\text{ with }} F(p,s) = p + s\, u(p) \, \nn (p) \, .
\end{equation}
Let $\nn (p,s)$ denote the unit normal to $\Sigma_s$ at the point $F(p,s)$; clearly, we have that $\nn (p,0) = \nn (p)$.  Similarly, let $H(p,s)$ be the mean curvature of $\Sigma_s$ at $F(p,s)$.

\begin{Lem}	\label{l:calcnH}
 We have that
 \begin{align}	 \label{e:Nip}
	\frac{\partial \nn }{\partial s} (p,0) & =  - \nabla u   (p) \, , \\
 	\frac{\partial H}{\partial s} (p,0) & =   - \Delta \, u (p) - |A|^2 (p)  \, u (p)  \, ,  \label{e:big4}
 \end{align}
where $\Delta$ and $A$ are the Laplacian and second fundamental form, respectively,  of $\Sigma$.
\end{Lem}

 Before proving the lemma, we
   choose an orthonormal frame $\{ e_i \}_{i\leq n}$ for the tangent space to $\Sigma$.
  This gives a frame $\{ e_1 , \dots , e_n , \, e_{n+1} = \nn \}$ for the tangent space to $\RR^{n+1}$ at points in $\Sigma$.     Let
  $a_{ij} = \langle \nabla_{e_i} e_j , \nn \rangle$ be the second fundamental form for $\Sigma$, so that
  \begin{equation}    \label{e:diffnn}
    \nabla_{e_i} \nn = -  a_{ij} \, e_j \, .
\end{equation}

\begin{Rem}
By convention, we sum over repeated indices in expressions such as \eqr{e:diffnn}.  In this section, often we sum over $j \leq n$ (as in \eqr{e:diffnn}), but we  occasionally sum over all $j \leq n+1$.  To avoid confusion, we will use the greek index $\alpha$   for sums over   $1, \dots, n+1$.
\end{Rem}

We will extend both the function $u$ and this frame to a
  small (normal) neighborhood of $\Sigma$ by parallel translation in the normal (i.e., $e_{n+1}$) direction, so that $\langle e_{n+1} , \nabla u \rangle =0$   and
\begin{equation}    \label{e:onstig}
    \nabla_{e_{n+1}} \e_i =   0 \, .
\end{equation}

The next lemma computes how covariant derivatives of   $e_{n+1}$ vary as we move off of $\Sigma$ in the normal direction.

\begin{Lem}	\label{l:chri}
If   we set $f(p, h) =  \langle \nabla_{e_i}  e_{n+1} , e_j \rangle (p + h \, \nn (p))$
for $p \in \Sigma$ and $h \in \RR$,
  then
\begin{equation}
    \frac{ \partial f}{ \partial h} (p,0) = - a_{ik} (p) \, a_{jk} (p) \, .
\end{equation}
\end{Lem}

\begin{proof}
Differentiating $f$ gives
\begin{equation}
    \frac{\partial f}{\partial h} (p,h) =  e_{n+1} \langle \nabla_{e_i} e_{n+1} , e_j \rangle =   \langle \nabla_{e_{n+1}} \nabla_{e_i} e_{n+1} , e_j
    \rangle  \, ,
    \end{equation}
    where the terms on the right are evaluated at $p+ h \, \nn(p)$ and  the last equality used \eqr{e:onstig}.
    Using that the Riemann curvature (of $\RR^{n+1}$) vanishes, we get
  \begin{equation}
    \frac{\partial f}{\partial h} (p,0)
    =   \langle  \nabla_{e_i}  \nabla_{e_{n+1}} e_{n+1} , e_j
    \rangle +
    \langle \nabla_{[ e_{n+1} , e_i]} e_{n+1} , e_j
    \rangle
    \, ,
\end{equation}
where  all the covariant
derivatives are now performed at $p \in \Sigma$.
To compute this, use \eqr{e:onstig}  and  \eqr{e:diffnn} to get   that $\nabla_{e_{n+1}} e_{n+1} = 0$    and
at $p$
\begin{equation}
	[e_{n+1} , e_i] = \nabla_{e_{n+1}} e_i - \nabla_{e_i} e_{n+1} = a_{ik} \,
e_k \, ,
\end{equation}
so that
\begin{equation}
   \frac{\partial f}{\partial h} (p,0)  =  a_{ik}
    \langle \nabla_{e_k } e_{n+1} , e_j
    \rangle = - a_{ik} \, a_{jk}
    \, .
\end{equation}
\end{proof}

\begin{proof}
(of Lemma \ref{l:calcnH}).
The tangent space to $\Sigma_s$
 is spanned by $\{ F_i \}_{i \leq n} $   where
\begin{equation}    \label{e:Xi}
    F_i (F(p,s)) =  dF_{(p,s)} \, (e_i(p))  = e_i(p)  + s\, u_i (p) \, \nn (p) - s\, u (p) a_{ik} (p) \, e_k (p)  \, .
\end{equation}
Here we used \eqr{e:diffnn} to differentiate $\nn$ and we used
that $\nabla_v p = v$ for any vector $v$ since $p$ is the position vector on $\Sigma$.
Using \eqr{e:Xi}, we compute the metric $g_{ij}$ for the graph relative to the frame $F_i$:
\begin{equation}
    g_{ij}(p,s) \equiv \langle F_i , F_j \rangle =   \delta_{ij} + s^2 \, u_i u_j +
    s^2 \, u^2  a_{ik}  a_{jk} - 2 \, s \, u  a_{ij}   \, ,
\end{equation}
where all of the quantities on the right are evaluated at the point $p \in \Sigma$.
 The second fundamental form  of $\Sigma_s$ at $F(p,s)$ relative to the frame $F_i$ is given by
\begin{equation}	\label{e:aijps}
 	a_{ij}(p,s) \equiv \langle \nabla_{F_i} F_j  , \nn \rangle = - \langle \nabla_{F_i} \nn   , F_j \rangle
	\, ,
\end{equation}
where this time  all quantities on the right are evaluated at the point $F(p,s)$.

{\bf{Differentiating the metric and normal}}:  By \eqr{e:Xi},
 the $s$  derivative of $F_i$   is
\begin{equation}    \label{e:Xip}
   \frac{\partial F_i}{\partial s} (p,0)  =   u_i (p) \, e_{n+1} (p) -  u (p) \, a_{ik} (p) \, e_k (p) \, .
\end{equation}
If we differentiate $\langle F_i , \nn \rangle = 0$ and $\langle \nn , \nn \rangle = 1$ with respect to $s$, then we get that the vector $\nn'(p) \equiv \frac{ \partial \nn }{\partial s} (p,0)$ satisfies
\begin{equation}
	  \langle e_i (p) , \nn'(p) \rangle = - \langle \nn (p) ,
	   \frac{\partial F_i}{\partial s} (p,0) \rangle
	   {\text{ and }} \langle \nn (p)  , \nn' (p) \rangle = 0 \, .
\end{equation}
Using \eqr{e:Xip}, the first equation becomes  $\langle e_i (p) , \nn'(p) \rangle = - u_i$ so we get
the first claim \eqr{e:Nip}.

Similarly, using \eqr{e:Xip} gives that
  the $s$ derivative of $g_{ij}$ at $s=0$ is
\begin{equation}    \label{e:gij}
   \frac{\partial g_{ij} }{\partial s} (p,0) =\langle e_i (p)  ,     \frac{\partial F_j}{\partial s} (p,0) \rangle
   + \langle    \frac{\partial F_i}{\partial s} (p,0) , e_j (p)  \rangle
    =  -   2 \, u(p) \, a_{ij} (p)  \, .
\end{equation}
Since $g_{ij} (p,0) = \delta_{ij}$, it follows   from \eqr{e:gij} that the derivative of the inverse metric is  \begin{equation}    \label{e:gijin}
   \frac{\partial g^{ij} }{\partial s} (p,0) =     2 \, u(p) \, a_{ij} (p)   \, .
\end{equation}

{\bf{The Taylor expansion of the second fundamental form}}:
 We will compute the first order Taylor expansion (in $s$) of $a_{ij}(p,s)$.
 We will use $Q$ below to denote terms that are at least quadratic in $s$; $Q$ may mean different things even in the same line.  We will use $F_i' (p)$ to denote $\frac{\partial F_i}{\partial s} (p,0)$ (whose value was recorded in \eqr{e:Xip}).

  Using \eqr{e:Nip}, we can expand $\nn (p,s)$ as
  \begin{equation}	\label{e:nps}
  	\nn (p,s) = e_{n+1} - s \, \nabla u (p) + Q \, .
  \end{equation}
Using this, we expand $  \nabla_{F_i} \nn $ to get
\begin{equation}	\label{e:a21}
     -\nabla_{F_i} \nn (p,s) = - \nabla_{e_i} e_{n+1} + s \, \nabla_{e_i}  \, (u_k (p) \, e_k)  - s \, \nabla_{F_i'} e_{n+1} + Q \, ,
\end{equation}
where each covariant derivative on the right is performed at   the point $F(p,s)$.  We will expand each of the three terms on the right in \eqr{e:a21} separately.
Using Lemma \ref{l:chri} and the fact that the frame is parallel in the $e_{n+1}$ direction, the first term is
\begin{equation}	\label{e:t1}
	- \nabla_{e_i} e_{n+1} \, (F(p,s))
	=  a_{ik} (p) \, e_k + s \, u  \,  a_{in}   \, a_{kn}   \, e_k (p) + Q \, .
\end{equation}
Since $s \, \nabla_{e_i} e_k \, (F(p,s)) = s\, \nabla_{e_i} e_k (p) +Q$,
the second term is
 \begin{equation}	\label{e:t2}
	s \, \nabla_{e_i}  \, (u_k (p) \, e_k) \, (F(p,s)) = s \, \nabla_{e_i} \nabla u (p) + Q
	 = s \, S(e_i , e_{\alpha}) \, e_{\alpha} (p)   + Q \, ,
\end{equation}
where  $S(v,w) = \langle \nabla_v \nabla u , w \rangle $ is the ($\RR^{n+1}$) hessian of $u$.  Similarly, using
Lemma \ref{l:chri} and the formula \eqr{e:Xip} for $F_i'$,  the
  third term is
\begin{align}	\label{e:t3}
- s \, \nabla_{F_i'} e_{n+1}  \, (F(p,s))  &=     -s \, u_i \, \nabla_{e_{n+1}}  e_{n+1} (p)
  + s \, u    a_{in}   \, \nabla_{e_n} e_{n+1} (p) + Q   \notag \\
  &=
	- s \, u   \,  a_{in}   \,  a_{kn}   \, e_k (p)+ Q \, ,
\end{align}
where the last equality used \eqr{e:diffnn}.
Since the $a_{in}  \, a_{kn} $ terms cancel in \eqr{e:t1} and \eqr{e:t3},
substituting these three into \eqr{e:a21} gives
\begin{equation}	\label{e:a22}
     -\nabla_{F_i} \nn (p,s) =  a_{ik} (p) \, e_k  + s \, S(e_i , e_{\alpha}) \, e_{\alpha} (p)   +  Q \, ,
\end{equation}
By \eqr{e:aijps}, taking the inner product of this with
\begin{equation}
	F_j (F(p,s)) = e_j (p) +  s \,
	 u_j   \, e_{n+1} (p) -  s \, u  \, a_{jm}   \, e_m (p)
\end{equation}
gives (all terms on the right are evaluated at $p$)
\begin{align}		\label{e:big3}
    a_{ij} (p,s) &=  \langle a_{ik} \, e_k  , e_j \rangle
    +  s \,  \langle  a_{ik} \, e_k ,   \, \left( u_j   \, e_{n+1}   -   u  \, a_{jm}   \, e_m   \right)    \rangle
   + s \,  \langle S(e_i , e_{\alpha}) \, e_{\alpha}    , e_j \rangle
      + Q  \notag \\
      &=  a_{ij} - s \, u \,  a_{ik}  a_{jk}   + s \, S_{\Sigma} (e_i , e_j) + Q
      \,  .
\end{align}
 Here, in the last equality, we used that the Euclidean hessian $S(e_i , e_j)$ agrees with the submanifold hessian
$S_{\Sigma} (e_i , e_j)$ at $p$ since
  the normal derivative of $u$ is zero.

  {\bf{Putting it all together to compute $H$}}:
  Since $H(p,s) = - g^{ij} (p,s) \, a_{ij} (p,s)$, we can use \eqr{e:gijin} and \eqr{e:big3} to expand $H$
\begin{align}		\label{e:big4a}
    H(p,s) &=  - \left( \delta_{ij} + 2 s\, u \,  a_{ij} \right) \, \left(  a_{ij} -  s \, u \,  a_{ik}  a_{jk}
      + s \, S_{\Sigma} (e_i , e_j)
    \right)  + Q  \notag  \\
    &=  - a_{ii} - s \, \left( \Delta  u + |A|^2 \, u \right) + Q
      \, ,
\end{align}
where all quantities on the last line are evaluated at the point $p \in \Sigma$, $\Delta$ is the Laplacian on $\Sigma$, and $|A|^2 (p) = \sum_{i,j \leq n} a_{ij}^2 (p)$.  This gives \eqr{e:big4}.
\end{proof}

\end{document}